\documentclass[11pt]{article}
\usepackage[utf8]{inputenc}
\usepackage[T1]{fontenc}
\usepackage[english]{babel}

\usepackage{amsmath}
\usepackage{amssymb}
\usepackage{amsthm}
\usepackage{enumitem}
\usepackage{appendix}
\usepackage{geometry}
\usepackage{eurosym}
\usepackage{hyperref}
\hypersetup{
    colorlinks,
    citecolor=blue,
    filecolor=blue,
    linkcolor=blue,
    urlcolor=blue
}
\usepackage{todonotes}
\usepackage{mathrsfs}
 
 \usepackage{rotating}
 
\def\be{\begin{eqnarray}}
\def\ee{\end{eqnarray}}

\def\b*{\begin{eqnarray*}}
\def\e*{\end{eqnarray*}}

\newtheorem{Proposition}{Proposition}[part]

\newtheorem{Lemma}{Lemma}[part]
\newtheorem{Corollary}{Corollary}[part]
\newtheorem{Remark}{Remark}[part]

\newcommand{\ba}{\begin{array}}
\newcommand{\ea}{\end{array}}
\newcommand{\ben}{\begin{equation*}} 
\newcommand{\een}{\end{equation*}}
\newcommand{\bea}{\begin{eqnarray}}
 \newcommand{\eea}{\end{eqnarray}}
\newcommand{\bean}{\begin{eqnarray*}} 
\newcommand{\eean}{\end{eqnarray*}}
\newcommand{\bel}{\begin{align}} 
\newcommand{\eel}{\end{align}}
\newcommand{\beln}{\begin{align*}} 
\newcommand{\eeln}{\end{align*}}
\newcommand{\bit}{\begin{itemize}}
\newcommand{\eit}{\end{itemize}}

\makeatletter \@addtoreset{equation}{section}

\@addtoreset{Definition}{section}

\@addtoreset{Theorem}{section}

\@addtoreset{Proposition}{section}

\@addtoreset{Property}{section}

\@addtoreset{Assumption}{section}

\@addtoreset{Corollary}{section}

\@addtoreset{Lemma}{section}

\@addtoreset{Remark}{section}

\@addtoreset{Example}{section}

\@addtoreset{Condition}{section}

\def \E{\mathbb{E}}
\def \F{\mathbb{F}}
\def \H{\mathbb{H}}
\def \L{\mathbb{L}}
\def \M{\mathbb{M}}
\def \N{\mathbb{N}}
\def \P{\mathbb{P}}
\def \Q{\mathbb{Q}}
\def \R{\mathbb{R}}

\def \Z{\mathbb{Z}}

\def \G{\mathbb{G}}

\def\Cc{{\cal C}}

\def\Fc{{\cal F}}
\def\Gc{{\cal G}}

\def\Kc{{\cal K}}

\def\Pc{{\cal P}}

\def\Vc{{\cal V}}

\def\no{\noindent}

\def\={\;=\;}
\def\.{\;.}

\def \i{1\!\mbox{\rm I}}
\def\1{{\bf 1}}
\def \ep{\hbox{ }\hfill$\Box$}

\def \proof{{\noindent \bf Proof. }}

 \def\normeL2#1{\left\|{#1}\right\|_{L^2}}
 
\newcommand{\alias}[2]{
\providecommand{#1}{}
\renewcommand{#1}{#2}
}

\alias{\P}{\mathbb{P}}
\alias{\N}{\mathcal{N}}
\alias{\L}{\mathcal{L}}
\alias{\Z}{\mathbb{Z}}
\alias{\Q}{\mathbb{Q}}
\alias{\R}{\mathbb{R}}
\alias{\C}{\mathcal{C}}
\alias{\T}{\mathbb{T}}
\alias{\E}{\mathbb{E}}
\alias{\H}{\mathcal{H}}
\alias{\1}{\mathbf{1}}
\alias{\B}{\mathcal{B}}
\alias{\M}{\mathcal{M}}
\alias{\G}{\mathcal{G}}
\alias{\Y}{Y_{\bullet}}

\newcommand{\nc}{\newcommand}
\nc{\cA}{{\mathcal A}} \nc{\cB}{{\mathcal B}} \nc{\cC}{{\mathcal
C}} \nc{\cD}{{\mathcal D}} \nc{\bbD}{\mathbb{D}}
\nc{\cG}{{\mathcal G}} \nc{\cF}{{\mathcal F}} \nc{\cS}{{\mathcal
S}} \nc{\cU}{{\mathcal U}} \nc{\cH}{{\mathcal H}}
\nc{\cK}{{\mathcal K}} \nc{\cM}{{\mathcal M}} \nc{\cO}{{\mathcal
O}} \nc{\cP}{{\mathcal P}} \nc{\bbE}{\mathbb{E}}
\nc{\bbEP}{\mathbb{E}_{\mathbb{P}}}\nc{\bbL}{\mathbb{L}}
\nc{\bbP}{\mathbb{P}} \nc{\bbQ}{\mathbb{Q}} \nc{\del}{\partial}
\nc{\Om}{\Omega} \nc{\om}{\omega} \nc{\bbR}{\mathbb{R}}
\nc{\bbC}{\mathbb{C}} \nc{\bfr}{\begin{flushright}}
\nc{\efr}{\end{flushright}} \nc{\dXt}{\Delta X_{t}}
\nc{\dXs}{\Delta X_{s}} \nc{\bs}{\blacksquare} \nc{\dX}{\Delta X}
\nc{\dY}{\Delta Y}
\nc{\dnkx}{\left(X(T^{n}_{k})-X(T^{n}_{k-1})\right)}
\nc{\esssup}{\mathrm{ess}\mbox{ }\mathrm{sup}}
\nc{\essinf}{\mathrm{ess}\mbox{ } \mathrm{inf}}
\nc{\dhats}{\widehat{\delta_s}}

\nc{\chf}{\mbox{$\mathbf1$}}
\nc{\ind}{\mathds{1}}

\begingroup\expandafter\expandafter\expandafter\endgroup
\expandafter\ifx\csname pdfsuppresswarningpagegroup\endcsname\relax
\else
\fi

\nc{\drm}{\mathrm{d}}

\nc{\zfb}{ z_{\mbox{\footnotesize\rm\sc fb}} }
\nc{\zsb}{ z_{\mbox{\footnotesize\rm\sc sb}} }
\nc{\zsbm}{ z_{\mbox{\footnotesize\rm\sc sb$_m$}} }

\nc{\zsbdot}{ z'_{\mbox{\footnotesize\rm\sc sb}} }

\nc{\gfb}{ \gamma_{\mbox{\footnotesize\rm\sc fb}} }
\nc{\gsb}{ \gamma_{\mbox{\footnotesize\rm\sc sb}} }
\nc{\afb}{ a_{\mbox{\footnotesize\rm\sc fb}} }
\nc{\asb}{ a_{\mbox{\footnotesize\rm\sc sb}} }
\nc{\bfb}{ b_{\mbox{\footnotesize\rm\sc fb}} }
\nc{\bsb}{ b_{\mbox{\footnotesize\rm\sc sb}} }

\nc{\xisb}{ \xi_{\mbox{\footnotesize\rm\sc sb}} }
\nc{\xifb}{ \xi_{\mbox{\footnotesize\rm\sc fb}} }
\nc{\xisbm}{ \xi_{\mbox{\footnotesize\rm\sc sb$_m$}} }
\nc{\xisbf}{ \xi_{\mbox{\footnotesize\rm\sc sb}}^{\mbox{\footnotesize\rm\sc f}} }
\nc{\xifbf}{ \xi_{\mbox{\footnotesize\rm\sc fb}}^{\mbox{\footnotesize\rm\sc f}} }
\nc{\xisbmf}{ \xi_{\mbox{\footnotesize\rm\sc sb$_m$}}^{\mbox{\footnotesize\rm\sc f}} }
\nc{\xisbv}{ \xi_{\mbox{\footnotesize\rm\sc sb}}^{\mbox{\footnotesize\rm\sc v}} }
\nc{\xifbv}{ \xi_{\mbox{\footnotesize\rm\sc fb}}^{\mbox{\footnotesize\rm\sc v}} }
\nc{\xisbmv}{ \xi_{\mbox{\footnotesize\rm\sc sb$_m$}}^{\mbox{\footnotesize\rm\sc v}} }

\nc{\piesb}{ \pi_{\mbox{\footnotesize\rm\sc sb}}^{\mbox{\footnotesize\rm\sc e}} }
\nc{\piefb}{ \pi_{\mbox{\footnotesize\rm\sc fb}}^{\mbox{\footnotesize\rm\sc e}} }
\nc{\piesbm}{ \pi_{\mbox{\footnotesize\rm\sc sb$_m$}}^{\mbox{\footnotesize\rm\sc e}} }

\nc{\pivsb}{ \pi_{\mbox{\footnotesize\rm\sc sb}}^{\mbox{\footnotesize\rm\sc v}} }
\nc{\pivfb}{ \pi_{\mbox{\footnotesize\rm\sc fb}}^{\mbox{\footnotesize\rm\sc v}} }
\nc{\pivsbm}{ \pi_{\mbox{\footnotesize\rm\sc sb$_m$}}^{\mbox{\footnotesize\rm\sc v}} }

\nc{\msb}{ m_{\mbox{\footnotesize\rm\sc sb}} }
\nc{\mfb}{ m_{\mbox{\footnotesize\rm\sc fb}} }
\nc{\msbm}{ m_{\mbox{\footnotesize\rm\sc sb$_m$}} }

\nc{\sigmasb}{ \sigma_{\mbox{\footnotesize\rm\sc sb}} }
\nc{\sigmafb}{ \sigma_{\mbox{\footnotesize\rm\sc fb}} }
\nc{\Vsb}{ V^{\mbox{\footnotesize\rm\sc sb}} }
\nc{\Vsbm}{ V^{\mbox{\footnotesize\rm\sc sb$_m$}} }
\nc{\Vfb}{ V^{\mbox{\footnotesize\rm\sc fb}} }
\nc{\Gsb}{ \G^{\mbox{\footnotesize\rm\sc sb}} }
\nc{\nusb}{ \nu^{\mbox{\footnotesize\rm\sc sb}} }
\nc{\nufb}{ \nu^{\mbox{\footnotesize\rm\sc fb}} }
\nc{\Hv}{ H_{\rm v} }
\nc{\Hm}{ H_{\rm m} }

\begin{document}

\title{Optimal electricity demand response contracting \\with responsiveness incentives}

\author{René Aïd\thanks{Universit\'e Paris--Dauphine, PSL Research University, LeDA. The author received the support of the 
Finance for Energy Markets Research Centre and of the ANR CAESARS ANR-15-CE05-0024-02. 
The author would like to thank the Isaac Newton Institute for Mathematical Sciences, Cambridge, for support and hospitality during the programme {\em Mathematics for Energy Systems} where work on this paper was undertaken. This work was supported by EPSRC grant no EP/R014604/1.\&\#34 .This work benefited from comments at the SIAM Conference on Financial Mathematics \& Engineering, Austin (2016), at the 
10$^{\text{th}}$ Bachelier World Congress, Dublin (2018). A vulgarisation of the model presented in this paper was published in 
{\em SIAM News}, June 2017. The author thanks for their comments Hung--Po Chao, Ivar Ekeland, Sean Meyn  and Stéphane Villeneuve. rene.aid@dauphine.psl.eu.  }
\and Dylan Possama\"i\thanks{Columbia University, IEOR, 500W 120th St., 10027 New York, NY, dp2917@columbia.edu. 
Research supported by the ANR project PACMAN ANR-16-CE05-0027.}
\and Nizar Touzi\thanks{\'Ecole Polytechnique, CMAP, nizar.touzi@polytechnique.edu}
}

\date{\today}

\maketitle
\begin{abstract}
Despite the success of demand response programs in retail electricity markets in reducing average consumption, the random responsiveness of consumers to price event makes their efficiency questionable to achieve the flexibility needed for electric systems with a large share of renewable energy. The variance of consumers' responses depreciates the value of these mechanisms and makes them weakly reliable. This paper aims at designing demand response contracts which allow to act on both the average consumption and its variance. 

The interaction between a risk--averse producer and a risk--averse consumer is modelled through a Principal--Agent problem, thus accounting for the moral hazard underlying demand response contracts. The producer, facing the limited flexibility of production, pays an appropriate incentive compensation to encourage the consumer to reduce his average consumption and to enhance his responsiveness. We provide closed--form solution for the optimal contract in the case of constant marginal costs of energy and volatility for the producer and constant marginal value of energy for the consumer. We show that the optimal contract has a rebate form where the initial condition of the consumption serves as a baseline. Further, the consumer cannot manipulate the baseline at his own advantage. The first--best price for energy is a convex combination of the marginal cost and the marginal value of energy where the weights are given by the risk--aversion ratios, and the first--best price for volatility is the risk--aversion ratio times the marginal cost of volatility. The second--best price for energy and volatility are non--constant and non--increasing in time. The price for energy is lower (resp. higher) than the marginal cost of energy during peak--load (resp. off--peak) periods.

We illustrate the potential benefit issued from the implementation of an incentive mechanism on the responsiveness of the consumer by calibrating our model with publicly available data. We predict a significant increase of responsiveness under our optimal contract and a significant increase of the producer satisfaction. 
\end{abstract}

\clearpage
\section{Introduction}

This paper studies the use of demand response contracts to achieve flexible power systems. In its common form, a demand response mechanism is a contract under which the consumer benefits from cheaper electricity than the standard tariff all the year except at certain peak load periods chosen by the producer where the price is much higher. These soft mechanisms appear to cumulate the virtues of consumption reduction while providing substitutes to hardware technologies as chemical batteries or flexible gas--fired plants. There exist many forms of demand response contracts. For domestic customers, the latter form is the most common. An alternative form consists in giving a cash premium at the beginning of the year and then substracting the value of the energy consumed during price events (see Jessoe and Rapson (2014) \cite{Jessoe14} for an example). For industrial customers, the payment during price events is indexed on the difference between the energy consumed compared to the energy they would have consumed if they had not received any signal, i.e. the baseline consumption. This last form is referred to as Peak--Time Rebate (PTR). It is known that demand response contracts indexed on the consumer's baseline are prone to manipulation (see Chao and De Pillis (2013) \cite{Chao13}  in the context  of the Baltimore stadium management by Enerwise company case\footnote{Enerwise was fined a  \$780.000 penalty  by the Federal Energy Regulation Commission 143 FERC 61218 as of June 7th, 2013 for manipulation of a demand response program.} and Hogan (2009,2010) \cite{Hogan09,Hogan10}, Crampes and Léautier (2015) \cite{Crampes15}, Brown and Sappington (2016) \cite{Brown16}). 

Moral hazard is a central issue in demand response contracts. Because the consumption during price event is random, it is not possible to know if the observed consumption has been reduced compared to what the consumer would have consumed if no price signal had been sent, the latter being non--observable. The success of a demand response mechanism depends on whether consumers do react on price events or not. But moral hazard makes it difficult to quantify the responsiveness of a given consumer or group of consumers, because this quantity is non--observable. This problem has been clearly identified in experiments but not necessarily connected to a moral hazard factor. For instance, in the context of the large--scale demand response experiment of Low Carbon London project, Schofield et. al. (2016) \cite{Schofield16} defines responsiveness by {\em a deliberate act made by the consumer to reduce her consumption}, a definition that brings out that responsiveness is the result of the non--observable actions of the consumers. 

The moral hazard problem translates in the variance of responses observed in demand response experiments. Consumers do not react to price signal with the discipline of a gas--fired plant. Upon receiving a price signal, a household may give up the commitment to reduce consumption even if the price is high to satisfy the constraints of the everyday life (see Torriti (2017) \cite{Torriti17} for sociological analysis of households’ energy consumption behaviour). Faruqui and Sergici (2010)'s survey \cite{Faruqui10} gives an idea of the {\em ex--ante} uncertainty of a demand response program. A sample of 15 dynamic demand response trials shows an average reduction that ranges between 10\% to 50\%.  Further, again in the Low Carbon London pricing trial, Carmichael et al. (2014) \cite[Section 4.3, Figure~4.9]{Carmichael2014} reports consumer's reaction to high price event that ranges from --200~Watt to +200~Watt for a consumption of order of magnitude of 1000~Watt. Mathieu (2011) \cite[Table~6]{Mathieu11} reports that a furniture store with a peak demand of 1,300~kW reduced on average its consumption by 78~kW with a standard deviation of 30~kW. The moral hazard problem makes demand response contract poorly efficient and thus, less valuable than a deterministic substitute.

\hspace{5mm}

In this paper, we address the efficiency and design problems of demand response contracts using optimal contract theory with moral hazard. We show that under certain conditions to be described below, the optimal contract has a rebate form. Although this form of demand response contract is intensively used in the economic literature on power systems, it was not shown to our knowledge that it was indeed the optimal form. The economic literature has focused on the adverse selection problem posed by demand response, i.e. the fact that consumers may report strategic reduction capacity (see Crampes and L\'eautier (2015) and also Fahrioglu and Alvarado (2000) \cite{Fahrioglu00}) or on the analysis of the first--best case (see Hogan (2010, 2011) \cite{Hogan09, Hogan10}, Chao (2011) \cite{Chao11} and Brown and Sappington (2016) \cite{Brown16}).  Besides, the optimal contract is baseline--proofness, i.e. not susceptible of artificial increase in the benefit of the consumer in the sense that whatever the initial condition of the consumption, the consumer gets no more than the reservation utility he asked for. Furthermore, we show that it is possible and valuable for risk--averse producers to provide incentives to the consumer to increase the response to price events, i.e. to provide regular responses across price events or to increase the predictability of the consumption. This result is the implementation of the simple fact that uncertainty in payments deteriorates the value of a contract for a risk--averse agent. But, beyond this trivial remark, the problem is how to design a contract that achieves a reduction of the consumption with a small standard deviation when responsiveness is non--observable? Even though responsiveness is not observable, we show that indexing the contract on the quadratic variation of the consumption, which is an observable quantity, provides the desired result of making the consumption reduction more reliable.

\hspace{5mm}

We formulate the mechanism design of demand response programs as a problem optimal contracting between a producer and a consumer. We focus on one single demand response event with fixed duration. A risk--averse producer has the obligation to serve the random electricity consumption of a risk--averse consumer. The producer has a constant marginal cost of energy generation $\theta$ and is also subject to a constant marginal cost of the consumption volatility $h$. This direct cost of volatility aims to consider the impact of intermittency on generation costs. The volatility of the consumption is observable and given by the quadratic variation of the consumption. The consumer has also a constant marginal value of energy $\kappa$. He can reduce the average level of his consumption during the demand response event by curtailing some activities (washing machine, air conditioning, television and so on). Besides, the consumer can also decrease the uncertainty on the reduction by showing some discipline in his routine. The costs of actions on the average reduction and on the regularity of the reduction depend on the corresponding usages of electricity (lights, oven, air conditioning, television, computer...). A peak--period (resp. off--peak) is a case when the marginal cost of energy generation is higher (resp. lower) than the marginal value for the consumer $\kappa \leq \theta$ (resp. $\theta \leq \kappa$). The producer observes the total consumption trajectory and its quadratic variation but has no access to the consumer's actions, efforts or consumption per usages. She aims at finding the optimal contract that minimises the expected disutility from the energy generation cost, the direct cost of volatility and the incentive payment, while anticipating the optimal response of the consumer's maximisation of his expected utility from the payment, the benefit value of his consumption and the cost of efforts. Finally, the producer's problem is subject to the consumer's participation constraint. The admissible contracts are indexed only on the observed level of consumption and its volatility, and not on any assumed counterfactual baseline consumption.

\hspace{5mm}

We solve the first--best and second--best optimal contracting problems in closed form. We prove that in both cases, the optimal contract takes the form of a Peak--Time Rebate contract where the initial level of the consumption serves as a baseline. The consumer is charged the energy he would have consumed if no efforts were made and then, he is charged or he is paid at a deterministic price the difference between his realised consumption and its initial value. In the limiting case where the consumer is risk--neutral, the second--best contract achieves the first--best optimum, as it is usual in optimal contract theory with moral hazard. In this case, the consumer is charged for the energy at the marginal cost of energy and for the volatility at the direct marginal cost of volatility. Furthermore, the price of volatility is simply a share of the direct cost of volatility, where the share is the relative risk aversion of the producer compared to the consumer.

Apart from this limiting case, the risk--sharing process leads to charging the consumer at energy and volatility prices different from their marginal costs. In the first--best situation, the price of energy is a convex combination of the energy marginal value for the consumer and the marginal cost of the producer. The weights are given by the risk--ratios. During peak--periods, the price paid is higher than the marginal value of energy for the consumer, providing thus an incentive to reduce his consumption. During off--periods, the price paid is lower than this marginal value and thus, provides no incentives to reduce consumption. 

During peak periods, the producer induces efforts to reduce the average consumption and to increase responsiveness. The incentive prices are no longer equal to the marginal cost of energy and volatility. They depend on the costs of effort of the consumer. The second--best price for energy is non--constant and non--linear even in the case of constant marginal cost and value of energy. It is a decreasing function of time that induces more effort in the beginning of the period than at the end. Still in peak time periods, the second--best contract fails to achieve the first--best optimum. This is because the second--best contract fails to achieve an optimal coordination of the reduction of consumption and the reduction of volatility. Inducing a reduction of the volatility influences the payment for the reduction of consumption.  The first--best contracting does not suffer from this constraint. The first--best contracting can separate the induced observed efforts and the prices for reduction and thus can achieve a better coordination.  The consumer enjoys an information rent, which is an increasing function of the volatility of his consumption. During off--peak periods, the producer induces effort only on responsiveness and achieves first--best value by the implementation of second--best optimal contracting. First--best and second--best prices are equal while prices for volatility differ.

Our result provides rationality for considering only rebate contracts when assessing the appropriate introduction of demand response mechanisms in a market. But, it also appears that, from a theoretical point of view, the design of efficient incentive policies to foster consumer's demand response cannot solely rely on the marginal cost of electricity of the producer ---  as proposed by the FERC in Order n.~719  ---, or only on the difference between this marginal cost and the value paid by the consumer. These incentives are proxies for the optimal contract. But the second--best optimal contract should also consider the disutility incurred by the consumer for his efforts to induce the reduction, as well as the risk the consumer takes. We stress that this result applies even when there are no responsiveness incentives implemented. In this case, the closed--form expression of the energy price embeds the costs of effort of the consumer.

We also find that, under the optimal contract, the resulting consumption volatility may decrease as required, but it may also increase depending on the risk aversion parameters of both actors. Because he is risk--averse, the consumer has a natural incentive to reduce the consumption volatility, even without contract. Thus, before contracting one could observe a lower level of consumption volatility than the no--effort situation. After contracting, if the producer's volatility cost is small enough, or if she is much less risk--averse than the consumer, she would accept the volatility risk and thus, allow the consumer to avoid making costly effort to reduce volatility.   This result illustrates how the risk sharing process between the producer and the consumer allows the electric system to bear more risk.

\hspace{5mm}

Only experiments could quantify the benefits of responsiveness incentives claimed in this paper. Nevertheless, we illustrate the potential benefit from such a mechanism design by using the Low Carbon London data to provide quantitative estimates. We calibrate the parameters of our model by interpreting this experiment as the implementation of our optimal contracting model under no responsiveness incentives, i.e. the consumer is only incentivised to reduce the average level of his consumption. We find an important consumption volatility reduction and a significant increase of the producers benefit for a small increase of the cost of efforts from the consumer.  We examine the robustness of this result to the assumption of constant marginal costs and value of energy. We implement a numerical approximation of the solution of the optimal contracting problem for a concave consumer's energy value function, and we perform the comparison with the corresponding linear contract approximation. We find that the concavity of the energy value function of the consumer has a downward effect on the benefit of the producer. The linear approximation of the energy value function exacerbates this downward effect for strongly concave functions. Further, because it may not be necessarily easy to explain to domestic consumers an incentive mechanism based on the quadratic variation of their consumption, we propose simplified alternatives for the design of new experiments.

\hspace{5mm}

Our model stands at the intersection of optimal contract theory in continuous-time, the economics of demand response for power systems and their empirical implications. From a methodological point of view, our work falls in the line of the works of H\"olmstrom and Milgrom (1979) \cite{holmstrom1979moral}, Grossman and Hart (1983) \cite{grossman1983analysis}, H\"olmstrom and Milgrom (1987) \cite{Holmstrom87} and Sannikov (2008) \cite{Sannikov08}. Our results rely on the recent advances of Cvitani\'c et al. (2018) \cite{cvitanic2018dynamic} which allows for volatility control (i.e. responsiveness effort) in the continuous time Principal Agent problem.  For an economic introduction to incentives theory, we refer to Laffont and Martimort (2002) \cite{Laffont02}. The economic literature on optimal contract theory and risk-sharing is vast. We also refer to Cadellinas et al. (2007) \cite{Cadenillas07}, who present a general setting for the analysis of the first--best optimal risk--sharing in the context of compensation plan for executives, and to M\"uller (1998) \cite{Muller98} who provides the first--best optimal sharing rule in the case of CARA exponential utility function and shows that it is also a linear function of the aggregated output as in the second--best case. The closest work on volatility control in the framework of Principal--Agent is in Cvitani\'c et al. (2017) \cite{cvitanic2017moral}, however our paper is the first to produce a closed form solution in this context.

Regarding demand response programs, their study is a long--dated subject in the economic literature, see Tan and Varaiya (1993) \cite{Tan93} for a seminal model of interruptible contracts for a pool of consumers. The framework of optimal contract theory has been used to formulate the enrolment of customers in demand response programs as an adverse selection program. This idea can be traced back to the work of Fahrioglu and Alvarado (2000) \cite{Fahrioglu00} on incentive compatible demand response program and has been recently reused both in the work of Crampes and L\'eautier (2015) \cite{Crampes15} to design suitable base--line consumption reference, and by Alasseur et al. (2017) \cite{Alasseur17} for peak--load pricing. More recently, big data methods have been used to provide precise response of consumers with a given probability, see Kwac and Rajagopal (2014) \cite{Kwac16}. The only known work to the authors modeling the cost of electricity demand volatility is Tsitsiklis and Xu (2015) \cite{Tsitsiklis15} who designed a model in discrete--time where they consider not only the generation cost of energy but also the cost of variation of generation between two--time steps. In their model, consumers are incited to reduce their consumption with a price signal which is the traditional marginal fuel cost of generation plus the marginal of cost of variation of generation. The complexity of the model in terms of represented constraints only allows for numerical simulations. Further, our work contributes to the empirical literature on moral hazard models, see Chiappori and Salani\'e (2000) \cite{Chiappori00} for insurance industry, Lewis and Bajari (20014) \cite{Lewis14} for public procurement and  Bandiera et al. (2005) \cite{Bandiera05} for workers’ productivity. We make the testable prediction that the implementation of responsiveness incentives decreases consumption volatility. This claim can be tested by the next generation of demand response pricing trials by performing controlled experiments to compare consumers with and without responsiveness incentives as described in our work.

The paper is organised as follows. Section~\ref{sec:model} describes the model. Section~\ref{sec:results} provides the first--best, the second--best optimal contract and the optimal contract when no incentives on responsiveness is provided. Section~\ref{sec:numeric} provides empirical results. Section~\ref{ssec:practice} gives practical ways to implement a responsiveness incentive scheme. Section~\ref{sec:conclusion} concludes.

\section{The model}
\label{sec:model}

Before formulating our model in precise mathematical terms, we describe its main features and explain its underlying assumptions. We consider a single producer who has obligation--to--serve the electricity demand of a single residential consumer enjoying a flat retail rate. This situation captures the main regulatory barrier to price--responsive demand (see Chao~(2011) \cite{Chao11}).  The producer wonders if it is interesting or not to propose to the consumer a demand response contract to incite him to change his consumption during a given period $[0,T]$. It might be more efficient during peak period to pay the consumer to reduce his consumption rather than deploy costly generation power plants. Besides potential energy consumption reduction, the producer wonders if it is interesting to incite the consumer to provide a more predictable consumption during the price event. A highly random consumption induces adjustment costs that could be avoided if the consumer exhibits a more regular consumption pattern. Further, a highly random consumption reduction from one price event to another reduces the value of the contract as it increases the uncertainty of the payments to or from the consumer. Thus, the concern of the producer is to propose a contract of demand response that would allow her to avoid high generation cost and make the consumer more predictable. Nevertheless, the producer is facing a moral--hazard problem. Once the demand response contract is signed, when the price event occurs, he may change his mind and prefer to consume rather than make the costly efforts of reduction. Because the producer only observes the consumption and not the actions of the consumer in his house, it is not possible to infer if a consumption level is the result of efforts or just chance.

We focus on one single demand response event with fixed duration, making the hypothesis that all successive demand response events exhibit the same characteristics. This is genuinely the case in electricity industry, all demand response events share the same price structure. The timing of the operation is the following. At time zero, the producer proposes a paying rule for demand response, knowing the reservation utility of the consumer; at time one, the consumer accepts or rejects the contract; if the consumer accepts the contracts, then, later, the demand response event occurs, the producer measures in continuous--time the consumption and provides the payment (positive or negative) at the end of the event.

The consumption of the consumer is the sum of the consumption of different usages. The consumer values in monetary terms the electricity he can consume. Although it is more realistic to consider that the marginal value of each kWh is decreasing, we consider a constant marginal value of energy. The case of strictly concave energy value function is provided in the Appendix of the paper, but we focus on the constant marginal case where the results are explicit. The consumer can achieve a reduction of his consumption by taking costly actions. The actions are differentiated per usage. They consist in reducing or not the usage of each electric devices. These actions are costly in a sense that they may call for the usage of substitutes to achieve the same level of comfort. In times of demand response event during winter, households may use gas stoves or other heating devices. Instead of cooking the lunch and the dinner on the electric cooking range, they may use a camping gas stove. The consumer can also achieve regularity in his responses to the producer sollicitations by disciplining his consumption usages. Once the consumer has received the demand response signal, he decides the timing of his consumption and tries to stick to it. During the price event, random events occur that may drive away the consumer from his planned consumption pattern if no actions are taken to dampen their effects (kids coming back from school with friends, sudden drop of outside brightness...). This a way we can understand the control of the volatility of the consumption. The mitigation of the random events that occur during the price event would result in a lower volatility of the consumption which can be measured by its quadratic variation. In our setting, the responsiveness of the consumer is the unobservable actions he takes to reduce the impact of his daily life events on his consumption. In this way, we capture the fact that when the producer observes a reduction, it is not possible to determine surely if it is the result of voluntary actions of the consumer or if it is by chance.  Of course, the higher the discipline of the consumer, the higher the cost. Indeed, the consumer must make a compromise between the consumption pattern he wants to stick to and the random events of everyday life.  Reducing consumption and mitigating the effects of random events on a consumption pattern are two separate phenomena. A consumer could exhibit a high level of consumption but still be responsive if he sticks to a regular pattern of consumption. The cost function for consumption reduction and for responsiveness increase is supposed to be strictly convex. We specify a parametric form that captures convexity and separation of reduction and responsiveness cost that allows quantitative estimates. 

On her side, the producer bears two kind of costs: generation cost to meet consumption of the consumer on real--time basis and direct consumption volatility cost. The marginal generation cost is assumed constant. This assumption does not properly reflect the high volatility of electricity prices but it captures the Time--Of--Use tariffs which provide one price during peak period and another for off--peaks. The direct consumption volatility cost represents the costs induced by the non--predictable part of the consumption. There is quite a literature on the assessment of the direct and indirect costs induced by the massive introduction of intermittent renewable energy. For a review of this topic, we refer to Hirth (2015) \cite{Hirth15} and Strbac et. al. (2007) \cite{Strbac07}.  There is a consensus on the fact that the extra--generation cost is an increasing, probably strictly convex function of the volatility of the renewable energy sources. In our setting, we limit ourselves to a constant marginal cost of consumption volatility.

Furthermore, both the producer and the consumer are supposed to be risk--averse. The producer is risk--averse because the obligation to serve the consumer's demand is not associated with the possibility to transfer to the consumer all the risks involved with the uncertain generation costs. Besides, there is not necessarily a diversification effect in the aggregation of demand response across many consumers that would lead to compensation of errors. Indeed, large scale demand response event still exhibit a significant variance of consumption across price events. On their side, consumers are reluctant to engage in demand response contract because they fear unexpected high bill due to their uncertain capacity to reduce consumption during price event. For instance, in the case of the Low Carbon London experiment, the enrolled consumers were paid 100~\textsterling~at the beginning of the trial and 50~\textsterling~more if they completed all the trial and were assured not to lose any money because of the form of their contracts. We shall take for both actors an exponential utility function. This utility is known to exhibit independence with respect to initial wealth. Although there exists evidence that the risk--aversion of consumers decreases with wealth, the hypothesis of initial wealth independence is an acceptable approximation in the context of demand response contract where only a small fraction of the total income of the consumer is involved.

\subsection{The consumer}
\label{ssec:consumer}

The Agent consumption process is denoted by $X=\{X_t,t\in[0,T]\}$, and is the canonical process of the space $\Omega$ of scalar continuous trajectories $\omega:[0,T]\longrightarrow\R$, i.e. $X_t(\omega)=\omega(t)$ for all $(t,\omega)\in[0,T]\times\Omega$.  We denote by $\F=\{\Fc_t,t\in[0,T]\}$ the corresponding filtration. A control process for the consumer (the Agent) is a pair $\nu:=(\alpha,\beta)$ of $\F-$adapted processes, which are respectively $
A-$ and $B-$valued. More specifically, $\alpha$ represents the effort of the consumer to reduce the level of consumption and $\beta$ is the responsiveness effort consisting in reducing the consumption variability for each usage of electricity. We emphasise that $\alpha$ and $\beta$ are respectively $N-$ and 
$d-$dimensional vectors, thus capturing the differentiation between different usages, e.g. refrigerator, heating or air conditioning, 
lightning, television, washing machine, computers... This set of control processes is denoted by $\mathcal U$.

\vspace{0.5em}
For a given initial condition $X_0\in\R$, representing the consumption at the beginning of the price event, and some 
control process $\nu:=(\alpha,\beta)$, the controlled equation is defined by the following stochastic differential 
equation\footnote{For technical reasons, we need to consider weak solutions of the stochastic differential equations. 
However, for expositional purposes, we deliberately ignore this technical aspect in this section so as to focus on the main 
message of the present paper.} driven by a $d-
$dimensional Brownian motion $W$
 \begin{equation}
 \label{eq:demand1}
 X_t = X_0 - \int_0^t \alpha_s\cdot {\bf 1}\drm s + \int_0^t \sigma(\beta_s) \cdot \mathrm{d} W_s,~t\in[0,T],\; 
 \mbox{with}\; 
 \sigma(b):=(\sigma_1\sqrt{b_1},\ldots,\sigma_d\sqrt{b_d})^\top,
 \end{equation}
for some given parameters $\sigma_1,\ldots,\sigma_d>0$. All of our utility criteria depend only on the distribution $\P^\nu$ of the 
state process $X$ corresponding to the effort process $\nu$. Let $\Pc$ be the collection of all such measures $\P^\nu$. 

The state variable $X$ represents the consumer's consumption. An effort $\nu$ 
induces a separable cost $c(\nu):=c_1(\alpha)+\frac12 c_2(\beta)$ and a value $f(X)$ from the consumption $X$. Throughout this 
paper, we shall take
\[
 c_1(a) 
 :=
 \frac12 \sum_{i=1}^N\frac{a_i^2}{\mu_i},
 ~\mbox{and}~
 c_2(b)
 :=
 \sum_{j=1}^d \frac{\sigma_j^2}{\lambda_j}\big(b_j^{-1}-1\big),
 ~a\in A,~b\in B,
\]
for some $(\mu,\lambda)\in(0,+\infty)^N\times(0,+\infty)^d$. Notice that $c$ is convex, increasing in $a$ and 
decreasing in $b$ as the responsiveness effort consists in reducing the volatility, thus reproducing the requested effects of 
increasing marginal cost of effort. Moreover, $c_1(0)=c_2(1)=0$ captures the fact that there is no cost for making no effort. The 
cost function is quadratic in $a$, and illustrates that small deviations (as e.g. switching off the light when leaving some place) are 
painless, while large deviations (as e.g. reducing the consumption from heating or air conditioning) are more costly. The function $f$ is  taken to be linear $f(x) = \kappa \, x$. Results are provided in the Appendices \ref{app:cru},  \ref{app:fblinear} and  \ref{app:sblinear} for a general  increasing  value function $f$. 

The Agent controls the electricity consumption by choosing the effort process $\nu$ in the state equation 
\eqref{eq:demand1}, subject to the corresponding cost of effort rate $c(\nu)$, and the consumption value rate $f(X)$. For technical reasons, we need to consider bounded efforts, we then set
\[
 A := \big[0,\mu_1 A_{\rm max}\big]\times\cdots\times\big[0,\mu_N A_{\rm max}\big]
 ~
 \mbox{and}~
 B := [\varepsilon,1],
\]
for some constants $A_{\rm max}>0$ and $\varepsilon>0$.

The execution of the contract starts at $t=0$. The consumer receives the value $\xi$ from the producer at time $T$. The value $\xi$ can be positive (payment) or negative (charge). The producer has no access to the consumer's actions, and does not observe the 
consumer's different usages of electricity. She only observes the overall consumption $X$. Consequently, the 
compensation $\xi$ can only be contingent on $X$, that is $\xi$ is $\Fc_T-$measurable. We denote by $\mathcal C$ the set of $\Fc_T-$measurable random variables. The objective function of the consumer is then defined for all $(\nu,\xi)\in\mathcal U\times\mathcal C$ by
 \begin{equation}\label{JA}
 J_{\rm A}(\xi,\P^\nu)
 :=
 \mathbb E^{\mathbb \P^\nu}\bigg[ U_{\rm A} \bigg( \xi+\int_0^T \big( f(X_s) - c(\nu_s) \big) \drm s\bigg) \bigg],
 \mbox{ where }
 U_{\rm A}(x) := - \mathrm{e}^{-r x},
 \end{equation}
for some constant risk aversion parameter $r > 0$. It is implicitly understood that the limiting case where $r$ tends to zero corresponds to a  risk--neutral consumer. 

The consumer's problem is
 \begin{equation}\label{eq:AgentObj}
 V_{\rm A}(\xi)  
  := 
 \sup_{\P^\nu\in\Pc} J_{\rm A}(\xi,\P^\nu),
 \end{equation}
{\it i.e.} maximising utility from consumption subject to the cost of effort. A control $
\P^{\widehat \nu}\in\Pc$ will be called optimal if $V^{\rm A}(\xi)=J_{\rm A}\big(\xi,\P^{\widehat \nu}\big)$.
We denote by $\Pc^\star(\xi)$ the collection of all such optimal responses $\P^{\widehat \nu}$. We finally assume that the consumer has a reservation utility $R_0 \in \R_{-}$.  We denote by $L_0 := -\frac{1}{r} \log{(-R_0)}$, the certainty equivalent of the reservation utility of the consumer for the consumer.

\subsection{The producer}
\label{ssec:producer}

The producer (the Principal) provides electricity to the consumer, and thus faces the generation cost of the produced energy, and 
the cost induced by the variation of production. Her performance criterion is defined by 
 \begin{equation}\label{eq:JP}
 J_{\rm P}(\xi,\P^\nu) 
 := 
 \E^{\P^\nu}\bigg[U\bigg(- \xi - \int_0^T g(X_s) \mathrm{d}s -  \frac{h}{2} \langle X\rangle_T \bigg) \bigg],
 \; \mbox{with}\; 
 U(x):=-\mathrm{e}^{-px}.
 \end{equation}
Here $p>0$ is the constant absolute risk aversion parameter, $g$ is taken to be linear $g(x) = \theta x$. Results are provided in the Appendices for general  non--decreasing generation cost function. The parameter $h$ is a positive constant representing the {\em direct marginal cost} induced by the quadratic variation $\langle X\rangle$ of the consumption. The higher the volatility of the consumption, the more costly it is for the producer to follow the load curve. Note that due to risk--aversion, the producer bears also an {\em indirect cost} of volatility. 

An $\Fc_T-$measurable random variable $\xi$ will be called a {\it contract}, which we denote by $\xi\in\mathcal C$, if it 
satisfies the additional integrability property
 \begin{equation}\label{contract-growth}
 \underset{\P^\nu\in\Pc}{\sup}\ \E^{\P^\nu}\big[e^{-rm\xi}\big] 
 + \underset{\P^\nu\in\Pc}{\sup}\ \E^{\P^\nu}\big[e^{pm\xi}\big]<+\infty,
 \; \mbox{for some}\;
 m>1.
 \end{equation}
This integrability condition guarantees that the consumer's criterion \eqref{JA} and the principal one \eqref{eq:JP} are well-defined.

\vspace{0.5em}

Throughout this paper, we shall consider the two following standard contracting problems.
\begin{itemize}
\item {\it First best contracting} corresponds to the benchmark situation where the producer has full power to impose a contract to 
the consumer and to dictate the Agent's effort
 \begin{equation}\label{eq:FB}
\Vfb
 :=
 \sup_{(\xi,\P^\nu)\in\Cc\times\Pc}\big\{ J_{\rm P}(\xi,\P^\nu):J_{\rm A}(\xi,\P^\nu)\ge R_0 \big\},
 \end{equation}
\item {\it Second best contracting} allows the consumer to respond optimally to the producer's offer. We follow the standard 
convention in the Principal--Agent literature in the case of multiple optimal responses in $\Pc^\star(\xi)$, that the consumer 
implements the optimal response that is the best for the producer. This leads to the second best contracting problem
 \begin{equation}\label{eq:SB}
 \Vsb 
  := 
 \sup_{\xi \in \Xi }\ \sup_{\P^\nu\in\Pc^\star(\xi)} 
J_{\rm P}(\xi,\P^\nu),
\mbox{ where }
  \Xi
 :=
 \big\{ \xi\in\Cc:V_{\rm A}(\xi)\geq R_0\big\},
 \end{equation}
with the convention $\sup\emptyset=-\infty$, thus restricting the contracts that can be offered by the producer to those $
\xi\in\mathcal C$ such that $\Pc^\star(\xi)\neq\emptyset$. 
\end{itemize}

\subsection{Consumer's optimal response and reservation utility}
\label{ssec:response}

We collect here some calculations related to the consumer's optimal response which will be useful  throughout the paper. 
Following Cvitani\'c et al. (2018) \cite{cvitanic2018dynamic}, we introduce the consumer's Hamiltonian.
 \begin{equation} \label{eq:agentH}
 H(z,\gamma) 
 :=
 H_{\rm m}(z) + H_{\rm v}(\gamma), 
 ~z,\gamma\in\R,
 \end{equation}
where $H_{\rm m}$ and $H_{\rm v}$ are the components of the Hamiltonian corresponding to the instantaneous mean and 
volatility, respectively
 \begin{align}
 H_{\rm m}(z)  := 
 -\inf_{a\in A}\big\{a\cdot\1 z + c_1(a)\big\} ,
 \mbox{ and }
 H_{\rm v}(\gamma) := -\frac12\inf_{b\in B} \big\{c_2(b) -\gamma|\sigma(b)|^2\big\}. 
 \end{align}
Here, $z$ represents the payment rate for a decrease of the consumption and $\gamma$ represents the payment rate 
for a decrease of the volatility of the consumption. Both payment rates can be positive or negative. Given these payments, 
the consumer maximises the instantaneous rate of benefit given by the Hamiltonian to deduce the optimal response $\widehat a(z)
$ on the drift and $\widehat b(\gamma)$ on the volatilities. The following result collects the closed--form expression of the 
optimal responses. We denote $x^- := 0\vee(-x)$, $x\in\R$.

\begin{Proposition} {\rm (Consumer's best--response)} \label{prop:agentresponse}
The optimal response of the consumer to an instantaneous payment rate $(z,\gamma)$ is
\[
 \widehat a_j(z) :=\mu_j (z^-\wedge A_{\rm max}),
 \mbox{ and }
 \widehat b_j(\gamma):= (1\wedge(\lambda_j\gamma^-)^{-\frac{1}{2}}) \vee \varepsilon ,
 ~j=1,\ldots,N,
\]
so that, with $\bar\mu:=\mu\cdot{\bf 1}$, $\widehat\sigma(\gamma):=\sigma\big(\widehat b(\gamma)\big)$, $\widehat 
c_1(z):=
 c_1\big(\widehat a(z)\big)$, and $\widehat 
c_2(\gamma):=
 c_2\big(\widehat b(\gamma)\big)$,
 \[
 H_{\rm m}(z) = \frac12\bar\mu\big(z^-\wedge A_{\rm max}\big)^2
 \mbox{ and }
 H_{\rm v}(\gamma)=-\frac12\Big(\widehat c_2(\gamma)-\gamma\big|\widehat\sigma(\gamma)\big|^2\Big).
\]
\end{Proposition}

The $z$ payment induces an effort of the consumer on all usages to reduce the average consumption deviation and this effort is inversely
proportional to its cost $1/\mu_i$. The $\gamma$ payment induces an effort only on the usages whose cost $1/\lambda_j$ is lower than the payment. As a first use of the previous notations, we provide the following characterisation of the consumer's behaviour without contract. 

\begin{Proposition} {\rm (Consumer's behaviour without contract)} \label{prop:crulinear}
Let $f(x) = \kappa \, x,$ $x\in\R$, for some $\kappa \geq 0$. Then,  $V_A(0) = U_A(\ell_0)$ where $\ell_0 = \kappa X_0 T + E_0(T)$ and   
\[
 E_0(T) := \int_0^T H_{\rm v}\big(-\gamma(t)\big)  \drm t,
 \mbox{ and }
 \gamma(t) :=  -r \kappa^2 (T-t)^2.
 \]
The consumer's optimal effort on the drift and on each volatility usage are respectively 
 \[
 a^0=0, 
 \text{ and } 
 b_j^0(t) := \varepsilon \vee (1 \wedge \big( \lambda_j |\gamma(t)| \big)^{-\frac{1}{2}}),\; j=1,\dots,d,
 \]
thus inducing an optimal distribution $\P^0$ under which the deviation process follows the dynamics $\mathrm{d}X_t=\widehat{\sigma}
\big(b^0(t)\big)\cdot \mathrm{d}W_t,$ for some $\P^0-$Brownian motion $W$.
\end{Proposition}

\begin{proof} See Appendix \ref{app:cru}. \end{proof}

%

\section{Main results}
\label{sec:results}

We consider the case where 
 \begin{equation}\label{linear:energy}
 (f-g)(x) := \delta x,\; x\in\R.
 \end{equation}
for some constant parameter $\delta := \kappa - \theta$, called hereafter {\em energy value discrepancy}. The case $\delta\ge 0$ corresponds   to off--peak hours while negative $\delta$  corresponds to peak--load hours.

%
%
\subsection{First--best contract}
\label{ssec:linear-FB}

\begin{Proposition}  {\rm (First--best contract)} 
\label{prop:fblinear}
Let $(f-g)(x) = \delta x$. Then: 

\medskip
{\rm (i)} the first--best value function is given by 
\[
\Vfb= U( \bar v(0,X_0) -L_0 ) \quad \text{where} \quad \bar v(t,x) = \delta (T-t) x + \int_t^T \mfb(s) \drm s,
\]
 and  
\begin{align*}
\mfb(t)  \;:=\;  \frac12 \bar \mu  (\delta^-)^2 (T-t)^2 
                       + \Hv\big(- h - \rho \delta^2(T-t)^2\big), \,\,\,\,\, \rho := \frac{rp}{r+p}.
\end{align*}

{\rm (ii)} The consumer's optimal effort is given by 
\[
\afb(t)
 := 
 \mu  \delta^- (T-t), 
 \mbox{ and }  
\bfb(t) 
 := 
 1 \wedge \Big( \lambda_j  ( h + \rho \, \delta^2 (T-t)^2) \Big)^{-\frac{1}{2}},
\]
thus inducing an optimal distribution $\P^{\mbox{\footnotesize\rm\sc fb}}$ under which the deviation process follows the dynamics 
$\mathrm{d}X_t=
-  \mu  \delta^- (T-t) \cdot \1 dt + \sigma\big(\bfb(t)\big)\cdot \mathrm{d}W_t, $ for some $
\P^{\mbox{\footnotesize\rm\sc fb}}-$Brownian motion $W$.

\medskip
 {\rm (iii)} The optimal first--best contract is given by:
  $$
    \xifb
  =
  L_0   
  - \kappa X_0 T  
    + \int_0^T  c(\nu_t) \mathrm{d}t
  + \int_0^T \piefb \big(X_0 - X_t\big) \drm t
  - \frac{1}{2}  \int_0^T  \pivfb  \drm \langle X\rangle_t, 
  $$ 
where 
\[
\piefb := \frac{r}{r+p} \kappa + \frac{p}{r+p} \theta, \quad\quad \pivfb :=  \frac{p}{r+p} h.
\]
\end{Proposition}

\noindent {\bf Proof}: See Appendix~\ref{app:fblinear}. $\Box$

\hspace{5mm}

The first--best contract admits the form of a rebate contract: the consumer is paid his certainty equivalent reservation utility minus the average value of the energy he would consume if no effort is made ($\kappa T X_0$) 
\[
\xifbf :=   L_0     - \kappa T X_0,
\]
 and then, a variable payment consisting on the cost of his efforts, which are observable and enforced in the first--best, plus a part proportional to the difference between $X_0$ and a charge for volatility variation
\[
\xifbv :=      \int_0^T  c(\nu_t) \drm t
  + \int_0^T \piefb \big(X_0 - X_t\big) \drm t
  - \frac{1}{2}  \int_0^T  \pivfb  \drm \langle X\rangle_t.
\]
 
This form is a rebate contract where $X_0$ serves as baseline. 

If the consumer is risk--neutral, the first--best contract transfers all the uncertainty of the generation cost to the consumer as it is standard in the moral hazard optimal contract framework. The first--best price for energy is simply the weighted sum by their risk--aversion of the energy values for the consumer and the producer. The first--best price for the responsiveness is a constant fraction of the direct cost of volatility if the consumer is risk--averse and it is constant. As a consequence, we see that a contract that would be indexed only on the information of the cost function of the producer is optimal only in the case when the consumer is risk--neutral. The economic intuition that the marginal cost of generation triggers a socially optimal response is correct only if the consumer is risk--neutral. If not, the consumers needs a compensation payment for the risk he takes in the contract.

Regarding the induced behaviour of the consumer on the volatilities, reduction is performed only on those usage for which the marginal cost of effort measured by $1/\lambda_j$ is lower than the marginal cost of volatility for the producer measured by $h + \rho \delta^2 (T-t)^2$. It is only in the cases where the producer or the consumer is risk--neutral or if they agree on the energy value ($\delta = 0$) that responsiveness is triggered by the mere comparison of $h$ and the $\lambda_j$.

\subsection{Second--best contract}
\label{ssec:sblinear}

\begin{Proposition} {\rm (Second--best contract)} \label{prop:sblinear}
Let $(f-g)(x) = \delta x$. Then: 

\medskip
{\rm (i)} the producer's second--best value function is given by $\Vsb= U\big( v(0,X_0) - L_0 \big)$, with certainty equivalent function 
\[
v(t,x) = \delta (T-t) x + \int_t^T \msb(s) \drm s,
\]
 where  
\begin{align*}
\msb(t)  \;:=\;  \frac12 \bar \mu  \delta^2 (T-t)^2 
                   - \frac12 \inf_{z\in\R} \big\{  \bar\mu(z^-+ \delta(T-t))^2  -2 \Hm\big( - q(z)\big) \big\}.
\end{align*}
where $q(z) := h +r z^2+p(z- \delta(T-t))^2.$\\
{\rm (ii)}  The optimal payment rates are 
the deterministic functions 
\[
 \zsb(t)
 = 
 \text{{\rm Arg}}\min_{\!\!\!\! z\in\R} \big\{  \bar\mu(z^-+\delta(T-t))^2 -2 \Hm\big( -  q(z) \big)  \big\}, 
~\mbox{and}~
\gsb(t) = - q(\zsb(t)).                                             
\]
{\rm (iii)}  the second--best optimal contract is given by $\xisb = \xisbf + \xisbv$  where
 \begin{align*}
 \xisbf & := L_0
- \kappa T X_0 
- \int_0^T H(\zsb,\gsb)(t) \drm t, \quad
 \xisbv := 
  \int_0^T \piesb(t) \big(X_0 - X_t\big)  \drm t
 - \frac12  \int_0^T \pivsb(t) \drm \langle X\rangle_t, 
\end{align*}
and
\[
\piesb (t) := \kappa + \zsbdot(t), \quad\quad \pivsb (t) :=  h + p \big(\zsb(t) - \delta (T-t) \big)^2.
 \]
{\rm (iv)}  The induced dynamics of the consumption reduction is
\[
\drm X_t = - \widehat a(\zsb(t)) \drm t + \widehat \sigma(\gsb (t) ) \cdot \drm W_t, 
\] for some $\P^{\mbox{\footnotesize\rm\sc sb}}-$Brownian motion $W$
\end{Proposition}

\begin{proof} See Appendix~\ref{app:sblinear} \qed\end{proof}


\hspace{5mm}

Similar to the first--best case, the second--best optimal contract has the form of a rebate contract. But, in the fixed part payment, the producer no longer charges the real cost of the consumer's effort, which is not observable, but the net benefit of his optimal efforts. In peak periods, the second--best payment rates are obtained up to a scalar optimisation (Proposition~\ref{prop:sblinear}~(ii)). If the consumer is risk--neutral, we can obtain that:
\[ \xisbv =  \int_0^T \theta (X_0 - X_t) \drm t  - \frac12  \int_0^T  h \, \, \drm \langle X\rangle_t. \] 
The optimal prices are the marginal costs of energy and volatility. Apart from this limiting case, the optimal price for energy is non--constant, and non--linear (see Section~\ref{sec:numeric} for illustration). 

We have seen in the Introduction that demand response contracts written in rebate form are susceptible of baseline manipulation. If we define baseline manipulation as an artificial increase of the baseline consumption to obtain a higher utility, then the optimal contract given in Proposition~\ref{prop:sblinear} does not suffer from this drawback when the consumer's reservation utility is expressed in absolute terms. In this case, whatever the initial consumption level $X_0$ used as a baseline, the consumer gets no more than the certainty equivalent $L_0$ is asked for. 

\hspace{5mm}

During off--peak periods, it is possible to obtain more explicit result for Proposition~\ref{prop:sblinear}. 

\begin{Corollary} \label{prop:RespDeltaPos}
Suppose. $(f-g)'=\delta \ge 0$. Then:
\\
{\rm (i)} the optimal payments rate are deterministic functions of time given by 
\[
 \zsb(t)
 = 
 \frac{p}{r+p} \delta (T-t), 
 \mbox{ and }
 \gsb(t)
=
- h - \rho \, \delta^2 (T-t)^2
\]
 and the second--best optimal contract of Proposition~\ref{prop:sblinear} {\rm (iii)} is defined by  the constant prices
\[
\piesb := \frac{r}{r+p} \kappa + \frac{p}{r+p} \theta, \quad \pivsb := h + \rho \frac{r}{r+p} \, \delta^2 (T-t)^2.
\]

\noindent {\rm (ii)} the consumer's optimal effort on the drift and the volatility of the consumption deviation is
\[
 \widehat a\big( \zsb \big)=0,
 \mbox{ and }
 \widehat b_j\big(\gsb \big) = 1\wedge \Big(\lambda_j  \gsb(t)\Big)^{-\frac{1}{2}},
\]
so that the optimal response of the consumer induces the optimal probability distribution $\P^{\mbox{\footnotesize\rm\sc sb}}$ 
such that $\mathrm{d}X_t = \widehat \sigma(\gsb ) \cdot \mathrm{d}W_t, $ for some $\P^{\mbox{\footnotesize\rm\sc sb}}-$Brownian motion $W$.
\end{Corollary}

\begin{proof}  (ii) Applying Proposition~\ref{prop:RespDeltaPos} and Proposition~\ref{prop:sbnonlinear} (iii)
 to the case where $\delta \ge 0$ gives directly that $\zsb = \frac{p}{r+p} v_x$ $=$ $\frac{p}{r+p} \delta (T-t)$ and thus $\gsb$ 
follows.  Further, applying Proposition~\ref{prop:sbnonlinear} (iv) together with the expression of $\zsb$ gives directly $\xisb$. (iii) The 
application of Proposition~\ref{prop:agentresponse} gives the result. Collecting the former results together with 
Proposition~\ref{prop:RespDeltaPos} (i) and our guess, gives the expression of $\Vsb$ in (i).
\qed\end{proof}

\hspace{5mm}

In off--peak periods, the price for energy is lower than its value for the consumer $\kappa$. Thus, the consumer has no incentives to sell at $\piesbm$ a good that is worth $\kappa$ for him. The consumption is a martingale. Nevertheless, the payment rate on the drift $\zsb(t)$ is positive and not zero, as one would have expected. The producer is paying the consumer when the consumption increases. While this finding seems to be in contradiction with the producer's objective, it is explained as follows. The payment rate for the volatility is $\gsb(t) + r \zsb^2(t)$ $=$ $h + p (\zsb(t) - \delta (T-t))^2  + r \zsb^2(t),$ and thus, it depends on the payment rate on the drift. Taking $\zsb \equiv 0$ is not optimal. The producer is better off paying an increase of consumption at a price higher than her marginal cost in exchange for a greater decrease of the volatility.

Regarding the price of volatitity, we note that it is higher than its direct marginal cost for the producer and that is not constant. The producer requires more effort at the begining of the period than at the end. 

Further, the risk--sharing process between the producer and the consumer may lead to an increase in the observed volatility of consumption. To see that, consider the limiting case where $\delta =0$ and $h=0$. Then, the optimal payment $\gsb$ is zero which means that the producer does not provide responsiveness incentives, even if both players are risk--averse. In this situation, the observed  volatility is equal to the nominal volatility. But, it might happen that this quantity is greater than the observed volatility before contracting. If the consumer is sufficiently risk--averse, he might have performed some efforts to reduce the nominal volatility. In fact, there is a large set of parameters value for which the producer is taking the volatility risk for the consumer. From a risk--sharing point  of view, this means that optimal contracting allows the system to bear more risk. This phenomenom is illustrated in Figure~\ref{fig:totalSigma}.

\begin{figure}[bht!]
\begin{center}
\includegraphics[width=0.5\textwidth]{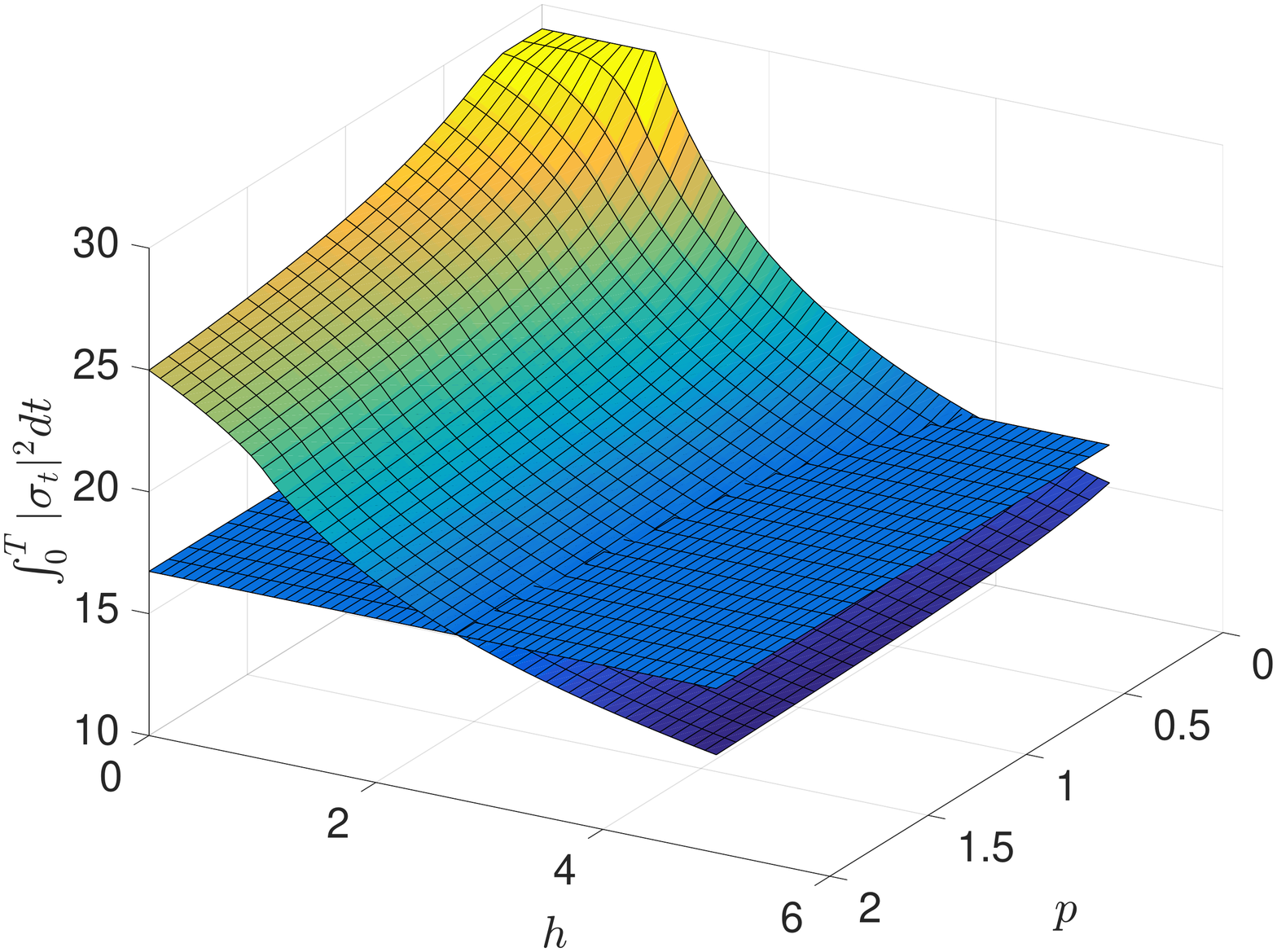}\includegraphics[width=0.5\textwidth]{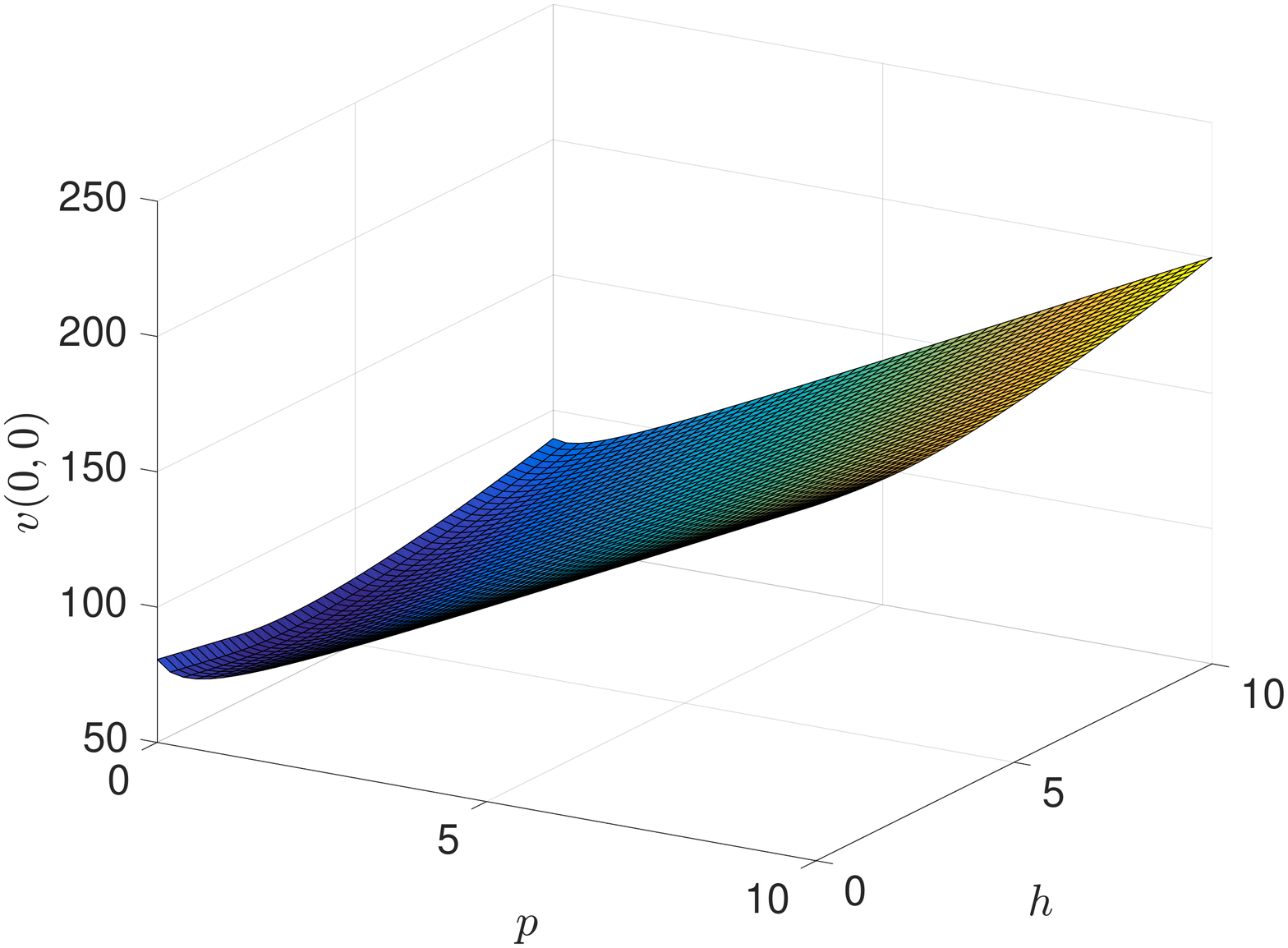} 
\caption{(Left) Total volatility of consumption deviation under optimal contract as a function of the direct volatility cost $h$ and the 
risk-aversion parameter $p$ of the consumer compared to the total volatility without contract (flat surface). (Right) Certainty 
equivalent of the producer with contract and without contract as a function of the direct volatility cost $h$ and the risk-aversion 
parameter $p$. parameters value:  two usages, $T$ $=$ $1$, $\mu$ $=$ $(2,5)$, $\sigma$ $=$ $(5,2)$, $\lambda$ $=$ 
$(0.5,0.1)$, $\kappa$ $=$ $5$, $\delta$ $=$ $3$.}
        \label{fig:totalSigma}
        \end{center}
\end{figure}

Figure~\ref{fig:totalSigma} shows the variation of the observed total volatility of consumption deviation and the benefit from the 
contract as a function of the direct marginal cost of volatility $h$ and the risk--aversion parameter of the producer $p$ when $\delta >0$. For a risk neutral producer with zero marginal direct cost of volatility, the producer requires no effort from the 
consumer and the volatility is equal to the nominal volatility. For a risk neutral producer with increasing direct marginal cost of 
volatility, we observe that volatility remains constant until the value of $h$ gets higher than a certain threshold, namely the lower 
cost of effort for volatility reduction of the consumer. Then, the producer starts to require an effort from the consumer. Then, the 
higher the value of $h$, the lower the observed volatility until an effort is required that reduces the volatility below its value before 
contracting. If we fix the value of $h$ and make the producer increasingly risk averse, we also observe a progressive yet slower 
reduction of volatility because the indirect cost of volatility is much lower than its direct cost. Further, we observe that the gain from 
the contract is a non--decreasing function of $p$ but a non-monotonic function of $h$. Indeed, even if the benefit from the contract 
is always positive, the producer who has a small direct marginal cost does not induce any 
effort from the consumer and thus, cannot reap any increased benefit from the behaviour of the consumer.

\subsection{Comparisons}

In order to assess the benefits from responsiveness incentives, we consider the situation where the producer does not provide incentives on the responsiveness of the consumer, although he could reduce his consumption's volatility if properly incentivised. We consider the  problem~\eqref{eq:SB} of the producer where the contracts are limited to incentives indexed on the observed consumption and not on its volatility. Proposition~\ref{prop:sblinear} in Appendix~\ref{app:sblinear} gives a precise meaning of this sub--class of contracts. We denote by $\Vsbm$ the value of the producer's problem in this case. In this situation, the behaviour of the consumer without contracting is still given by Proposition~\ref{prop:crulinear}.

\begin{Proposition}[Second--best without responsiveness incentives] \label{prop:nogamma} 
Assume that $f-g = \delta x$. Then:\\
(i) $\Vsbm  =  U\big( w(0,X_0) - L_0\big)$ where 
\[
w(t,x) = \delta (T-t) x + \int_t^T \msbm(s) \drm s,
\] 
 and  
\begin{align*}
\msbm(t)  \;:=\;  \frac12 \bar \mu  \delta^2 (T-t)^2 
                   - \frac12 \big\{  q(\zsbm) |\sigma|^2
                   + \bar\mu( \zsbm^- + \delta(T-t))^2 \big\},
\end{align*}
with $q(z) := h +r z^2+p(z- \delta(T-t))^2.$\\
(ii) the second--best optimal payment rate is $\zsbm(t) = \Lambda \delta (T-t) $ with 
 $\Lambda$ 
$ :=$  
$ \frac{p |\sigma|^2+\bar \mu \1_{\{\delta<0\}}} {(p+r) |\sigma|^2+ \bar \mu \1_{\{\delta<0\}}}.$\\                        
(iii) the second--best optimal contract is given by $\xisbm = \xisbmf + \xisbmv$ where
$$
 \xisbmf
 =
L_0
 - \kappa  T X_0 
 + \frac12  \int_0^T r \zsbm^2 (t) |\sigma|^2 \drm t
 - \int_0^T H_{\rm m}(\zsbm(t)) \drm t, \quad \xisbmv =  \int_0^T\piesbm \big(X_0 - X_t\big) \drm t,
$$
and the price for energy $\piesb$ and the price for volatility $\pivsbm$ are given by:
\[
\piesbm := (1-\Lambda) \kappa + \Lambda \theta, \quad \pivsbm := 0.
\]
\end{Proposition}

\begin{proof} See Appendix~\ref{app:nogamma} \qed\end{proof}

\hspace{5mm}

The price for energy is still a weighted sum of the value of energy for the consumer and the producer. Besides, we have:
\[
\piefb - \piesbm := \underbrace{(\Lambda - \frac{p}{r+p})}_{\geq 0} \delta.
\]

During off--peak period ($\delta >0$), the second--best price is lower than the first--best optimum. During peak--load period, the producer pays or charges a price greater than what is socially optimal.

\noindent We turn now to the measure of the welfare loss due to the information rent of the consumer because of his hidden actions. Define the informational rent as $I := - \frac{1}{p} \log \big(\frac{\Vfb}{\Vsb}\big)$. The proof of the following Proposition giving the value of the information is found in Appendix~\ref{app:inforent}.

\begin{Proposition} {\rm (Comparisons)}\label{prop:inforent}
Assume that $f-g = \delta x$. Then:
\medskip
{\rm (i)} If $\delta \ge 0$, then there is no informational rent, $I = 0$ and $\xisb = \xifb$. 

\medskip
{\rm (ii)} When $\delta \le 0$ and $h+r \delta^2 T^2 \le 
\frac{1}{\bar \lambda}$, the informational rent is
\[
 I 
 = 
 \frac{\delta^2T^3r^2}{6(p+r)}
 \frac{1}{\frac{1}{|\sigma|^2}+\frac{p+r}{\bar\mu}}.
\]
\end{Proposition}

\vspace{5mm}

During off--peak period, there is no information rent. In this case, in terms of volatility reduction, the optimal contract implements the same  efforts as would do the social planner. In this situation, the producer has only one objective to achieve, reducing the volatility and not the expected consumption. He can achieve this objective with the two instruments which are the two payment rates $\zsb$ and $\gsb$.  During peak periods, this is no longer the case. The two payment rates are constrained by one another because the payment rate on the volatility is a deterministic function of the payment rate on the drift. Note that the information rent remains positive in the case~(ii) even if there is only one usage. The intuition would be that if there is only one usage, the producer can recover the effort by observing the quadratic variation and thus enforce optimal effort on the volatility. Nevertheless, this is not enough to enforce the first--best as the efforts on the drift remains unobservable to the producer.

\section{Numerical illustration}
\label{sec:numeric}

The purpose of this section is to provide an estimate of the potential gains from the implementation of a responsiveness incentive 
mechanism, both in terms of increase in consumer's response to price events and of benefits for the producer. We first calibrate our model with linear energy 
valuation and energy generation cost on the publicly available data set of Low Carbon London Pricing Trial using the idea that it corresponds to the implementation of an optimal contract without responsiveness incentives. All parameters can be estimated with this strategy at the exception of  the marginal costs of volatility reduction $\lambda_i$. We screen the effects of a range of potential values for this last parameter, design a reference value and proceed to robustness analysis.

\subsection{Data and model calibration} 
\label{ssec:data}

The Low Carbon London Project of demand-side response (DSR) trial  performed in 2012--2013 was conducted at the initiative of the UK energy regulator (Ofgem) in partnership with both industrial players and academic institutions amongst which two have to be cited for our study, Imperial College who treated the data of the experiment and EDF Energy who acted as the energy provider and enroller for the consumers. The data consists in a 
set of 5,567 London households whose consumption have been measured at an half hourly time--step on the period from February
2011 to February 2014. For the dynamic Time--of--Use (dToU) tariff trial, the population was divided in two groups. One group of 
approximately 1,117 households were enrolled by EDF Energy in the dToU tariff while the remaining 4,500 households were not 
subject to this dynamic tariff. The dToU was applied during the year 2013 (January, 1st to December, 31st). Tariffs were sent to the households on a day--ahead basis using a Home Display or a text message to the customer mobile phone. Prices had three levels: High (67.20 p/kWh), Normal 
(11.76 p/kWh) and Low (3.99 p/kWh). Standard tariff is made of a flat tariff of 14.228 p/kWh. 

The precise description of the dToU trial performed in 2013 is given in Tindemans et al. (2014)  \cite[chap.~3]{Tindemans14}. The 
total number of events (High and Low) were 93 to deal with supply events (shortage of generation) and 21 for distribution network 
events. In our study, we are only interested in the High price events. There were 69 such events of High prices (45 for supply 
reasons and 24 for network reasons). The duration of an event could be 3, 6, 12 or 24 hours. The Low Carbon London 
Demand--Side Response Trial was designed to be as close as possible to a random trial experiment, while accounting for the 
operational constraints related to the enrolment of a large set of customers within the portfolio of given UK utility (EDF Energy). 
The events were randomly placed over the trial period while targeting the highest peaks of demand in the year.

The data collected by the Low Carbon London Project of demand-side response (DSR) trial 
performed in 2012--2013 can be downloaded freely at London DataStore website (\url{https://data.london.gov.uk}) under the 
section {\em Smart Meter Energy Consumption Data in London Households}. The demand response trial is extensively described 
in a series of reports amongst which the reports Tindemans et al. (2014) and Schofield (2014) \cite{Tindemans14,Schofield14} are 
the most relevant for our study. Out of this dataset, we eliminated all consumers for which data were not complete or exhibited outliers. The resulting sample 
consists in 880 consumers in the control group and 250 consumers in the dToU group.

As mentioned in Section~\ref{ssec:sblinear}, the optimal contract is the sum of a constant term plus a term proportional to the consumption. Thus, the optimal contract has the same form as the LCL pricing trial contract: a fixed premium to get enrolled plus a term proportional to the consumption. This provides the rationality for the calibration of the optimal contract without responsiveness control to the data of the LCL pricing trial. Thus, our strategy to calibrate our model is to use Proprosition~\ref{prop:nogamma} to answer the question: what should be the parameters value of the consumer's behaviour model $\mu$ that would lead to the observed consumption reduction of the LCL pricing trial? Furthermore, because our model relies on simplifying assumption regarding the pattern of daily consumption, we fix the initial condition of the consumption to be zero ($X_0 =0$), making $X_t$ directly the observed reduction of consumption.  

\no {\bf Duration of the price event $T$.} In the LCL pricing trial, there were 69 High Price events for a total of 778 half--hours. The events could last 3, 6, 12 or 24 hours. Only one exceptional event lasted a full day (24 hours). Removing this outlier, we find an average duration of price event of 5.44~hours. We set $T = 5.5$~hours.

\no {\bf Energy value parameters $\kappa$ and $\theta$.} We have seen in Proposition~\ref{prop:sblinear} that $\kappa$ should be lower than $\theta$ to justify an average consumption reduction. Thus, we set the marginal value of electricity of the consumer to $\kappa = 11.76$~pence/kWh, which is the price the consumer enrolled in the dToU group pays in normal situation and we set the marginal cost of electricity generation to $\theta = 67.2$~pence/kWh. This setting clearly refers to a high peak demand situation when the producer has a strong interest in avoiding costly generation.

\no {\bf Nominal volatility $\sigma$.} As pointed in the introduction, there is significant noise in the observed reduction of the consumers enrolled in the dToU tariff. For an average reduction of consumption of 40~W for a consumption of magnitude of 1~kW, the estimation of the responses range between -200~W and + 200~W. No direct estimation of the standard deviation of the reductions are reported in LCL Pricing Trial reports. We thus performed an estimate of the volatility of the consumption of the control group during price event using the fact that, given our model, 
$$\mathrm{Var}\Big[\frac{1}{T}\int_0^T X_t \drm t \Big] = \frac13 T \sigma^2,$$
where the variance is computed under the no--effort distribution $\P^{(0,1)}$. We  estimate an average volatility of $\sigma = 85$~W.h$^{\frac12}$.

\no {\bf Producer's risk-aversion $p$.} Let $S$ denote the spot-price of electricity for a given hour and $F$ the forward price quoted the day before, one has $\E[e^{-p S}] \approx e^{-p (\E[S] - \frac12 p \sigma^2_S)}$, and by equating the certainty equivalent with the forward price, we obtain the risk--premium $\text{RP}:=F - \E[S]=\frac12 p \sigma^2_S$. The risk--premium electricity utilities are ready to pay to avoid the day--ahead spot price risk has been extensively analysed and estimated in the financial economics literature. Bessembinder and Lemon (2002) \cite{Bessembinder02} followed by Longstaff and Wang \cite{Longstaff04}, Benth et al. \cite{Benth08} and Viehmann (2011) \cite{Viehmann11} estimated the relation between the risk premia on each hour of delivery and the variance of the spot price on this hour. They find consistent and convergent estimation both on the sign of the risk premia (negative for off--peak hours and positive for peak hours). If one focuses on the peakest hour of the day (typically 7 or 8pm), the former authors find that dependance of the risk--premium with respect to the variance of the spot price is $0.31$ for Viehmann (2011) (Table 5, hour 20), which makes $p = 0.62$; Benth et al. (2008) estimates that $p$ is no lower than $0.421$ (page 14); Longstaff and Wang (2004) find a dependence of the risk--premium to the variance of the spot price of $0.29$ (page 1895, Table VI, hour 20), which makes $p=0.58$. Bessembinder and Lemon (2002) estimates risk-premia not for day-ahead spot price risk but for monthly prices, which is less relevant in our context. Thus, we take as a nominal value for the risk--aversion parameter of the producer $p=0.6$ per pound.

\no {\bf Consumer's risk-aversion $r$.}  There is a large and not necessarily consensual economic literature on the relevant estimation of consumer's risk-aversion parameters, in particular when using CARA utility function (see Gollier (2004) monography \cite{Gollier04}). Nevertheless, in the context of the LCL Pricing trial, the consumers were facing a small variation of their electricity bill which is itself a fraction of their expenses, making the approximation of independence of decision with respect to wealth sustainable. Further, it is possible to provide an estimate of the risk--aversion parameter $r$ of the population who accepted to enrol in the dynamic ToU. Indeed, the enrolled consumers were paid 100~\textsterling~at the beginning of the trial and 50~\textsterling~more if they completed all the trial. Besides, we estimate the financial risk taken by consumers adopting dynamic ToU tariff. We computed for each consumer of the control group the electricity bill with the two possible tariffs, the standard flat tariff and the dynamic ToU tariff. We found that the consumers were facing a risk with a statistically significant standard-deviation of $23$~\textsterling~at the 5\% level. Using the relation between the risk-premium ($150$~\textsterling) and the risk level ($23$~\textsterling) in the relation giving the certainty equivalent of a risk of known standard deviation for an exponential utility function, we estimated an absolute risk--aversion $r=0.56$~per pound, which is very close to the producer's risk aversion parameter.

\no {\bf Consumption variation cost $h$.} This parameter is related to the flexibility of the producer's generation capacities. The higher the flexibility of the generation, the lower the variance of consumption induces costs. With the development of intermittent energy sources, the quantification of the flexibility of a given power system has attracted the attention of researchers. For a review of this topic, we refer to Hirth (2015) \cite{Hirth15}. The value for $h$ depends on the whole electric system, and not only on the capacity of a single power plant. There is a difference of flexiblity between the electric system of Norway which relies only on hydraulic generation and an electric system based on wind generation and coal--fired plants. Nevertheless, if we focus on peak period of the day where flexibility is provided by gas--fired plants, we can make use of the estimations that exist for the cost of flexibility provided by power plants (see Kumar et al. (2012) \cite{Kumar12} and Oxera (2003) \cite{Oxera03}, Table~3.2 p.~8 and Van den Bergh and Delarue (2015) \cite{VandenBergh15}, Table~IV). Estimates find consistent values of order of magnitude of 25 to 42~\euro/MW$^2$.h for gas fired plants, which is in general the technology used in peaking period of the day. Thus, we choose a nominal value of $h$ $=$ 40~\euro/MW$^2$.h to consider a not so flexible system in which there may be room for flexibility exchange.

\no {\bf Costs of effort on average consumption $\mu$.} Because we do not have access to data at the usage level, we consider a single average usage. In order to fix an estimate of $\mu$, we interpret the LCL experiment as the implementation of our demand--side model when there is no control of responsiveness or volatility (see Proposition~\ref{prop:nogamma}). The conclusion of Schofield et. al. (2014) provides an estimate for the realised average consumption reduction of 40~W.  According to Proposition~\ref{prop:nogamma}, the absolute value of the average consumption deviation is given by:\[ \frac{1}{T} \Big| \E\Big[\int_0^T X_t \drm t \Big]\Big| = \frac{1}{3} \Lambda \bar \mu |\delta| T^2. \] Recalling that $T=5.5$, we obtain the value $\bar \mu = 9.3 \, 10^{-5}.$ With this value of the parameter $\mu$, the reduction of $40$~W is obtained at the expense of a cost for the consumer given by  
$$\int_0^T \widehat c_1(z_t) \drm t=\frac{1}{6} \bar \mu  \Lambda^2 \delta^2 T^3 = 4.7~\text{pence}.$$

\no {\bf Costs of effort on consumption volatility $\lambda$.} We recall that higher values of $\lambda$ correspond to lower costs of effort for responsiveness. Because we have no way to calibrate on data a possible value for this parameter, we observe the total cost of volatility reduction as a function of $\lambda$ and compare it to the total cost of average consumption reduction. Figure~\ref{fig:FindingLambda} (Left) shows the numerical estimation of the total cost of volatility reduction and the total cost of average consumption reduction as a function of $\lambda$ with all other parameters fixed at the values above. For low values of $\lambda$, the corresponding consumer's effort is too costly and thus, the resulting cost is zero (no effort). Then, as $\lambda$ increases, the consumer starts marking efforts to be more responsive. When $\lambda$ becomes large, the total cost of volatility reduction starts to decrease. We note that in this setting, the cost of volatilty reduction is one order of magnitude lower than the cost for the average consumption reduction. Because we are interested in the question of what would happen, if the consumer accepts to sign contracts indexed on his responsiveness, we take as a reference value for $\lambda$ the value that corresponds to the maximum of total cost of efforts. This choice corresponds to a worst case scenario for the consumer in terms of costs. It does not correspond to a maximum volatility reduction as Figure~\ref{fig:FindingLambda} (Right) demonstrates. We find a value of $\lambda = 2.8\,10^{-2}$.

\begin{figure}[h!]
\begin{center} 
\includegraphics[width=0.45\textwidth]{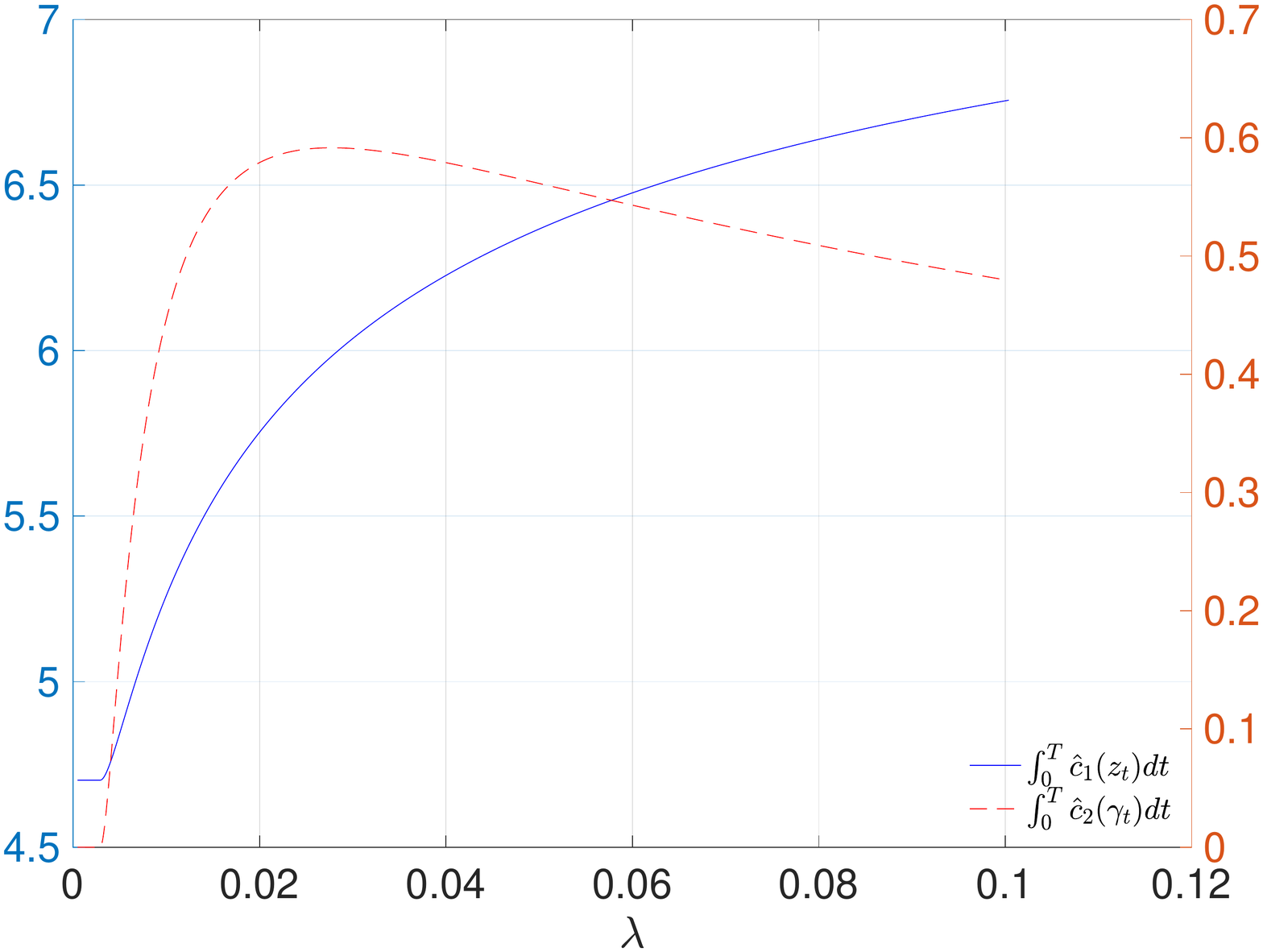}
\includegraphics[width=0.45\textwidth]{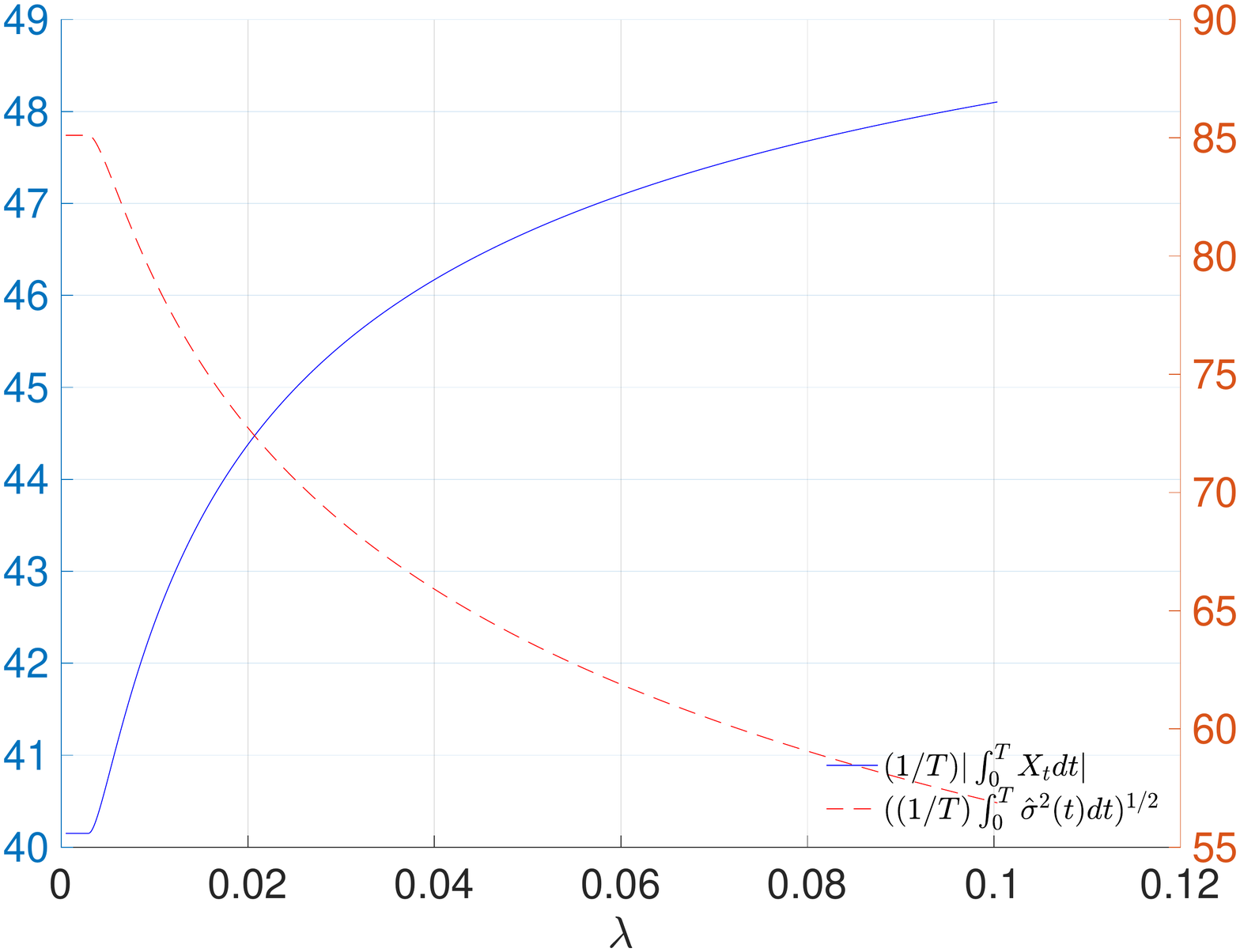}
\caption{ (Left) Total costs of effort for average consumption reduction (left axis) and for volatility reduction (right axis). Values in pence. (Right) Total volatility and average consumption reduction in Watt.}
        \label{fig:FindingLambda}
        \end{center}
\end{figure}

 We summarise in Table~\ref{tab:params} below the reference case for the calibration of our model.
\begin{table}[thb!]
        \centering
        \begin{tabular}{ |c|c|c|c|c|c|c|c|} 
        \hline\hline
$T$ &        $h$   &     $\delta$    & $p$           &  $r$        &   $\sigma$   & $\mu$           &     $\lambda$           
        \\
{\small (h)} &   {\small (p/kW$^2$h)}      & {\small (p/kWh)}  & {\small (p$^{-1}$)}  & {\small (p$^{-1}$)} & {\small (W/h$^{1/2}$)} & {\small (kW$^2$h$^{-1}$p$^{-1}$)} &  {\small ($p^{-1}$kW$^2$h)}          \\ \hline
        & & & & & & & 
         \\
	5.5 & 4.0~10$^{-4}$ & $-55.44$ & $0.6~10^{-2}$ & $0.57~10^{-2}$ & 85 & 9.3~$10^{-5}$ & $2.8~10^{-2}$ 
	\\
       & & & & & & &
	\\ \hline\hline
        \end{tabular}
        \caption{Nominal values for model parameters.}
        \label{tab:params}
\end{table}


%

\subsection{Responsiveness incentive estimated benefit} 
\label{ssec:responsiveness}

Before we examine the potential benefits of the implementation of a responsiveness incentive mechanism such as the one proposed by our  model, we show on Figure~\ref{fig:prices} the different energy and volatility prices defined in Section~\ref{sec:model}. The second--best price of energy without responsiveness incentives is significantly different from the first--best and also from the marginal cost of energy. The incentive for responsiveness leads to a non--constant price of energy. It lies between the marginal cost of energy and the second--best price without responsiveness incenticve. Further, it is higher at the beginning of the price event to trigger a quick response of the consumer. The price for volatility follows the same pattern of decreasing value. Note that it significantly different from the first--best price of volatility, which is here $\piefb = h \frac{p}{r+p}$ $=$ $2\, 10^{-3}$~pence/kW$^2$.

\begin{figure}[h!]
\begin{center} 
\includegraphics[width=0.45\textwidth]{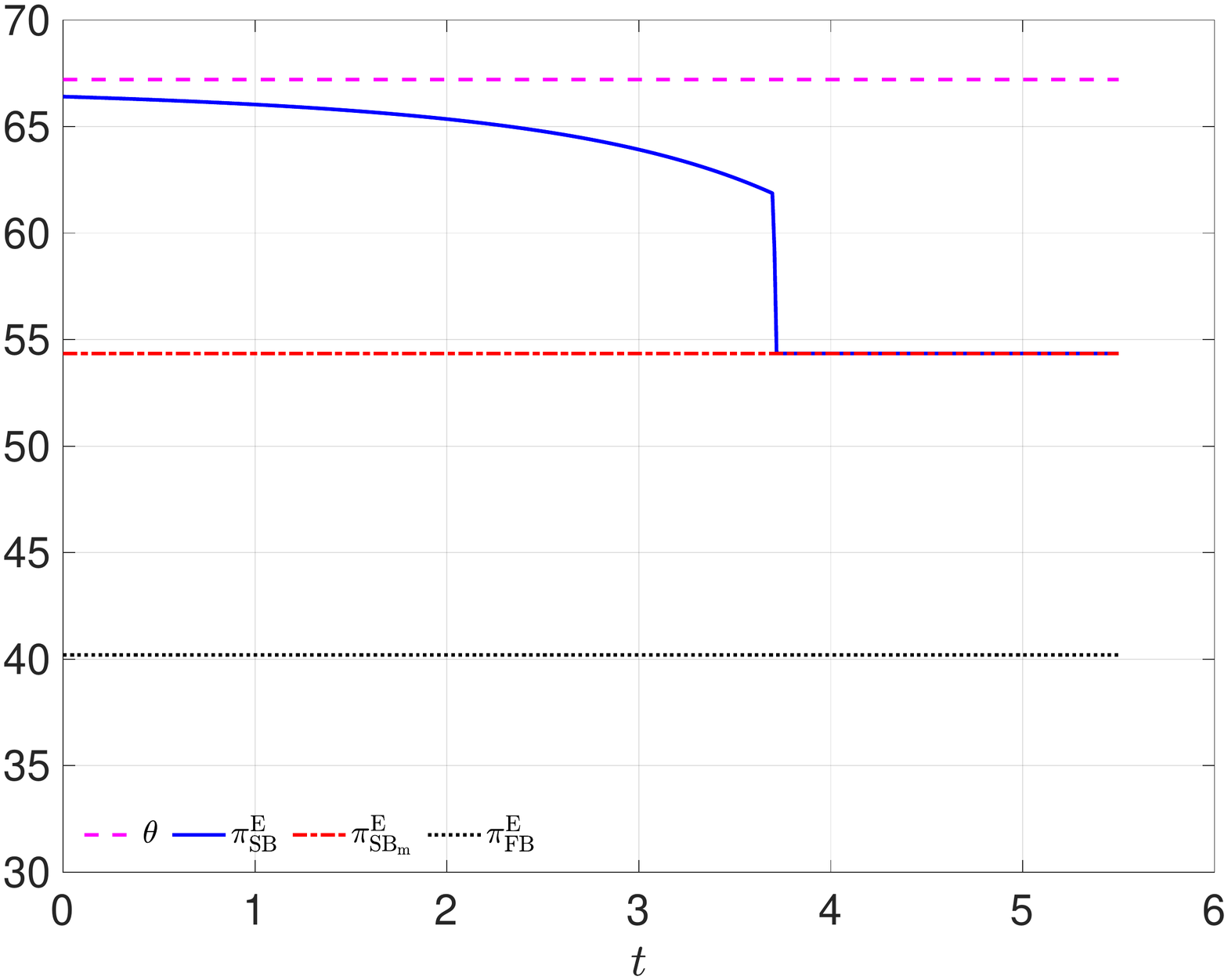}
\includegraphics[width=0.45\textwidth]{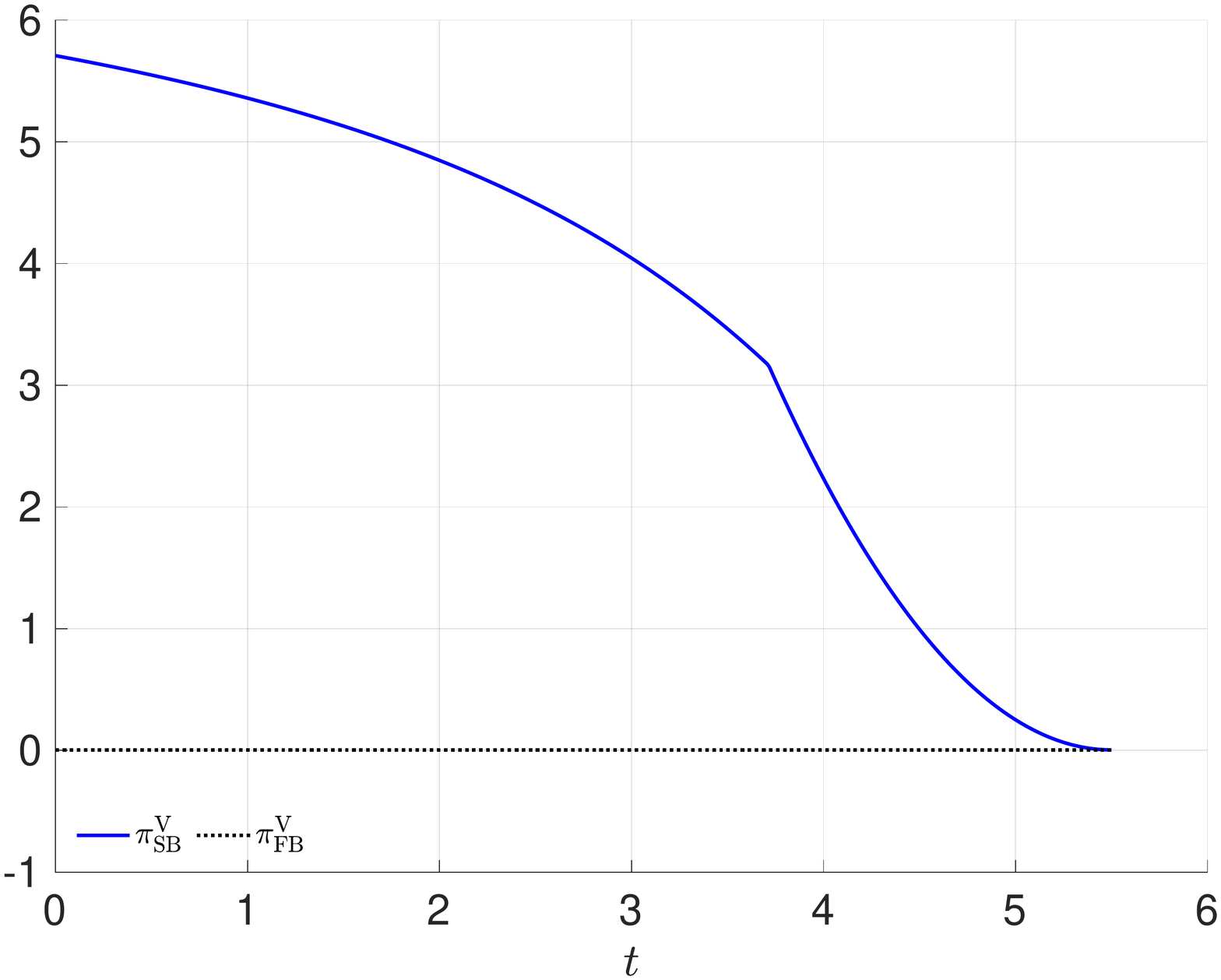}
\caption{Prices for energy (left) and volatility (right).}
        \label{fig:prices}
        \end{center}
\end{figure}

Table~\ref{tab:payments} provides the consequences in terms of costs of efforts for the consumers and benefits for the producer of the calibration of our model with the nominal parameters value summarised in Table~\ref{tab:params} and variants for the parameter $\lambda$.

%
%
\begin{table}[hbt!]
        \centering
        \begin{tabular}{ |l c c c |}
        \hline\hline
                            & First--best & Second--best & Second--best \\
        & & with  & without   \\ 
        & &  responsiveness &  responsiveness  \\  
& & & \\ 
Cost of effort $c_1$        & $5.97$ & $5.97$  & $4.68$ \\ 
Cost of effort $c_2$   &        $0.40$ & $0.59$ & 0 \\ 
Total cost of effort          & $6.37$    & $6.56$ &  $4.68$  \\    
& & &  \\ 
Producer's benefit               & $6.76$ & $6.21$ & $5.40$ \\ 
& & & \\ 
Average  consumption reduction & $52.15$    & $45.17$ &  $40.00$ \\ 
Standard deviation of reduction & $46.49$    & $39.61$ &  $85.06$  \\ \hline\hline 
    \end{tabular}
        \caption{Costs in pence, consumption and standard deviation in Watt.}
        \label{tab:payments}
\end{table}

Because in the second--best with responsiveness incentive, the payment rate $\zsb(t)$  also depends on $\lambda$, the cost of average energy reduction raises from 4.68~pence when there is no responsiveness incentive to 5.97~pence when there is one. This increase is explained by the increase in average consumption reduction. Without responsiveness incentive, the average consumption is 40~Watt as given by the calibration of the model. The incentive to reduce volatility leads to an average reduction of 45~Watt. This increase in average reduction explains most of the increase of the efforts performed by the consumer. The cost of effort for the reduction represents only ten percent of  the cost for average consumption reduction. The increase in efforts of the consumer is almost all converted into benefit for the producer. The producer is able to increase her certainty equivalent by  15\% when implementing responsiveness control and to divide by two the consumption volatility. The first--best indicates that the socially optimal strategy would be to reduce less the volatility and more the average consumption. Besides, in this nominal situation, the best the producer could hope to achieve is to increase her certainty equivalent by 25\%.

\begin{figure}[hb!]
\begin{center}
\includegraphics[width=0.32\textwidth]{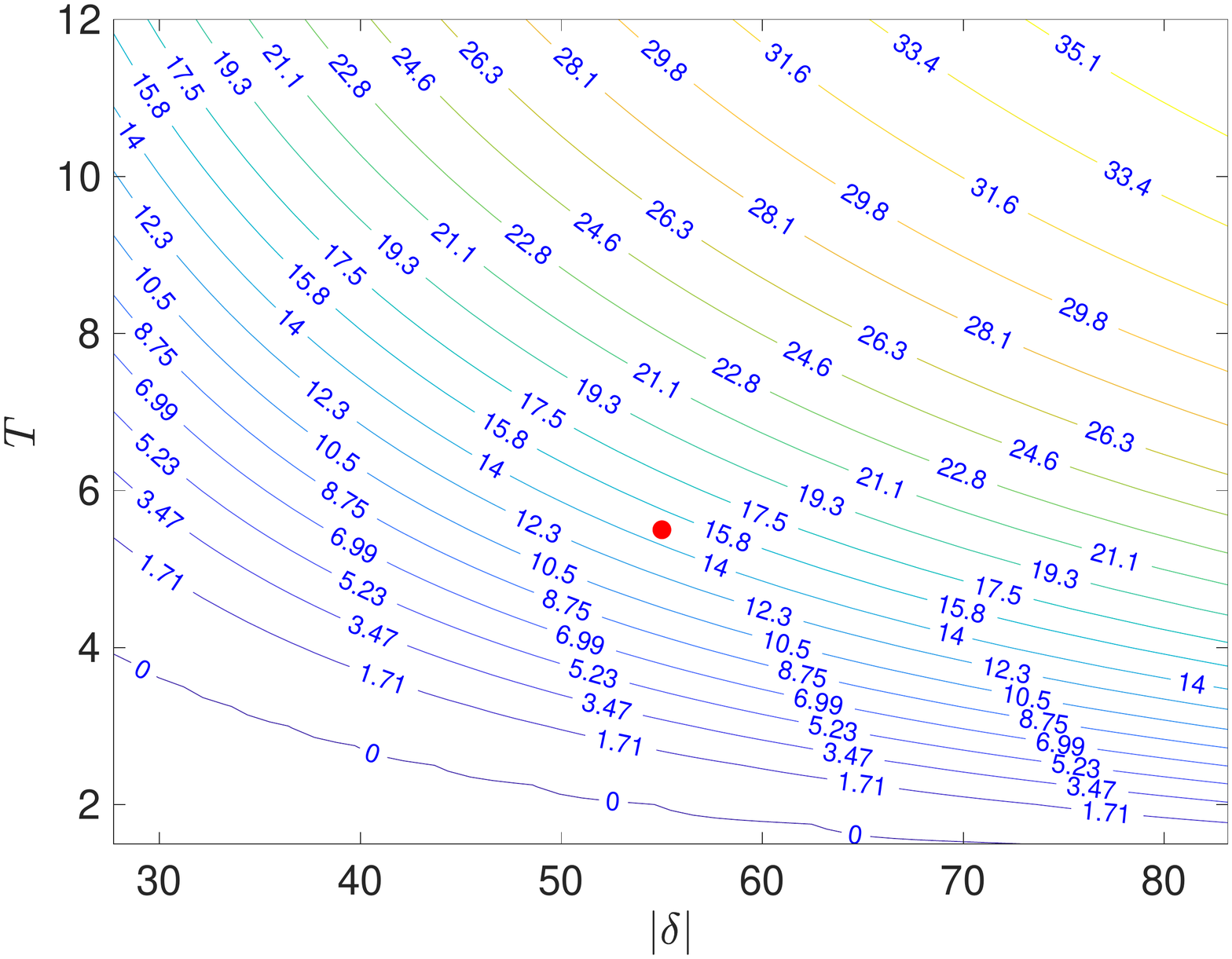}
\includegraphics[width=0.32\textwidth]{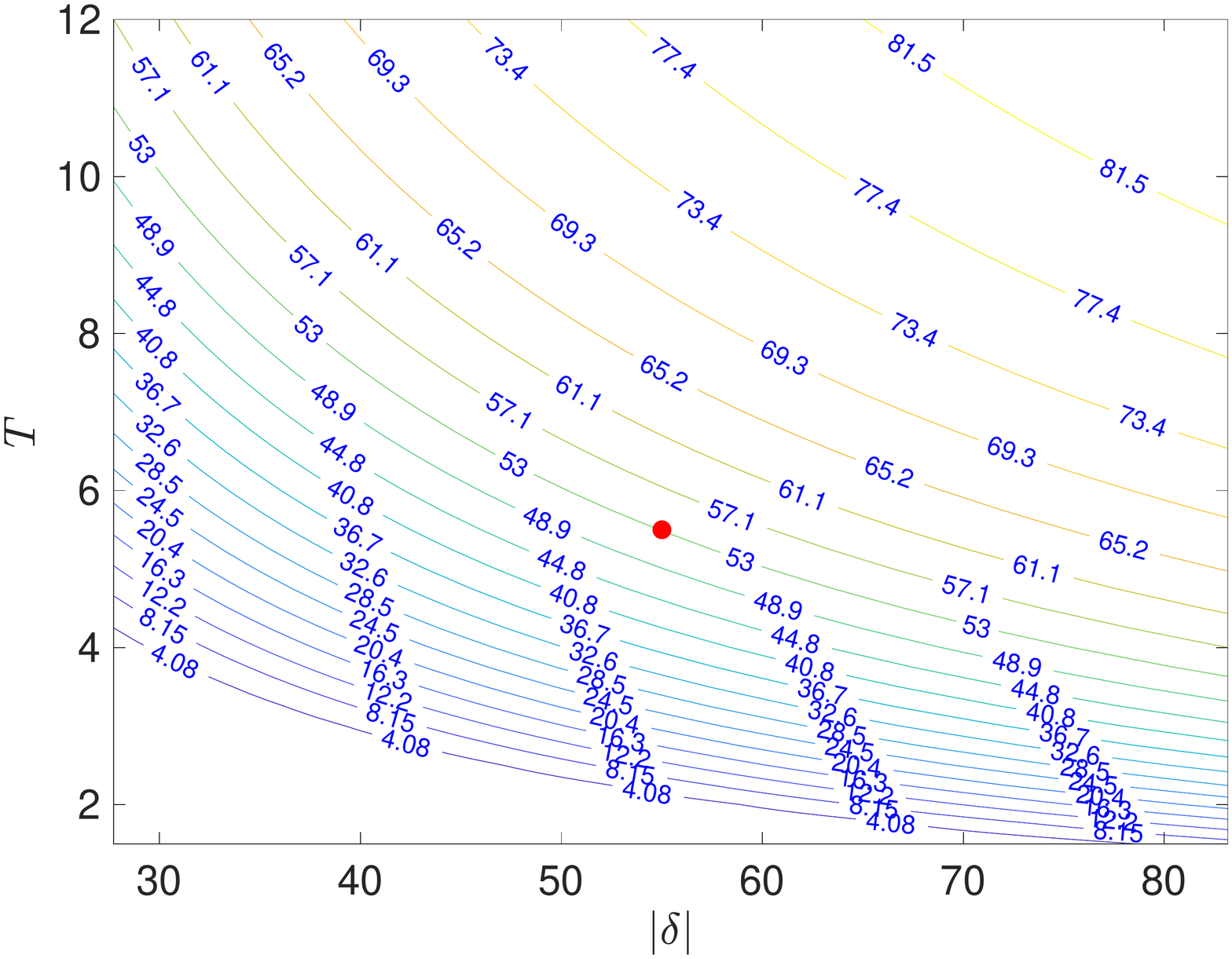} 
\includegraphics[width=0.32\textwidth]{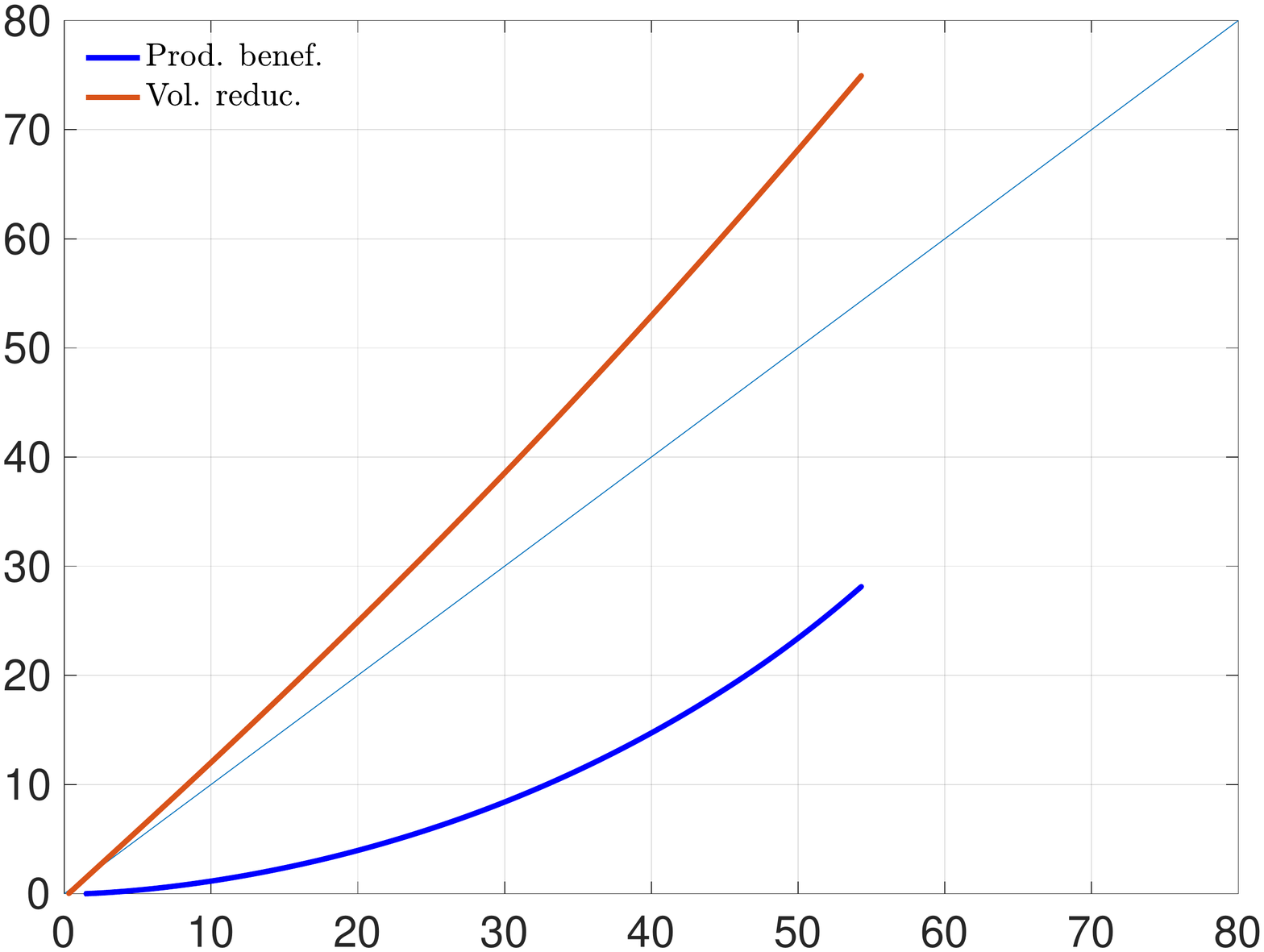} 
\caption{(Left) Gain for the producer from responsiveness control and (Midlle) volatility reduction as a function of the price event duration $T$ and the absolute value of the energy value discrepancy $\delta$; all in percentage. (Right) Producer's gain from responsiveness incentive and volatility reduction as a function of the percentage of increase of the total efforts of the consumer when $\lambda$ varies.}
        \label{fig:Tdelta}
        \end{center}
\end{figure}

In order to assess the sensitivity of the last results to the calibration of the model, Figure~\ref{fig:Tdelta} (Left and Middle) presents the sensitivity analysis of the gain from responsiveness control and the reduction of of the volatility as functions of $T$ and $ |\delta| = \theta - \kappa$. The red dot in the pictures represents the nominal situation. We considered shorter and longer price event up to 12 hours and considered situations with lower energy value discrepancy. We observe that there is a threshold of values of energy value discrepancy and price event duration under which no benefit should be expected from the responsiveness incentives. The lower the energy value discrepancy, the longer the price event should be to ensure a significant benefit of responsiveness control. The incentive on volatility needs time or a large energy value discrepancy to show its benefits. But, on the other hand, the reduction of volatility is less prone to this dependence on the energy value discrepency. Even modest differences can induce subtantial reduction of volatility for a standard duration of a price event. This phenomenon is stressed in Figure~\ref{fig:Tdelta} (Right). We vary the parameter $\lambda$ from a very low value that triggers no effort to very large values and computed the percentage of increase in the producer's certainty equivalent and the percentage of decrease of volatility induced by the responsiveness incentive. The resulting pictures can be read in the following way: an increase by 40\% of the efforts of the consumer increases by 15\% the certainty equivalent of the producer and reduces by 53\% the volatility of consumption. Besides, whatever the increase of $\lambda$, the increase in the consumer's total cost of effort stabilises at 50\% and the producer's benefit increase also reaches a limit of 30\% while the volatility reduction is limited to 75\%.

\begin{figure}[ht!]
\begin{center}
\includegraphics[width=0.33\textwidth]{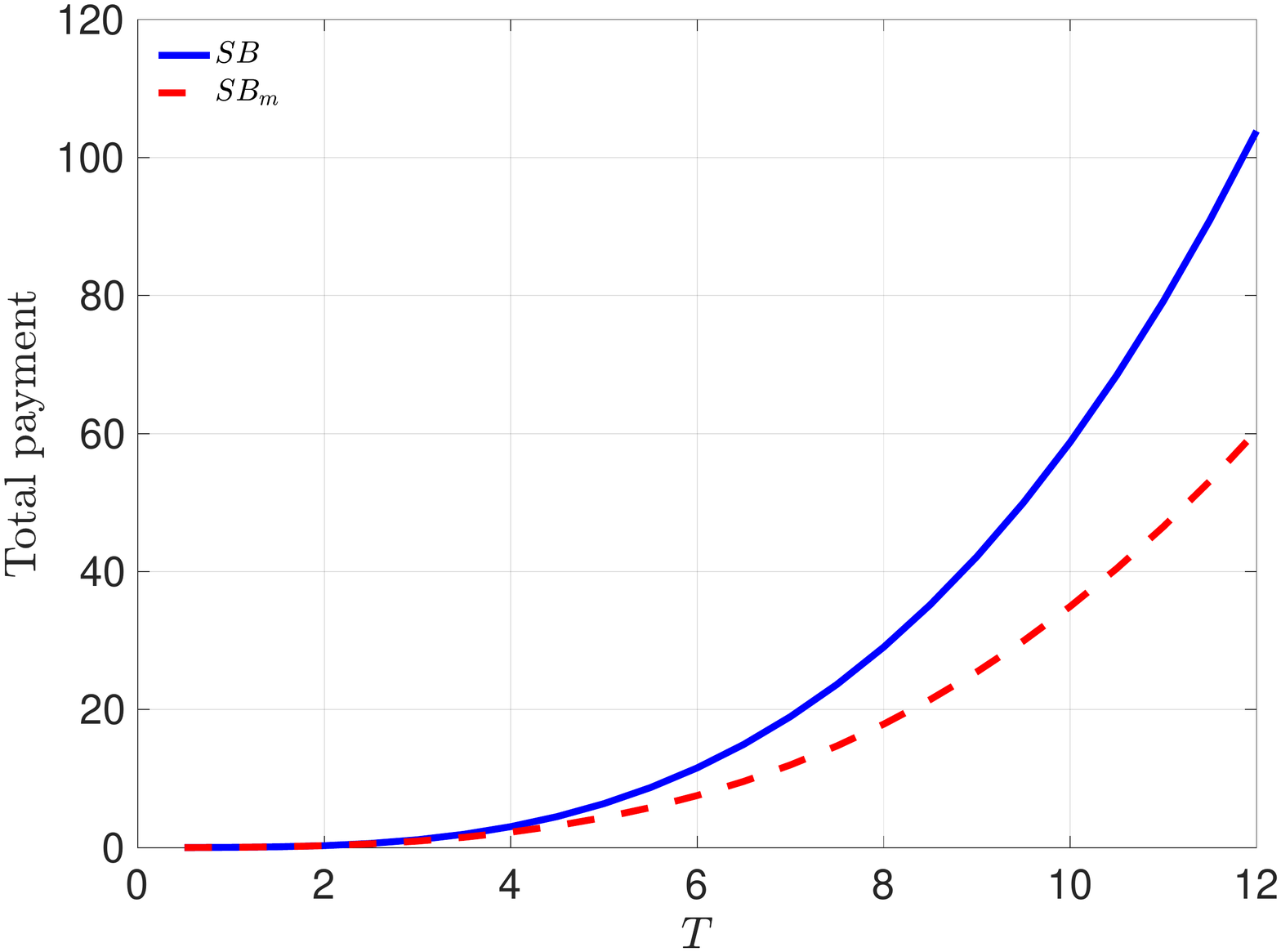}\includegraphics[width=0.33\textwidth]{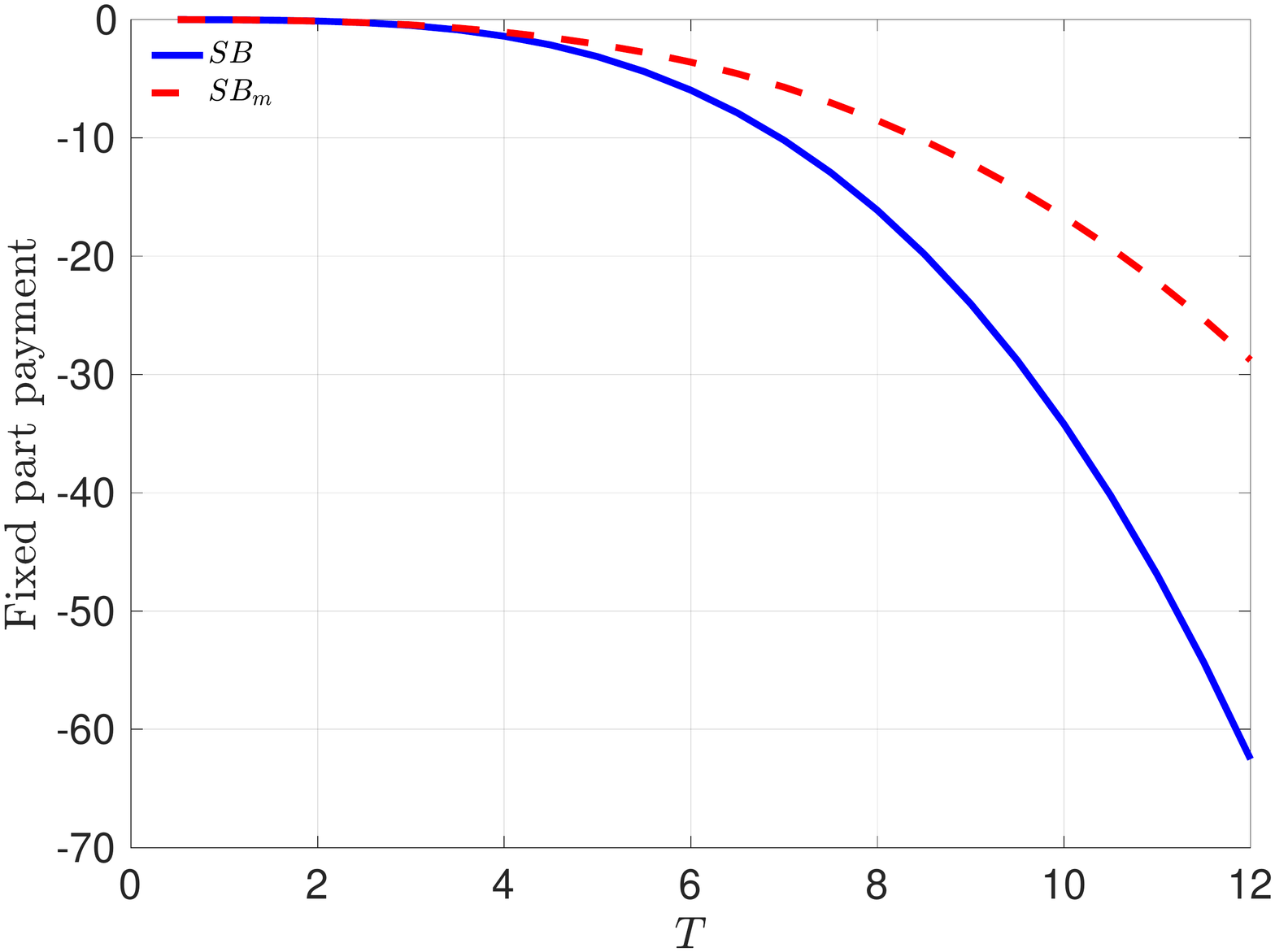}\includegraphics[width=0.33\textwidth]{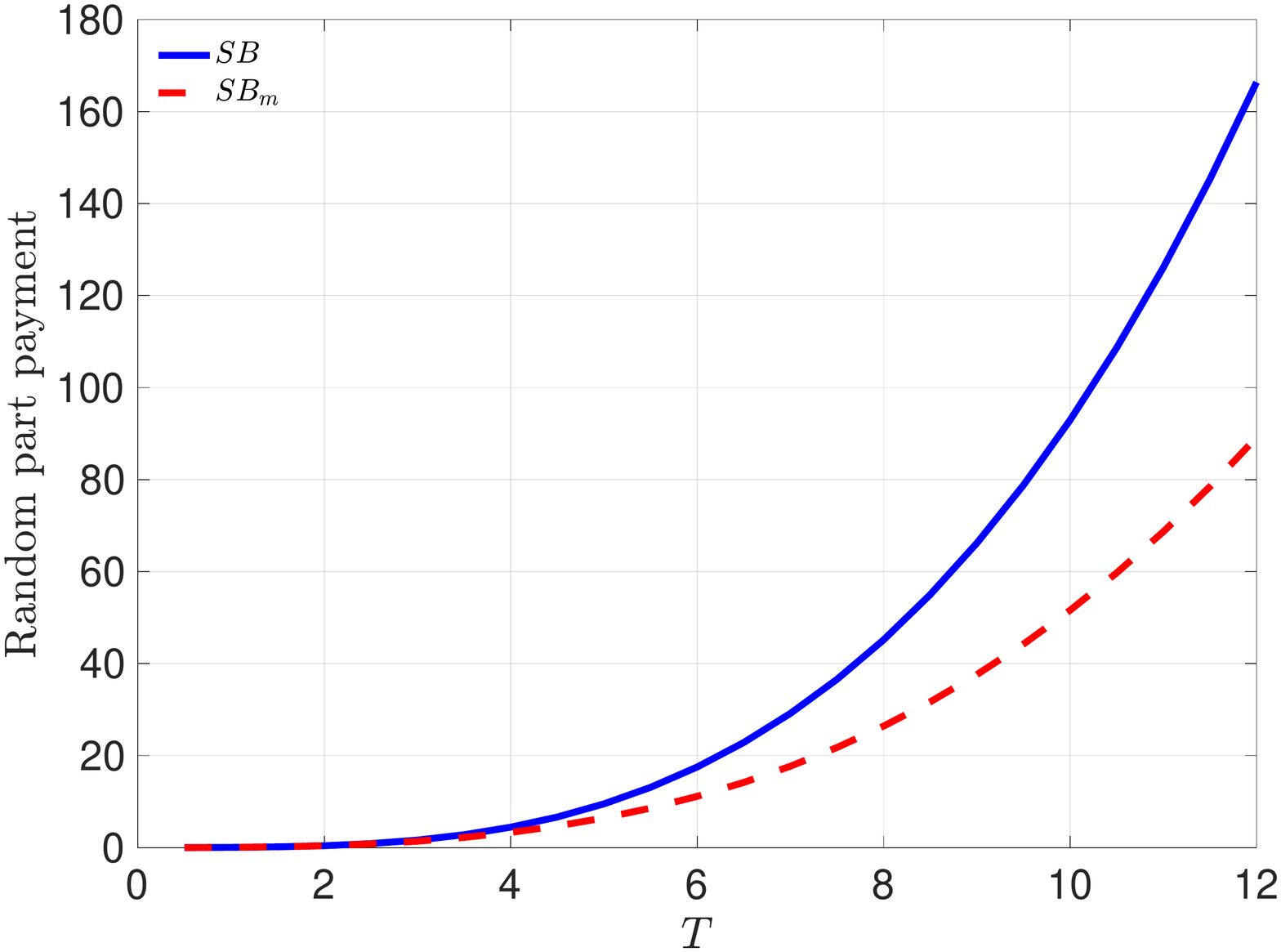}
\caption{Total (left), fixed part (middle) and  certainty equivalent of the random part (right) of the optimal payment with responsiveness control (blue) and without (red) as a function of the price event duration $T$ in pence.}
        \label{fig:payment}
        \end{center}
\end{figure}

We conclude these numerical illustration by showing the decomposition of the payment to the consumer between the fixed part and the random part as defined in Proposition~\ref{prop:sblinear} as a function of the duration of the price event. Indeed, we have seen that in the second--best with responsiveness incentives the prices of energy and volatilities are not constant. The random payment increases thus more rapidly than the payment without responsiveness incentive. Figure~\ref{fig:payment} shows the total payment and its decomposition between its fixed part and the certainty equivalent of the random part as a function of the duration of the price event $T$ with and without responsiveness control. In both cases, the total payment is positive and increases when the duration of the effort becomes large. As expected, the payment with responsiveness control is larger than without it because it requires more efforts from the consumer. The remarkable result comes from the decomposition of the contract between its deterministic part and its random part. The producer charges more the consumer when implementing responsiveness incentive than without, but provides higher certainty equivalent. The implementation of sound response from the consumer starts by charging him a lot more but also by rewarding him a lot more in case of appropriate result. The longer the consumer is asked to make an effort, the higher this difference should be. We cannot resist the temptation of making this result a general principle: if one wants to induce regular results from an agent on a long term basis, one should first reduce his income compared to his peers but then, pay him much more in case of success.

\subsection{Robustness analysis} 
\label{ssec:robust}

We check in this section the robustness of the hypothesis of a linear value of energy. For the sake of simplicity, we concentrate on the effect of a decreasing marginal value of energy for the consumer and leave aside its counterpart on the generation side (increasing marginal cost of generation). We consider now the following specification of the function $f$:
 \begin{align}
 f(x) 
 = 
 \kappa \frac{ 1- e^{-k_1 x}}{k_1},
 \label{eq:fnonlin}
 \end{align}
so that, for small values of $\kappa_1$, we recover the linear case with $f(x) \approx \kappa x$.

Further, we calibrated our model for a single average usage. But, in the case of a single usage, the producer can identify the effort on the responsiveness and thus, be close to the first--best. This is no longer the case when there are more usages. Thus, we assess also the mean payment and the benefit of the contract for the producer in the context of two and four usages. In this case we split the parameters $\mu$, $\lambda$ and $\sigma$ of the nominal situation provided by Table~\ref{tab:params} with the vector of weight $(1/4 \; \;3/4)$ for the two--usage case and $(1/8 \;\; 1/8 \;\; 1/2 \;\; 1/4)$ for the four--usage case. The choice of the vector of weight is guided by the idea of making a contrasted difference between usages.

We compute the numerical solution of the PDE of the second--best optimal contract with responsiveness incentive and non--linear energy value and generation given in Proposition~ \ref{prop:sbnonlinear}~(i). The PDE was solved using a standard finite difference method together with an implicit--Euler scheme. We compare in Figure~\ref{fig:sensitivityG} the resulting producer's certainty equivalent benefit with the one obtained when sending to the consumer the second--best contract with the linear approximation of the energy value function~\eqref{eq:fnonlin} and given by Proposition~\ref{prop:sblinear}~(i). In both cases, the initial condition of the contract is given by the reservation utility of the consumer given by Proposition~ \ref{prop:cru} with $f$ being given by relation \eqref{eq:fnonlin}. Thus, Figure~\ref{fig:sensitivityG}  just measures the benefit loss issued from the linear approximation of the energy value function. Without surprise, we oberve that the more the concavity of the energy value function, the more the linear approximation induces a loss of benefit for the producer. In the case of one usage only, a five fold multiplication of the concavity of $f$ leads to a loss of 1~pence out of 4. The introduction of more usages has two effects: a general decrease of benefit independent of the linearisation and an effect induced by this approximation. The linear approximation of concavity can reduce by half the benefit of the producer.

\begin{figure}[th!]
\begin{center}
\includegraphics[width=0.33\textwidth]{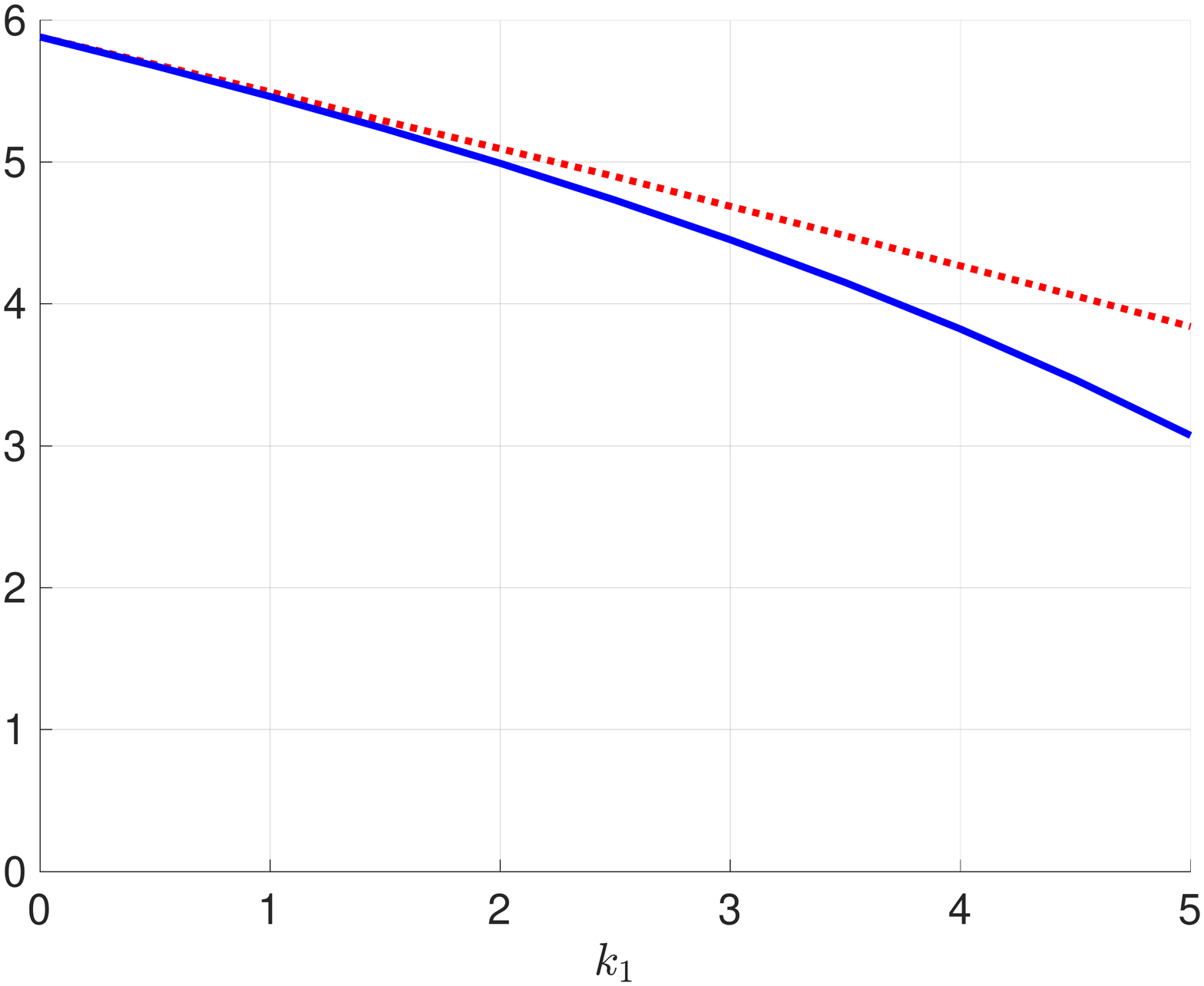}\includegraphics[width=0.33\textwidth]{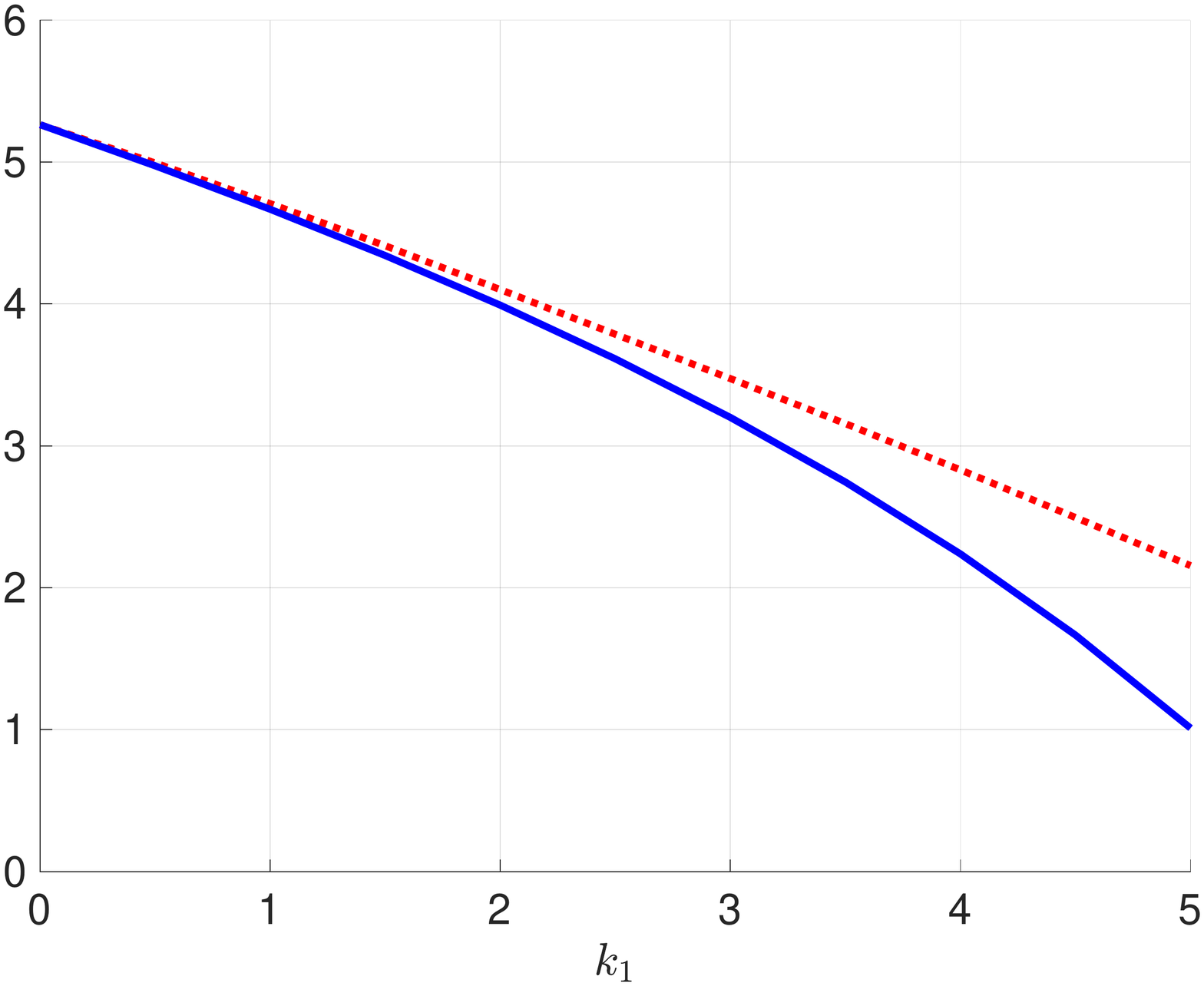}
\includegraphics[width=0.33\textwidth]{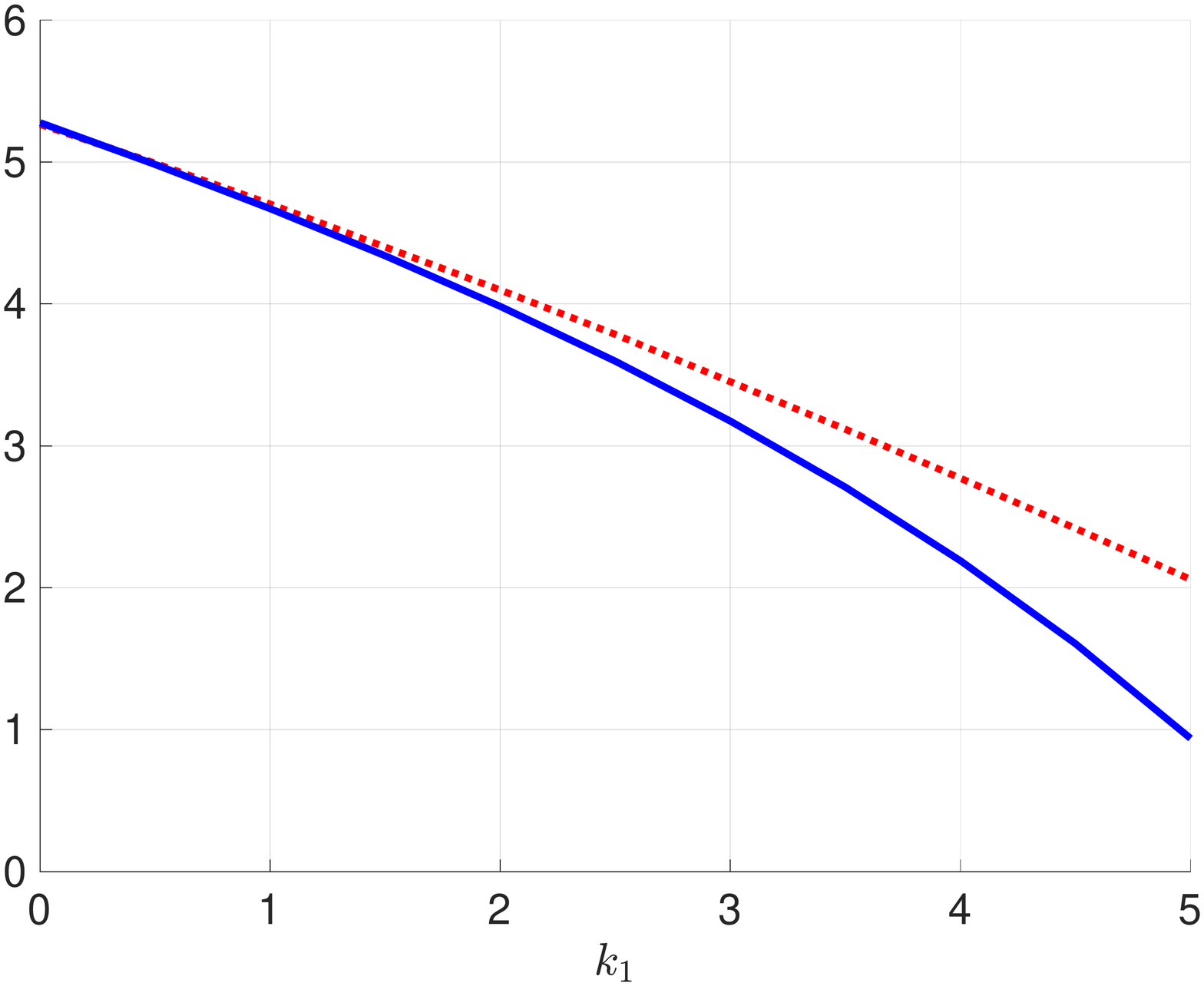}
\caption{Producer certainty equivalent benefit with one (Left), two (Middle) and four (Right) usages  with the linear approximation of $f$ (blue continuous line) and with the non-linear $f$ (red dotted line).}
        \label{fig:sensitivityG}
        \end{center}
\end{figure}

Figure~\ref{fig:sensitivityS} shows the total volatilities oberved under the second--best optimal contract with the linear approximation of $f$ compared to its true value. We observe that convavity increases the gap between the reduction of volatility that could be obtained with the nonlinear energy value function and its linear approximation. Nevertheless, the second--best contract with the linear approximation of $f$ still succeeds in achieving a significant decrease of the volatility before contracting, even with an increasing number of usages. 

\begin{figure}[th!]
\begin{center}
\includegraphics[width=0.33\textwidth]{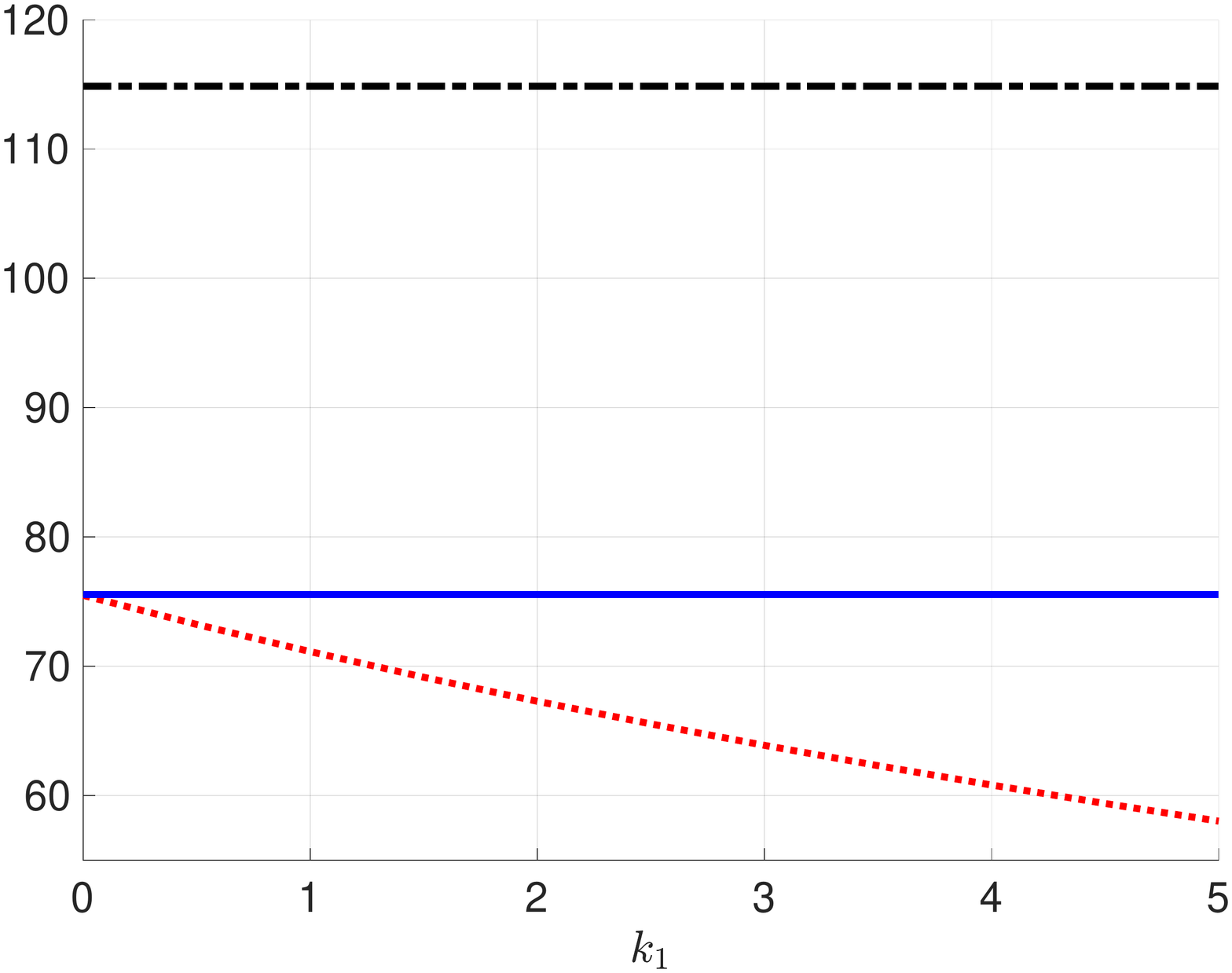}\includegraphics[width=0.33\textwidth]{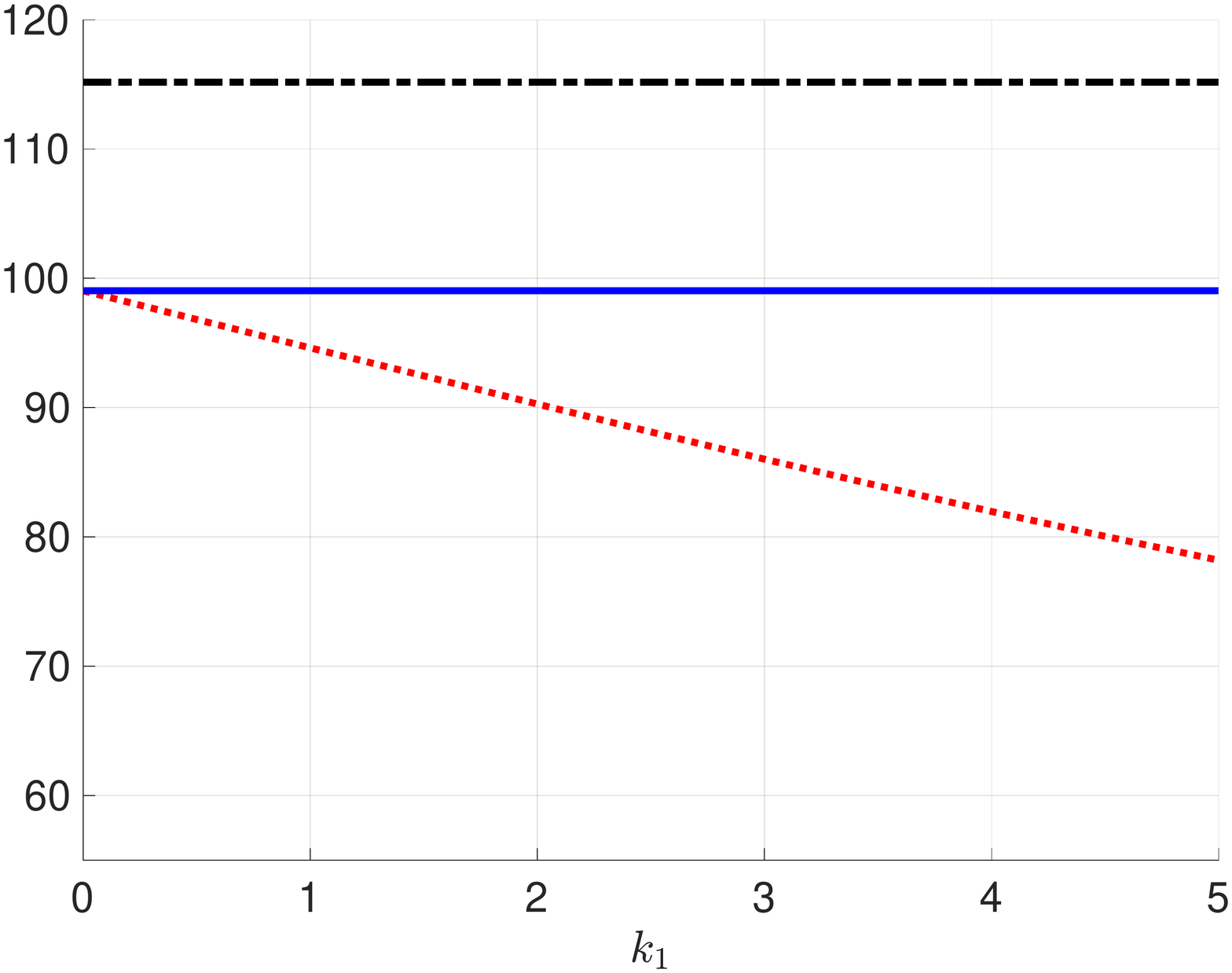}
\includegraphics[width=0.33\textwidth]{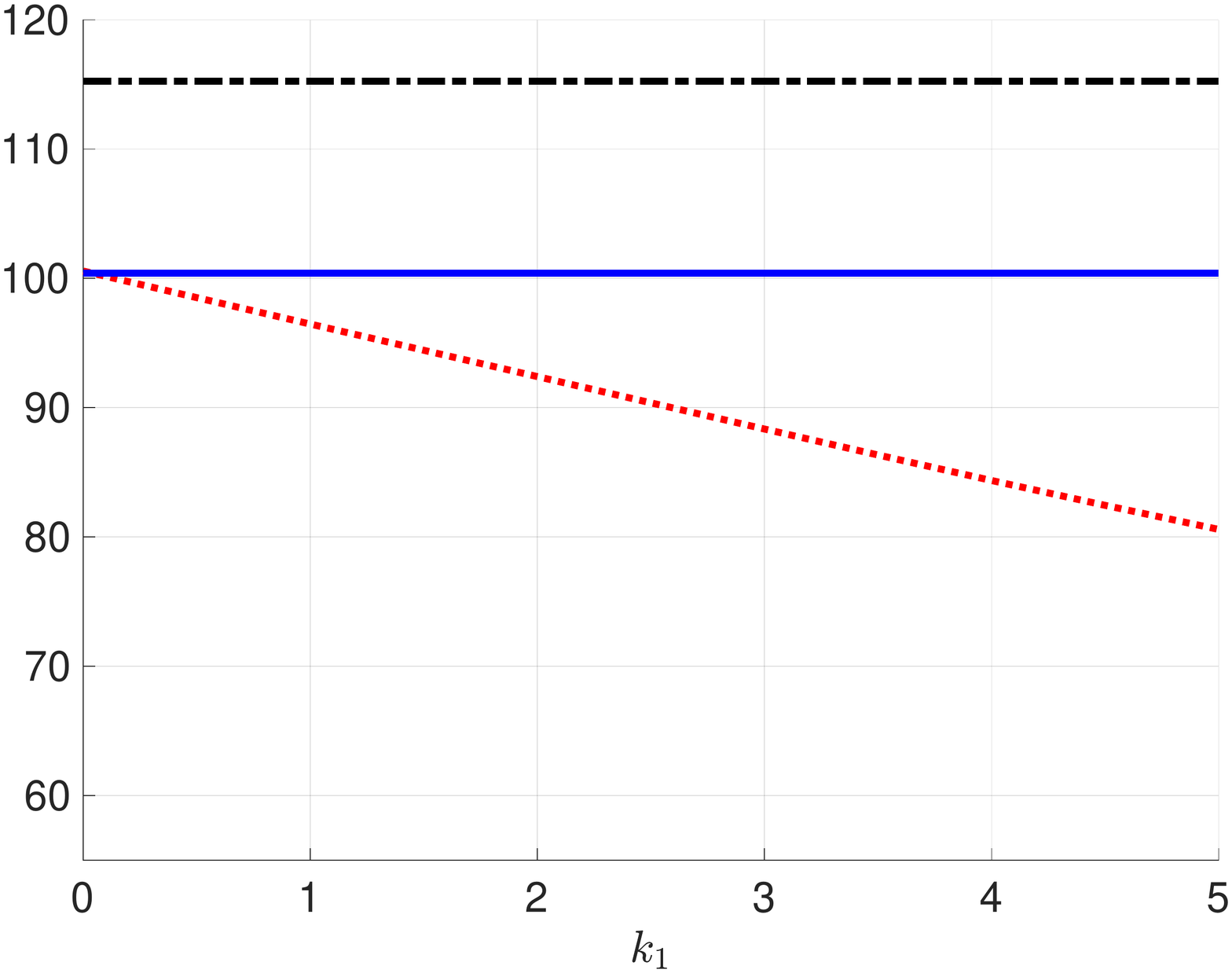} 
\caption{Volatilities (Watt) with one (Left), two (Middle) and four (Right) usages with the linear approximation of $f$ (blue dots) and with the non-linear $f$ (red stars). Black dotted line gives the volatility of the consumption without contract.}
        \label{fig:sensitivityS}
        \end{center}
\end{figure}

Figure~\ref{fig:sensitivityGR} presents the certainty equivalent benefit of the producer in the case of the implementation of the linear contract with responsiveness and without responsiveness as a function of the concavity of the energy value function of the consumer. Thus, this figure gives the loss induced not by the linearisation of the contract, but by the absence of responsiveness incentives in the linearisation of the energy value function. First, taking into account multiple usages reduces the difference of effects between the two contracts because of a degradation of the performance of the contract with responsiveness incentives. Second, as the concavity increases, the difference of benefits increase making the implementation of a responsiveness incentive mechanism more profitable.

\begin{figure}[th!]
\begin{center}
\includegraphics[width=0.33\textwidth]{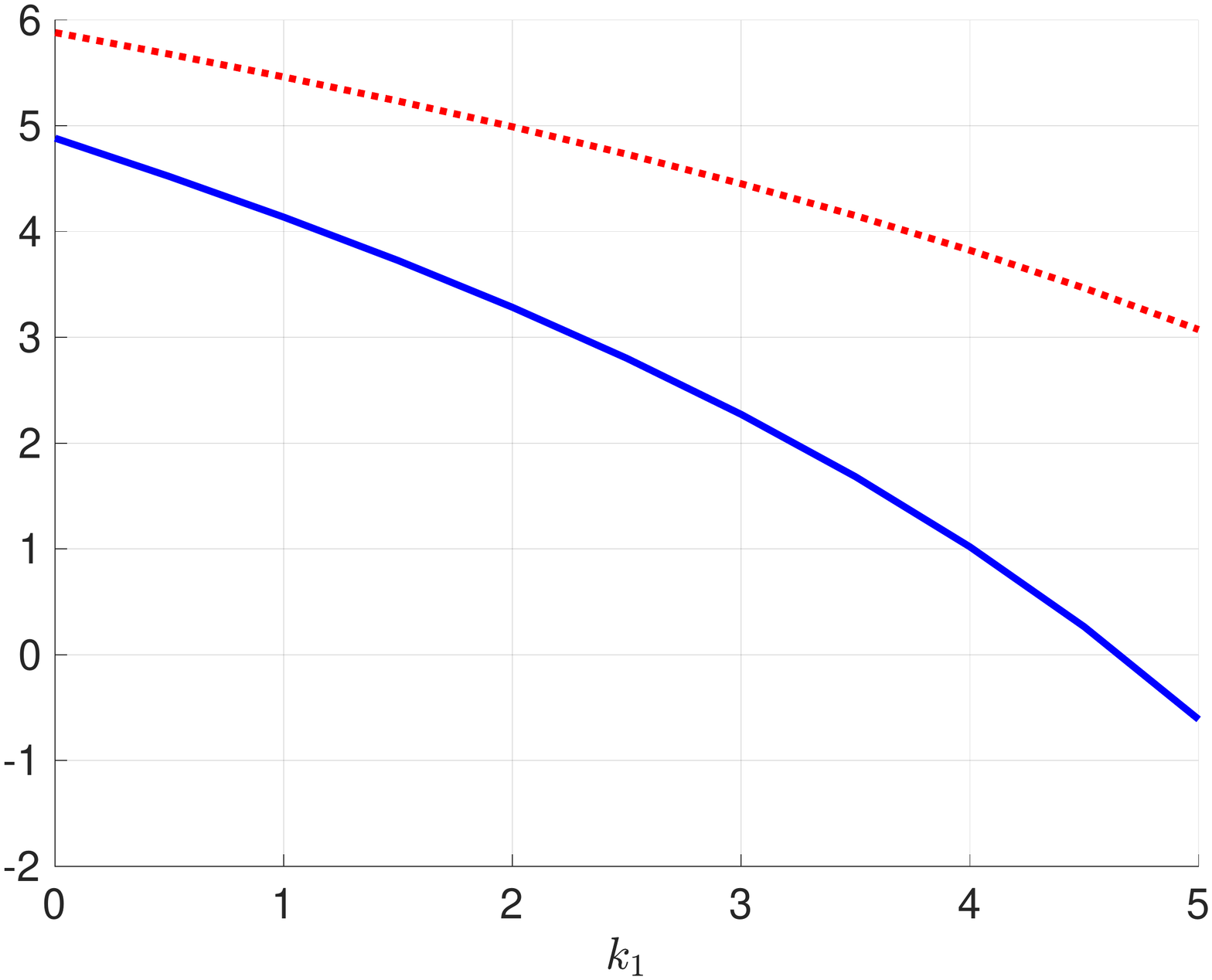}\includegraphics[width=0.33\textwidth]{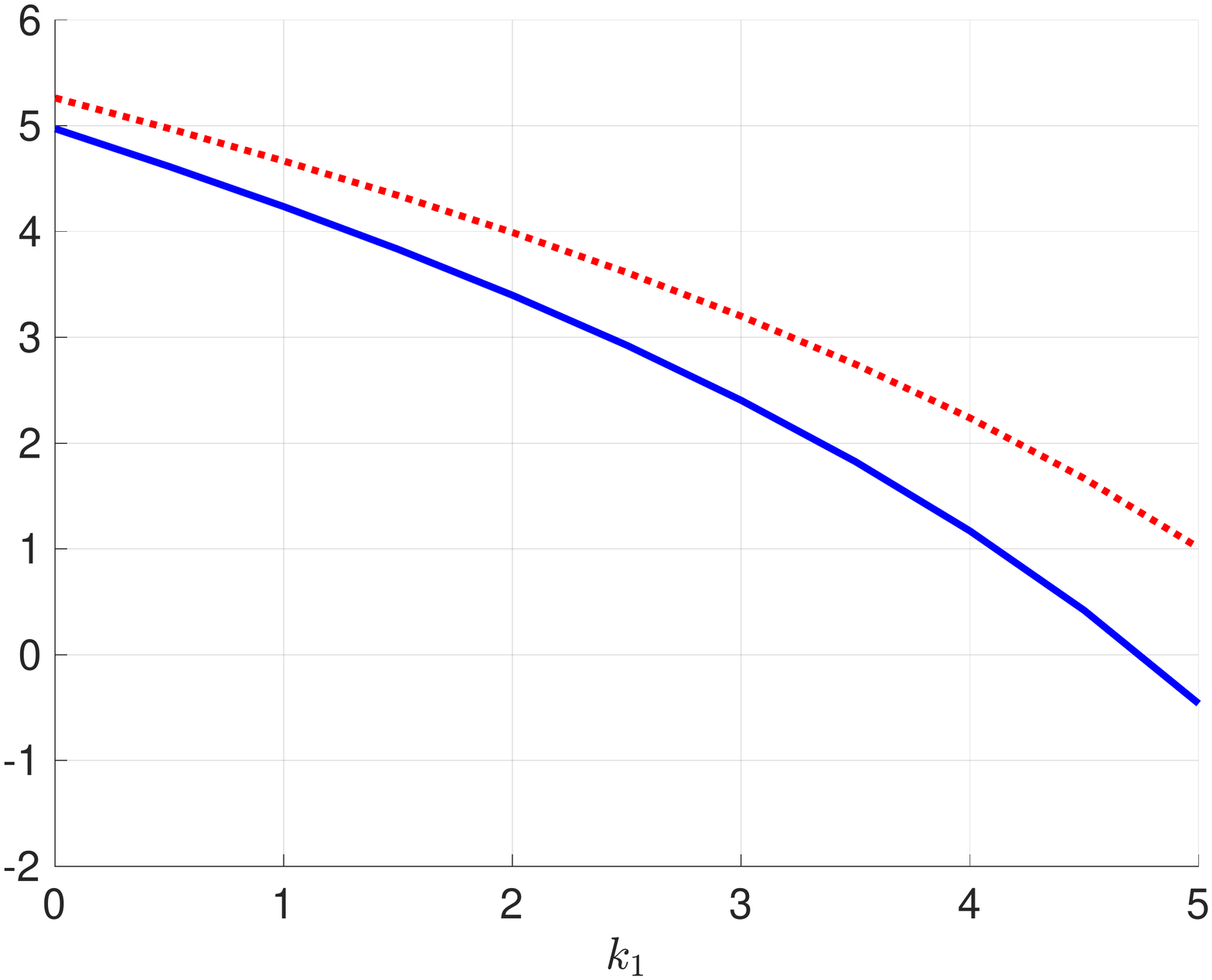}
\includegraphics[width=0.33\textwidth]{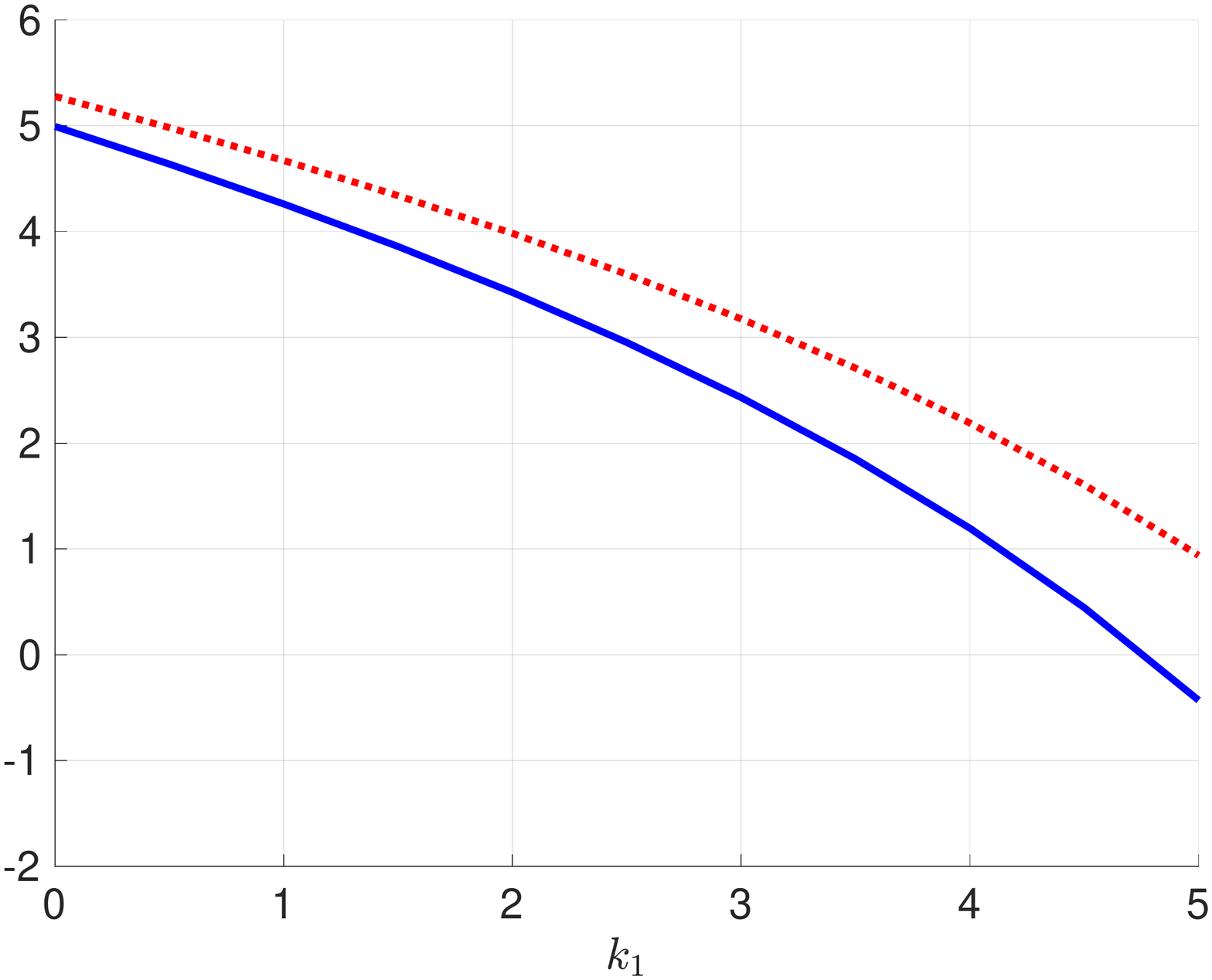} 
\caption{Certainty equivalent of the producer's benefit with the linear approximation contract with one (Left), two (Middle) and four (Right) usages with responsiveness incentive (red dotted line) and without responsiveness incentive (blue continuous line).}
        \label{fig:sensitivityGR}
        \end{center}
\end{figure}

\subsection{Practical issue}
\label{ssec:practice}

We have seen in the preceeding section that the optimal contract to induce an increase in the responsiveness of consumers to price signal should be written on the quadratic variation of the consumption. In discrete time $0 = t_0$ $< t_1$ $\cdots$ $t_n = T$, the quadratic variation $\langle X \rangle_T$ can be approximated by
$$ \langle X \rangle_T \approx \sum_{i=0}^{n-1} \big( X_{t_{i+1}} - X_{t_{i}} \big)^2.$$

Regarding energy consumption, households make a clear connection between the fact that they reduce the heating of the house and the reduction of their consumption. But, it is not obvious that domesic consumers first would understand why they should be charged a price proportional to that quantity and second, how the different actions they take during the day are related to it. The genuine form of the contract might not be acceptable for consumers. It is necessary to make simplifications to increase its potential acceptability.

A way one could think of is to index on the event--per--event variations of the average consumption. On event $k$, the producer measure the quantity $\bar X_k :=   \int_0^T X_t \drm t$ and charges the consumer a cost proportional to the quantity $\mathbb V_k :=    | \bar X_k - X_{\rm c} |$ where $X_{\rm c}$ is a contractualised targeted consumption in the spirit of Chao (2011) \cite{Chao11}. This contract is clearly sub--optimal: we lose the fact that the incentive price for responsiveness should be higher in the beginning of the price event, we loose the fact that the contract should be written on the quadratic variation, and not on the $L^1$--norm and so on. But, we gain a simple way to provide incentives to the consumer to remain as close as possible to a given pattern of consumption. The quantity $\mathbb V_k$ is measured in kWh and can be understood to be charged at a price measured in \euro/kWh just as the energy. 


\section{Conclusion}
\label{sec:conclusion}

We presented in this paper a new point of view on the demand response contract using the moral hazard problem in Principal--Agent framework, and showed how it makes it possible to reduce the average consumption, while improving the responsiveness of the consumer. We provided a closed--form expression for the optimal contract in the case of constant marginal cost and value of energy. We showed that the optimal contract has a rebate form and that the prices for energy and volatility differs from their marginal cost or value. We also showed how the optimal contract allows the system to bear more risk.  The calibration of our model to pricing trial data predicts that the cost of efforts of the consumer to reduce his average consumption will lead to significant benefits for producers and significant increase in the responsiveness of consumers and thus, on the efficiency of demand response programs. These predictions are testable. If our claim is true, the indexing of the payment to consumer on their regularity of consumption across price events should deeply enhance the efficiency of demand response programs. We do call for experiments to test this prediction.

\appendix

\section{Technical proofs}
\subsection{Proof of Proposition \ref{prop:crulinear} {\rm Consumer's behaviour without contract} }
\label{app:cru}

We provide first the following characterisation of the consumer's reservation utility.

\begin{Proposition} \label{prop:cru}
Assume that $f$ is concave, non--decreasing, and Lipschitz. Then the following holds.

\medskip
{\rm (i)} The consumer reservation utility is concave in $X_0$, and is given by $R_0=-e^{-r E(0,X_0)}$, where the corresponding 
certainty equivalent $E$ is a viscosity solution of the {\rm HJB} equation
 \begin{equation}\label{eq:PDE}
 \partial_t E 
 + H_{\rm v}\big(E_{xx}-rE_x^2\big) 
 + f 
 = 0~\mbox{on}~ [0,T)\times\R,
~\mbox{and}~ 
 E(T,x) = 0,\; x\in\R,
 \end{equation}
with growth controlled by $|E(t,x)|\le C(T-t)|x|$, for some constant $C$.

\medskip
{\rm (ii)} Assume that the PDE \eqref{eq:PDE} has a $C^{1,2}$ solution $E$ with growth controlled by $|E(t,x)|\le C(T-t)|x|$, for 
some constant $C$; then the optimal effort of the consumer is defined by the feedback controls
\[
 a^0 := 0
 \mbox{ and } 
 b^0_j := 1 \wedge \big(\lambda_j (E_{xx}-rE_x^2)\big)^{-\frac12 },\; j=1,\dots,d.
\]
\end{Proposition}

\begin{proof} 
$(i)$ Since $f$ is increasing, the consumer has no reason to make an effort on the drift of consumption deviation, as no compensation 
is offered for this costly effort. However, this argument does not apply to the effort on the volatility, due to the consumer's risk 
aversion. As the consumer has constant risk aversion utility with parameter $r$, her reservation utility reduces to
 \be\label{eq:R2}
 R_0
 := 
 \sup_{\P^{(0,\beta)}\in\Pc}  
 \E^{\mathbb \P^{(0,\beta)} }\Big[ -e^{-r \int_0^T ( f(X_t) - \frac12 c_2(\beta_t) ) \drm t}
                                             \Big] .
 \ee
The concavity of $R_0$ in the initial data $X_0$ is a direct consequence of the concavity of $f$ and the convexity of $c_2$. By 
standard stochastic control theory, it follows that $R_0=R(0,X_0)$, where the function $R$ is the dynamic version of the 
reservation utility, with final value $R(T,x)=-1$, and is a viscosity solution of the corresponding HJB equation:
$$
 0
 =
 \partial_tR - rfR
 +\sup_{b\in(0,1]^d} \frac12\big( |\sigma(b)|^2 R_{xx}+rc_2(b)R\big)
 =
 \partial_tR - rf R
 -rR\,H_{\rm v}\Big(\frac{R_{xx}}{-rR}\Big).
$$
Denote by $X^{t,x}_s:=x+X_{s-t}$ the shifted canonical process started from initial data $(t,x)$. As $f$ is Lipschitz, notice that 
\begin{align*}
 R(t,x)
 \ge
 \E^{\P^{0,1}}\Big[ -e^{-r\int_t^T f(X^{t,x}_s) \drm s} \Big]
 &\ge
 -\E^{\P^{0,1}}\Big[ e^{r |f'|_\infty \int_t^T |X^{t,x}_s| \drm s} \Big]
 \\
 &\ge
 -\E^{\P^{0,1}}\Big[ e^{r |f'|_\infty ((T-t)|x|+\int_t^T |X^{t,0}_s| \drm t)} \Big]
 \ge
 -C_1e^{r |f'|_\infty (T-t)|x|},
 \end{align*}
where 
$$C_1:=\E^{\P^{0,1}}\Big[ e^{r |f'|_\infty \int_0^T |X^{t,0}_s| \drm s)} \Big]<\infty,$$ 
since $X^{t,0}_s$ is a centred Gaussian random variable for all $s\in[t,T]$. As $c_2\ge 0$, we also have
\[
 R(t,x)
 \le
 \sup_{\P^{(0,\beta)}\in\Pc}
 \E^{\P^{0,\beta}}\Big[ -e^{-r\int_t^T f(X^{t,x}_s)\drm s} \Big]
 =
 \E^{\P^{0,1}}\Big[ -e^{-r\int_t^T f(X^{t,x}_s)\drm s} \Big]
  \le
 -C_2e^{-r |f'|_\infty (T-t)|x|},
\]
 where 
 $$C_2:=\E^{\P^{0,1}}\Big[ e^{-r |f'|_\infty \int_0^T |X^{t,0}_s| \drm s)} \Big]<\infty,$$ 
 by the same argument as previously.
 
Then, the certainty equivalent function $E$, defined by $R=:-e^{-rE}$, satisfies the PDE \eqref{eq:PDE}, and has growth controlled 
by $E(t,x)\le (C_1\vee C_2)(T-t)|x]$. 

$(ii)$ We now assume that the PDE \eqref{eq:PDE} has a $C^{1,2}$ solution $E$ with growth controlled by $|E(t,x)|\le C(T-t)|x|$. 
Then $\hat R:=-e^{-rE}$ is also $C^{1,2}([0,T]\times\R)$. Denote $K^\beta_t:=e^{-r \int_0^t ( f(X_s) - \frac12 c_2(\beta_s) ) \drm s}$, 
and $T_n:=\inf\{t>0:|X_t-X_0|\ge n\}$, we compute by It\^o's formula that for all $\P^{(0,\beta)}\in\Pc$,
\begin{align*}
 \hat R(0,X_0)
 =
 \E^{\P^{0,\beta}}
 \bigg[K^\beta_{T_n}\hat R(T_n,X_{T_n}) 
        -\int_0^{T_n} K^\beta_t
                               \Big(\partial_t \hat R 
                                      + \frac12|\sigma(\beta_t)|^2v_{xx}
                                      - r\Big(f-\frac12 c_2(\beta_t)\Big) \hat R
                               \Big)(t,X_t)\drm t 
 \bigg]
 \\
 \ge
 \E^{\P^{0,\beta}}
 \Big[K^\beta_{T_n}\hat R(T_n,X_{T_n})\Big]
 \;\longrightarrow\;
 \E^{\P^{0,\beta}}
 \Big[K^\beta_{T}\hat R(T,X_T)\Big] 
 = 
 \E^{\P^{0,\beta}}
 \Big[-K^\beta_{T}\Big],
\end{align*}
where the local martingale part verifies that $\E^{\P^{0,\beta}}\big[\int_0^{T_n} K^\beta_t\hat R_x(t,X_t)
\sigma(\beta_t)\drm W_t\big]=0$, by the fact that $\hat R_x$ is bounded on $[0,T_n]$ and $\sigma(\beta)$ is bounded. The second 
inequality follows from the PDE satisfied by $\hat R$, and the last limit is obtained by using the control on the growth of $R$ 
together with the final condition $\hat R(T,.)=-1$. By the arbitrariness of $\P^{(0,\beta)}\in\Pc$, this implies that $\hat R(0,X_0)\ge 
R_0$.

To prove equality, we now observe from Proposition \ref{prop:agentresponse} that by choosing $b^0(t,x):=\widehat b(E_{xx}-
E_x^2)(t,x)$ as defined in Proposition \ref{prop:agentresponse}, the (unique) inequality in the previous calculation is turned into an 
equality provided that the stochastic differential equation $dX_t=\sigma(b(t,X_t)) \drm W_t$ has a weak solution. This, in turn, is 
implied by the fact that the function $\sigma\circ b$ is bounded and continuous, see Karatzas and Shreve \cite[Theorem 5.4.22 
and Remark 5.4.23]{karatzas1991brownian}. Consequently, $\hat R(0,X_0)=R_0$, and $(a^0,b^0)$ are optimal feedback controls.
\qed
\end{proof}

We now prove the proposition. 

\begin{Proposition} {\rm (Consumer's behaviour without contract)}
Let $f(x) = \kappa \, x,$ $x\in\R$, for some $\kappa \geq 0$. Then,  $V_A(0) = U_A(\kappa X_0 T + E_0(T))$ where   
\[
 E_0(T) := \int_0^T H_{\rm v}\big(-\gamma(t)\big)  \drm t,
 \mbox{ and }
 \gamma(t) :=  -r \kappa^2 (T-t)^2.
 \]
The consumer's optimal effort on the drift and on each volatility usage are respectively 
 \[
 a^0=0, 
 \text{ and } 
 b_j^0(t) := \varepsilon \vee (1 \wedge \big( \lambda_j |\gamma(t)| \big)^{-\frac{1}{2}}),\; j=1,\dots,d,
 \]
thus inducing an optimal distribution $\P^0$ under which the deviation process follows the dynamics $\mathrm{d}X_t=\widehat{\sigma}
\big(b^0(t)\big)\cdot \mathrm{d}W_t,$ for some $\P^0-$Brownian motion $W$.
\end{Proposition}

\proof
The value $V^A(0)$ is concave in $X_0$, and is given by $V_A(0)=-\mathrm{e}^{-r E(0,X_0)}$, where the corresponding  certainty equivalent $E$ is a viscosity solution of the {\rm HJB} equation
\[
- \partial_t E 
 = f 
 + H_{\rm v}\big(E_{xx}-rE_x^2\big) 
 ~\mbox{on}~ [0,T)\times\R,
~\mbox{and}~ 
 E(T,x) = 0,\; x\in\R.
\]
By directly plugging the guess $E(t,x)=C(t)x+E_0(t)$ in the PDE \eqref{eq:PDE}, we obtain
 \[
 C^\prime(t) x + E^\prime_0(t) + H_{\rm v}\big(-rC^2(t)\big)+\kappa x=0,
 \mbox{ with }
 C(T)=E_0(T)=0.
 \]
This entails $C(t)=\kappa(T-t)$ and $E_0(t)=\int_t^T H_{\rm v}\big(-rC^2(s)\big)\mathrm{d}s$, $0\leq t\le T$. Finally the expression of the 
maximiser $b^0$ follows from Proposition \ref{prop:agentresponse}. Since this smooth solution of the PDE has the appropriate 
linear growth, we conclude from Proposition \ref{prop:cru} $(ii)$ that it is indeed the value function inducing $V_A(0)$.\ep

\subsection{Proof of Proposition \ref{prop:fblinear} {\rm First--best contract}}
\label{app:fblinear}

We provide here the solution of the first--best problem \eqref{eq:FB} corresponding to the case where the producer chooses both the 
contract $\xi$ and the level of effort of the consumer $\nu$, under the constraint that the consumer's satisfaction is above the 
reservation utility. Introducing a Lagrange multiplier $\ell\ge 0$ to penalise the participation constraint, and applying the classical 
Karush--Kuhn--Tucker method, we can formulate the producer's first--best problem as
 \begin{equation}
 \label{FB-Lagrange}
\Vfb
 =
 \inf_{\ell\ge 0} \bigg\{
 -\ell R
 +\sup_{(\xi,\P^\nu)} \E^{\P^\nu} \big[ U \big(- \xi - \mathcal K^\nu_T \big) 
                                                            + \ell U_{\rm A}\big( \xi+ G^\nu_T \big)
                       \big] \bigg\},
 \end{equation}
where $\Gc^\nu_T :=  \int_0^T g(X_s) \mathrm{d}s +  \frac{h}{2} \langle X\rangle_T$ and $\Kc^\nu_T :=  \int_0^T(f(X_s)-c(\nu_s))\mathrm{d}s$. The 
first--order conditions in $\xi$ are
\[
- U'\big(-\xi_\ell - \Gc^\nu_T\big) 
+ \ell U_{\rm A}'\big(\xi_\ell + \Kc^\nu_T\big) = 0.\]
In view of our specification of the utility functions, this provides the optimal 
contract payment for a given Lagrange multiplier $\ell$
 \begin{equation} \label{ximu}
 \xi_\ell
 :=
 \frac{1}{p+r} \ln{\bigg(\frac{r\ell}{p}\bigg)}
 - \frac{p}{p+r} \Gc^\nu_T - \frac{r}{p+r} \Kc^\nu_T.
 \end{equation}
Substituting this expression in \eqref{FB-Lagrange}, we see that the Principal's first--best problem reduces to
 \[
\Vfb
 =
 \inf_{\ell\ge 0}\bigg\{ \ell\bigg(-  R
                                       + \bigg(1+\frac{r}{p}\bigg)  \bigg(\frac{r\ell}{p}\bigg)^{\frac{-r}{r+p}} \bar V
                               \bigg)\bigg\},
 \mbox{ with }
 \bar V
  :=
  \sup_{\P^\nu} \E^{\P^\nu} \Big[ -\mathrm{e}^{-\rho\big( \int_0^T ((f-g)(X_t)-c(\nu_t)) \mathrm{d}t - \frac{h}{2} \langle X\rangle_T \big)}
                                           \Big],
 \]
and $\frac{1}{\rho}:=\frac{1}{r}+\frac{1}{p}$. Notice that $\bar V$ does not depend on the Lagrange multiplier $\ell$. Then direct 
calculations lead to the optimal Lagrange multiplier and first--best value function
 \begin{equation}\label{FB-barV}
 \ell^\star
 := 
 \frac{p}{r}\bigg(\frac{\bar V}{R}\bigg)^{1+\frac{p}{r}},
 \mbox{ so that }
 V^{\mbox{\footnotesize\sc fb}}
 =
 R
 \bigg(\frac{\bar V}{R}\bigg)^{1+\frac{p}{r}}.
 \end{equation}
This lead to the following proposition.

 \begin{Proposition} \label{prop:fb}
Assume that $f-g$ is Lipschitz continuous.Then

\vspace{0.5em}
\noindent {$(i)$} $\bar V =  - \mathrm{e}^{ - \rho \bar v(0,X_0)}$, where $\bar v$ has growth $|\bar v(t,x)|\le C(T-t)|x|$, for some constant $C>0$, and is a viscosity solution of 
the {\rm PDE}
 \begin{equation}
 \label{ed:vbar}
 -\partial_t \bar v
 =  (f-g)  + H_{\rm m}(\bar v_x) + H_{\rm v}\big( \bar v_{xx} - \rho \bar v_x^2 - h\big),
 ~\mbox{on}~[0,T)\times\R,~ \mbox{and}~
 \bar v(T,.)=0,
 \end{equation}
 so that, by \eqref{FB-barV}, the first--best value function $\Vfb= U( \bar v(0,X_0) -L_0 )$.

\vspace{0.5em}
\noindent {$(ii)$} If in addition $\bar v$ is smooth, the optimal efforts to induce a reduction of the consumption deviation and of its volatility 
are given by
 \begin{equation}\label{eq:zgf}
  \afb(t,X_t) := \widehat a\big( \zfb(t,X_t)\big),
 ~\mbox{and}~
 \bfb(t,X_t) := \widehat b\big( \gfb(t,X_t)\big),\; t\in[0,T],
\end{equation}
where
 \[\zfb(t,x) := \bar v_x(t,x),
 ~
 \gfb(t,x):= \bar v_{xx}(t,x) - \rho \bar v_x^2(t,x) - h,\; (t,x)\in[0,T]\times\R.
\]

 \noindent {$(iii)$} Denoting $\nu_{ {\mbox{\footnotesize\rm\sc fb}}}:=(a_{{\mbox{\footnotesize\rm\sc fb}}}, b_{{\mbox{\footnotesize\rm\sc fb}}})$, the optimal first--best contract can be written as
 \[
  \xifb
  =
  L_0  
  - \frac{p}{p+r}  \bar{v}(0,X_0)  
  - \frac{r}{p+r} \int_0^T \big(f(X_t)-c(\nu_{{\mbox{\footnotesize\rm\sc fb}}}(t,X_t))\big)\mathrm{d}t
  - \frac{p}{p+r} \int_0^T  g(X_t)\mathrm{d}t  + \frac{h}{2} \langle X\rangle_T. 
  \]
\end{Proposition}

\begin{proof} By standard stochastic control theory, $\bar V$ can be characterised by means of the corresponding HJB equation
$$\begin{cases}
  \displaystyle \partial_t \bar V
  -\rho(f-g) \bar V
  +\sup_{a\in\R_+^d}\big\{ -a\cdot \1 \bar V_x+\rho c_1(a) \bar V \big\}
  +\frac12 \sup_{b\in[0,1)^d} \big\{ |\sigma(b)|^2(\bar V_{xx}+\rho h \bar V)+\rho c_2(b) \big\} 
  =
  0
  \\
  
  \displaystyle  \bar V(T,.) = -1.
  \end{cases}$$
Setting $\bar V(t,x) =  - e^{ - \rho \bar v(t,x)}$, we obtain by direct substitution the PDE satisfied by $\bar v$
 \begin{align}
 \label{ed:vbar2}
 \begin{cases}
\displaystyle  -\partial_t \bar v
 =  (f-g)  - \inf_{a} \big\{ a \cdot \1 \bar v_x + c_1(a)  \big\}  - \frac12 \inf_{b} \big\{ c_2(b)- |\sigma(b)|^2 \big( \bar v_{xx} - \rho \bar 
v_x^2 - h\big)  \big\}
\\
\displaystyle   \bar v(T,x)=0, 
\end{cases}                  
 \end{align}
which coincides with the PDE in the proposition statement, by definition of the consumer's Hamiltonian \eqref{eq:agentH}. We next 
prove the control on the growth of $v$. First, as the cost function $c$ is non-negative, we have
\[
 \bar V 
 \le 
 \sup_{\P^\nu} \E^{\P^\nu} \Big[ -e^{-\rho( \int_0^T (f-g)(X_t)\drm t)}
                                           \Big]
 \le 
 \sup_{\P^\nu} \E^{\P^\nu} \Big[ -e^{-\rho( (f-g)(X_0)+|f'-g'|_\infty\int_0^T |X_t|\drm t)}
                                           \Big]
 <-\infty.
\]
On the other hand, as the cost of no-effort $c(0,1)=0$, it follows that
\[
 \bar V 
 \ge 
 \E^{\P^{0,1}} \Big[ -e^{-\rho( \int_0^T (f-g)(X_t)\drm t)}
                                           \Big]
 \ge 
 \sup_{\P^\nu} \E^{\P^\nu} \Big[ -e^{-\rho( (f-g)(X_0)-|f'-g'|_\infty\int_0^T |X_t|\drm t)}
                                           \Big]
 >-\infty.
\]
This shows that $e^{C_1T|X_0|}\le\bar V\le e^{C_2T|X_0|}$, for some constants $C_1,C_2>0$, and therefore $|E(0,X_0)|\le 
(C_1\vee C_2)|x| T$. By homogeneity of the problem, we deduce the announced control on growth by simply removing the time 
origin to any $t<T$.

\vspace{0.5em}
Under smoothness assumptions, we follow the line of the verification argument of the proof of Proposition 
\ref{prop:cru} to prove that the optimal consumer's response derived in 
Proposition~\ref{prop:agentresponse} is an optimal feedback control for the problem $\bar V$. Using the fact that $\Vfb = R_0 
\big(\frac{1}{R_0} e^{- \rho (\bar v(0, X_0))} \big)^{\frac{p+r}{r}}$ and $L_0 = -\frac{1}{r} \log(-R_0)$, one gets 
$$\Vfb = -e^{-p (\bar v(0,X_0) - L_0)}.$$
Finally, the expression of $\xifb$ follows by direct substitution of the optimal Lagrange multiplier \eqref{FB-barV} in \eqref{ximu}.
\qed
\end{proof}

\hspace{5mm}

In the case when $(f-g) x = \delta x$, we make the guess that $\bar v(t,x) = \delta (T-t) x + \int_0^t \bar m(s) \drm s$ satisfies the PDE~\eqref{eq:vbar} where 
\[
\bar m(t) = \Hm(\delta(T-t)) + \Hv(-h - \rho \delta^2 (T-t)^2).
\]
The value function $\bar v$ satisfies the hypothesis of Proposition~\ref{prop:fb} and the PDE~\eqref{eq:vbar}. Rearranging the terms in the expression of the first--best contract~(iii) of Proposition~\ref{prop:fb} leads to the form of the contract in Proposition~\ref{prop:fblinear}. 

\subsection{Proof of Proposition \ref{prop:sblinear} {\rm Second--best contract}}
\label{app:sblinear}

As the volatility induced by responsiveness effort is uniformly bounded above zero, and the level reduction effort is bounded, we may follow the general methodology of Cvitani\'c et al. (2018) \cite{cvitanic2018dynamic}, based on Sannikov (2008) 
\cite{Sannikov08}. Let $\mathcal V$ be the collection of all pair processes $(Z,\Gamma)$ and constants $y_0\in\R$, inducing the 
subclass of contracts $\xi=Y^{y_0,Z,\Gamma}_T$, where
 \begin{equation}\label{eq:Y}
 Y^{y_0,Z,\Gamma}_t
 :=
 y_0
 +\int_0^t Z_s \mathrm{d}X_s + \frac{1}{2}\int_0^t\big(\Gamma_s +r Z_s^2\big)\mathrm{d}\langle X\rangle_s -\int_0^t\big(H(Z_s,\Gamma_s) + f(X_s) 
\big) \mathrm{d}s,
 ~t\in[0,T].
 \end{equation}
We recall from Cvitani\'c et al. (2018) that $U_A\big(Y_t^{y_0,Z,\Gamma}\big)$ represents the Agent's continuation utility so that 
$Y_t^{y_0,Z,\Gamma}$ is the time $t$ value of the Agent's certainty equivalent. This contract is affine in the level of consumption 
deviation $X$ and the corresponding quadratic variation $\langle X\rangle$, with linearity coefficients $Z$ and $\Gamma$. The 
constant part $\int_0^T (H(Z_s,\Gamma_s) + f(X_s) )\mathrm{d}s$ represents the certainty equivalent of the utility gain of the consumer that 
can be achieved by an optimal response to the contract, and is thus subtracted from the Principal's payment, in agreement with 
usual Principal--Agent moral hazard type of contract (see \cite[Chapter 4]{Laffont02}). Further, in the present setting, the risk 
aversion of the Agent implies that the infinitesimal payment $Z_ t \mathrm{d}X_t$ must be compensated by the additional payment $\frac12 
r Z_t^2 \mathrm{d}\langle X \rangle_t$, so that, formally, the resulting impact of the payment  $Z_ t \mathrm{d}X_t$ on the Agent's expected utility is
 \[
 - \exp\bigg(- r \bigg( Z_t \mathrm{d}X_t + \frac12 r Z_t^2 \mathrm{d}\langle X \rangle_t \bigg)\bigg) +1 
 \sim
 -1 + r\bigg(  Z_t \mathrm{d}X_t + \frac12 r Z_t^2 \mathrm{d}\langle X \rangle_t \bigg) - \frac12 r^2 Z^2_t \mathrm{d}\langle X \rangle_t  + 1
 \;\sim\;
 r Z_t \mathrm{d}X_t.
 \]
Under the optimal response of the consumer, the dynamics of the consumption deviation and the certainty equivalent of the 
consumer are given by
 \begin{align*}
 X^{Z,\Gamma}_t
 :=&\
 \displaystyle
 X_0-\int_0^t\widehat a(Z_s)\cdot{\bf 1}\mathrm{d}s+\int_0^t\widehat\sigma(\Gamma_s)\cdot \mathrm{d}W_s
 \\
 Y^{Y_0,Z,\Gamma}_t
 =&\
 \displaystyle
 Y_0
 +\int_0^t\Big( c\big(\widehat a(Z_s), \widehat b(\Gamma_s)\big) - f(X^{Z,\Gamma}_s) + \frac12 r Z_s^2  | \widehat 
\sigma(\Gamma_s) |^2   \Big)\mathrm{d}s
 +\int_0^t Z_s\widehat\sigma(\Gamma_s)\cdot \mathrm{d}W_s,
\end{align*}
so that the average rate of payment consists in paying back the consumer her costs minus benefit $c-f$, and an additional compensation for the risk taken by the consumer for bearing the volatility of consumption deviation. Note that the average rate of payment can be positive or negative.

\begin{Remark} \label{rem:R0}
{\rm 
\noindent $(i)$ Assume that the producer proposes the contract $\xi^0$ defined by $y_0 = - 
\frac{1}{r} \ln (-R_0)$, $Z =\Gamma \equiv 0$, i.e.  $\xi^0 = - \frac{1}{r} \ln (-R_0) - \int_0^T f(X_t^\nu) \mathrm{d}t$, as the Hamiltonian 
satisfies here $H_{\rm v}(0) \equiv 0$. Then, the optimal response of the consumer is obtained by solving the utility maximisation 
problem 
\[V_A(\xi^0)
 =
 \sup_{\P^\nu\in\Pc} \E\Big[  U_{\rm A} \Big( -  \int_0^T  c(\nu_t) \mathrm{d}t\Big)  \Big].\] As the cost function $c$ is non--decreasing in the 
effort, at the optimum, the consumer makes no effort, neither on the drift, nor on the volatility. Comparing with Proposition \ref{prop:cru}, this shows that the absence of contract is different from the above 
contract $\xi^0$, with zero payment rates.

\vspace{0.5em}
\noindent $(ii)$ We may also examine the case where the producer offers a contract with payment $\xi=0$ to the consumer. This is achieved 
by choosing $Z_t=E_x(t,X_t)$ and $\Gamma_t=(E_{xx}-E_x^2)(t,X_t)$, where the certainty equivalent reservation $E$ is defined 
in Proposition \ref{prop:cru}. From the point of view of the consumer, such a contract is clearly equivalent to 
the no contracting setting, and thus induces a positive effort on the volatility in the consumer's optimal response.}
\end{Remark}

By the main result of Cvitani\'c et al. (2018), we may reduce the Principal's problem to the optimisation over the class of 
contracts $Y^{y_0,Z,\Gamma}_T$, where $y_0\ge L_0$ and $(Z,\Gamma)\in\Vc$. By the obvious monotonicity in 
$y_0$, this leads to the following standard stochastic control problem
\[
\Vsb
 =
 \sup_{Z,\Gamma}
 \E\big[U\big(-L_T^{Z,\Gamma}\big)\big],
 \mbox{ with }
 L^{Z,\Gamma}_t
 :=
 Y^{Z,\Gamma}_t
 +\int_0^t g\big(X^{Z,\Gamma}_s\big)\mathrm{d}s
                +\frac{h}{2}\mathrm{d}\langle X^{Z,\Gamma}\rangle_s,
~t\in[0,T],
\]
and starting point $y_0 = L_0$. The state variable $L$ represents the loss of the producer under the optimal 
response of the consumer, and is defined by the dynamics
\[
 \mathrm{d}L^{Z,\Gamma}_t
=
 \frac12\big( 2(g-f)(X^{Z,\Gamma}_t) 
                    + \widehat c_1(Z_t)  
                    + f_0(r Z_t^2+h,\Gamma_t)
            \big)\mathrm{d}t
 + Z_t\widehat\sigma(\Gamma_t)\cdot \mathrm{d}W_t,
 ~t\in[0,T).
\]
where
 \begin{equation}\label{F0f0}
 f_0(q,\gamma)
 :=
 q\big|\widehat\sigma(\gamma)\big|^2 + \widehat c_2(\gamma).
 \end{equation}
The function $f_0(q, \gamma)$ measures the total cost the producer incurs from the volatility, when the unit cost of volatility is $q$ and the rate of payment for the volatility reduction is $\gamma$. The term $q | \widehat \sigma(\gamma) |^2$ is the instantaneous cost of volatility while the term $
\widehat c_2(\gamma)$ is the cost of effort incurred by the consumer. This last cost will be paid by the producer, and enters thus in 
the evaluation of the cost of volatility. The producer aims at making the term $| \widehat \sigma(\gamma) |^2$ as small as possible. 
To achieve this objective, a sufficiently large $\gamma$ should be paid to reduce $| \widehat \sigma(\gamma) |^2$, but this can be 
done only at the expense of an increasing cost $\widehat c_2(\gamma)$. 

\begin{Lemma} \label{lem:F0}
Let $F_0(q): =\inf_{\gamma\le 0} f_0(q,\gamma)$. Then $F_0(q)=f_0(q,-q)=-2H_{\rm v}(-q)$ is non--decreasing.
\end{Lemma}

\noindent {\bf Proof}:  \noindent Recall that
$ f_0(q,\gamma) = q\big|\widehat\sigma(\gamma)\big|^2 + \widehat c_2(\gamma)$ 
where
\[\widehat c_2(\gamma)  
=  
\sum_{j=1}^d \frac{\sigma_j^2}{\lambda_j}\Big( \widehat b_j(\gamma)^{-1} -1\Big),
\;
\big|\widehat\sigma(\gamma)\big|^2 
=   
\sum_{j=1}^d \sigma_j^2 \widehat b_j(\gamma)
\;
\text{and} 
\;
\widehat b_j(\gamma):=1\wedge(\lambda_j\gamma^-)^{-\frac{1}{2}}\vee\varepsilon. \]
If $\lambda_j \gamma^- \leq 1$ then $\widehat b_j(\gamma) = 1$ and thus $\widehat b^\prime_j(\gamma) =0$. Similarly, if $\lambda_j \gamma^-\geq\varepsilon^{-(1+1)}$, then $\widehat b_j(\gamma) = \varepsilon$ and thus $\widehat b^\prime_j(\gamma) =0$. Finally, if $\varepsilon^{-(1+1)}>\lambda_j \gamma^- > 
1$ then  $\widehat b_j(\gamma) = (-\lambda_i \gamma)^{- \frac12 }$  and thus 
\begin{align*}
\widehat b^\prime_j(\gamma)  & =   \frac{1}{2} \lambda_j^{- \frac{1}{2}} (-\gamma)^{- \frac{1}{2}-1} = -  \frac{1}{2 \gamma }  \widehat b_j(\gamma).
\end{align*}
Define now 
\[
f_j(\gamma) := \frac{\sigma_j^2}{\lambda_j}\big( \widehat b_j(\gamma)^{-1} -1\big) + q \sigma_j^2 \widehat 
b_j(\gamma).\]
We have
\[
f^\prime_j(\gamma) =  \frac{\sigma_j^2}{\lambda_j} \Big( - \widehat b_j(\gamma)^{-2 }  \widehat b'_j(\gamma) \Big) + q 
\sigma_j^2 \widehat b'_j(\gamma).
\]
If $-\varepsilon^{-(1+1)}<\lambda_j \gamma < -1$, one has $\widehat b_j(\gamma)^{-2} = - \lambda_j \gamma$ and thus
\begin{align*}
f^\prime_j(\gamma)  =  \sigma_j^2 \big( \gamma + q ) \widehat b'_j (\gamma)  = - \Big( 1 + \frac{q}{\gamma} \Big) \frac{\sigma_j^2}{2} 
\widehat b_j(\gamma) \1_{\{-\varepsilon^{-2}< \lambda_j \gamma < -1\}}.
\end{align*}
Hence
\[
 \frac{\partial f_0}{\partial\gamma}(q,\gamma)
 =
 -\Big(1+\frac{q}{\gamma}\Big)
 \sum_{i=1}^d \frac{\sigma_i^2}{2}
                       \widehat b_i(\gamma)
                       \1_{\{-\varepsilon^{-2}<\lambda_i\gamma<-1\}}.
\]
Therefore, the minimum of $f_0(q,\gamma)$ over negative $\gamma$ is reached either at $\gamma = -q$ when there is at least one index $i\in\{1,\dots,d\}$ such that $q\in[1/\lambda_i,\varepsilon^{-2}/\lambda_i]$, and otherwise $f_0$ does not depend on $q$. Hence, we can consider that the minimiser is always $\gamma = -q$, and thus $F_0(q) = f_0(q,- q)$.

\vspace{0.5em}
\noindent Direct computations now show that
\[
f_j( - q ) = \frac{\sigma_j^2}{\lambda_j}
                          \bigg( \1_{\{\lambda_j q\le 1\}} \lambda_j q
                                  + \1_{\{\varepsilon^{-2}>\lambda_j q > 1\}}
                                     \Big( 2 (\lambda_j q)^{\frac{1}{2}} - 1 \Big)
                                             + \1_{\{ \varepsilon^{-2} \leq \lambda_j q \}}  \bigg( \lambda_j \varepsilon q + \varepsilon^{-1} -1 \bigg)\bigg),
\]

so that by adding all the terms
\[
  F_0(q)
 =
 \sum_{j=1}^d \frac{\sigma_j^2}{\lambda_j}
                          \bigg( \1_{\{\lambda_j q\le 1\}} \lambda_j q
                                  + \1_{\{\varepsilon^{-2}>\lambda_j q > 1\}}
                                     \Big(2 (\lambda_j q)^{\frac{1}{2}} - 1\Big)
                                             + \1_{\{\varepsilon^{- 2} \leq \lambda_j q \} } \bigg( \lambda_j \varepsilon q + \ \varepsilon^{-1} - 1  \bigg)\bigg),
\]
from which it is clear that $F_0$ is non--decreasing. 
\qed  \hfill$\Box$

\vspace{0.5em}
The value function of the second--best problem can be characterised as follows.

\begin{Proposition}[Second--best contract] \label{prop:sbnonlinear}
Assume that $f-g$ is Lipschitz continuous. Then

\vspace{0.5em}
\noindent {$(i)$} $ \Vsb  =  -\mathrm{e}^{-p(v(0,X_0) - L_0)}$ where $v$ has growth $|v(t,x)|\le C(T-t)|x|$, for some constant $C>0$, and is a viscosity 
solution of the {\rm PDE}
 \begin{equation}
 \label{eq:vbar}
\begin{cases}
\displaystyle - \partial_t v  
 =  f-g  + \frac12 \bar \mu\,  v_x^2 
     - \frac{1}{2} \inf_{z\in\R} \big\{ F_0\big( q(v_x, v_{xx}, z) \big)
                                                       +\bar\mu\big((z^-+v_x)^2+\eta_A(v_x,z)\big)\big\}, \; \mbox{on}\; [0,T)\times \R,\\
v(T,.)=0,     
\end{cases}             
 \end{equation}
with $q(v_x, v_{xx},z) :=  h-v_{xx} +r z^2+p(z-v_x)^2$, and $\eta_A(v_x,z):=(v_x+(z^--A)^+)^2-v_x^2\longrightarrow 0$, as $A\nearrow\infty$.

\vspace{0.5em}
\noindent {$(ii)$} If in addition $v$ is smooth, the optimal payment rate $\gsb$ to incentivise the agent responsiveness is
 \begin{equation}\label{eq:zgstarra}
 \gsb(t,X_t)
 :=
 - h + v_{xx}(t,X_t) - r \zsb^2(t,X_t) 
  - p \big(\zsb(t,X_t) - v_x(t,X_t)\big)^2, \; t\in[0,T],
 \end{equation}
and the optimal payment rate for the consumption deviation reduction is the minimiser $\zsb$ in \eqref{eq:vbar}, satisfying for large $A$:
\[
 \zsb \in \bigg(v_x, \frac{p}{r+p} v_x\bigg),
 ~\mbox{when}~
 v_x \le 0,
 \mbox{ and } 
 \zsb = \frac{p}{r+p} v_x,~\mbox{when}~v_x \ge 0.
\]

\noindent {$(iii)$} The second--best optimal contract is given by
 \[
 \xisb
 :=
 \frac{-\log{(-R)}}{r}
 + \int_0^T \zsb(t,X_t)\mathrm{d}X_t
                 +\frac12( \gsb
                               +r \zsb^2)(t,X_t)\mathrm{d}\langle X\rangle_t
                  -\big(H(\zsb,\gsb)
                           +f
                   \big)(t,X_t)\mathrm{d}t.
\]
\end{Proposition}

\begin{Remark}
{\rm 
$(i)$ Consider the case of a risk--neutral consumer $r=0$. As $F_0$ is non--decreasing by Lemma \ref{lem:F0}, we see that the 
minimum in the PDE \eqref{eq:vbar} is attained at $\zsb=v_x$, thus reducing the {\rm PDE} \eqref{eq:vbar} to
\[
 -\partial_t v - (f-g)  - \frac12 \bar \mu\,  \big(v_x^-\big)^2 
 = 
 -\frac12 F_0(h-v_{xx})
 =
 H_{\rm v}(v_{xx}-h),        
 \]
where the last equality follows from Lemma \ref{lem:F0}. Notice that this is the same PDE as the first--best characterisation 
given in Proposition \ref{prop:fb} $(i)$, since $\rho=0$ in the present setting.

\vspace{0.5em}
In particular, the producer's value function is independent of her risk--aversion parameter $p$. The optimal payment rates for the 
effort on the drift and the volatility are given by $\zsb = v_x$ and $\gsb = - h + v_{xx} $, so that the resulting optimal contract is 
also independent of the producer's risk--aversion $p$. This is consistent with the findings of H\"olmstrom and Milgrom 
{\rm\cite{Holmstrom87}} in the context of a risk--neutral agent, where  the optimal effort of the Agent is independent of the Principal 
risk aversion parameter.
}
\end{Remark}

\begin{proof}
By standard stochastic control theory, the dynamic version of the value function of the Principal, denoted by $V(t,x,
\ell):=\Vsb(t,x,\ell)$, is a viscosity solution, with appropriate growth at infinity, of the corresponding 
HJB equation
 \begin{equation*}
- \partial_t V 
=
(g -f) V_\ell
+\frac{1}2\sup_{(z,\gamma) \in \R^2} \!\! \Big\{  \big|\widehat\sigma(\gamma)\big|^2
                                                         \big[ (h+r z^2) V_\ell + V_{xx} +  z^2 V_{\ell\ell} + 2z V_{x\ell}\big]
                                                        - \bar\mu\big(2z^- V_x-(z^-)^2 V_\ell\big)
                                                        + \widehat c_2(\gamma) V_\ell
                                              \Big\}  ,	
 \end{equation*}
with terminal condition $V(T,x,\ell) 
 =
 U( -\ell)$, for $(x,\ell) \in\R^2$. Under the constant relative risk aversion specification of the utility function of the producer, it 
follows that
  \begin{align*}
 -p\partial_t v 
 &= 
 p (f-g)  
 - \frac12 \inf_{z,\gamma} \Big\{ \big|\widehat\sigma(\gamma)\big|^2 
                                                              \big((h+r z^2) p - p v_{xx} + p^2 (v_x)^2 + z^2 p^2 -  2 z p^2 v_x \big) 
  \\
&\hspace{9em}  + \bar \mu\big(2 z^- p v_x + p (z^-)^2 \big) + p \widehat c_2(\gamma) \Big\},	
 \end{align*} 
 which reduces to the PDE \eqref{eq:vbar}. 

The control on the growth of the function $v$ is deduced from the control on the growth of  $V$ by following the same line of 
argument as in the proof of Proposition \ref{prop:fb}, using the Lipschitz feature of the difference $f-g$. Similarly, under 
smoothness condition, the same verification argument leads to the optimal feedback controls, defined as the maximisers of the 
second--best producer's Hamiltonian, which determine the optimal payment rates.

We finally verify that the additional properties of the optimal payment rates hold. First, if $v_x \ge 0$, the map 
$z \longmapsto F_0\big(h -v_{xx} +r z^2+p(z- v_x)^2\big) +\bar\mu(z^-+v_x)^2\big\}$
is non--increasing for $z \leq \frac{p}{r+p} v_x$ because $F_0$ is a non--decreasing function. Thus, the minimum of the map is 
reached for the minimum of $q(v_x, v_{xx},z)$ which is $\zsb = \frac{p}{r+p} v_x$. Second, if $v_x \le 0$, the preceding map is non--longer monotonic on the interval $(v_x, \frac{p}{r+p} v_x)$. But, it is non--increasing for $z 
\le v_x$ and non--decreasing for $z \ge \frac{p}{r+p} v_x$, making its infimum lie between $v_x$ and $\frac{p}{r+p} v_x$. In both 
cases, the optimiser with respect to $\gamma$ can be deduced from Lemma~\ref{lem:F0}, and is given by \eqref{eq:zgstarra}.
\qed
\end{proof}

\subsection{Proof of Proposition \ref{prop:nogamma} {\rm Second best uncontrolled responsiveness}}
\label{app:nogamma}

\begin{Proposition}[Second--best uncontrolled responsiveness] 
Assume that $f-g = \delta x$. Then, $\Vsbm  =  U\big( w(0,X_0) - L_0\big))$ where $w$ has growth $|w(t,x)|\le C(T-t)|x|$, for some 
constant $C>0$, and is a viscosity solution of the {\rm PDE}
 \begin{equation}
 \label{eq:w}
-\partial_t w  
 =  (f-g)  + \frac12 \bar \mu\,  w_x^2 
     - \frac{1}{2}\, \inf_{z\in\R} \{ q(w_x, w_{xx}, z) | \sigma|^2 
                                                       +\bar\mu(z^-+w_x)^2\big\}, 
\mbox{ on } [0,T)\times \R,
\mbox{ and }
w(T,.)=0,           
 \end{equation}
with $q(w_x, w_{xx},z) =  h-w_{xx} +r z^2+p(z-w_x)^2$, and the second--best optimal contract is given by
$$
 \xisb^0
 =
 \frac{-\log{(-R)}}{r}
 + \Lambda \delta X_0 
 + \frac12  \int_0^T r \zsb^2 (t) |\sigma|^2 dt
 - \int_0^T H_{\rm m}(\zsb(t)) dt
 - \int_0^T (\kappa - \Lambda \delta) X_t dt
$$
where
 $
 \zsb 
 = \Lambda \delta (T-t) $ with 
 $\Lambda$ 
$ :=$  
$ \frac{p |\sigma|^2+\bar \mu \1_{\{w_x<0\}}} {(p+r) |\sigma|^2+ \bar \mu \1_{\{w_x<0\}}}.$
\end{Proposition}

\begin{proof} When $\Gamma \equiv 0$, we have $F_0(q) = q |\sigma|^2$, and the PDE of Proposition~\ref{prop:sbnonlinear} (i) reduces  thus to~\eqref{eq:w}. The minimizer is obtained directly by writing first order conditions, and the optimal contract follows. The form of the contract is obtained by applying integration by part to the term $\int_0^T \zsb(t) dX_t$ of the general form of the contract~\eqref{eq:Y}.
\qed
\end{proof}

\subsection{Proof of Proposition \ref{prop:inforent} {\rm Information rent}}
\label{app:inforent}

From Proposition~ \ref{prop:fb}, we have $\Vfb  = U( \bar v(0,X_0) + \frac{1}{r} \log(-R)) $. When $(f-g)(x) = \delta x$, 
we have 
\[
\bar v(0,X_0) = \delta T X_0 - \frac12 \int_0^T F_0(-\gfb(t)) \drm t,
\]
because $\zfb(t) \ge 0$ and thus,  $\widehat c_1(\zfb(t)) = 0$. Further, recall from Corollary~\ref{prop:RespDeltaPos} that when $\delta \ge 
0$, we have $\Vsb = U( v(0,X_0) + \frac{1}{r} \log(-R) )$
 with
\b*
v(0,X_0) = \delta T X_0 - \frac12 \int_0^T F_0(- \gsb(t)) \drm t.
\e*
In this case, we also have $\gsb = \gfb$. Thus, the certainty equivalent of the value function in the first--best and in the second--best 
are equal, and $ I = 0.$  Further, the equality of the certainty equivalent of the first--best and the second--best implies that the 
payments $\xifb$ and $\xisb$ are equal because the actions of the consumer are the same in both cases.

When $\delta < 0$ and $h + r \delta^2 T^2 \le \frac{1}{\bar \lambda}$, we know that $\zsb(t) = 
\Lambda \delta (T-t)$, so that
\[
 \inf_{z\in\R} \Big\{ F_0\big(h +r z^2+p(z- A(t))^2\big)
                                                       +\bar\mu(z^-+A(t))^2\Big\} = \big( h+ r \Lambda \delta^2 (T-t)^2 \big) |\sigma|^2,
\]
and therefore
\[
\psi(t) = \frac12 \int_t^T \bar \mu \delta^2 (T-t)^2 \drm t - \frac12 \int_0^T \Big( h+ r \Lambda \delta^2 (T-t)^2 \Big) |\sigma|^2 \drm t.
\]
Hence, in this case
\begin{align*}
\frac{\log \big(- \Vsb  \big)}{p}   & =
L_0 - \delta T X_0 - \psi(0) \\
 & =  \pi 
    + \frac12 \int_0^T \big(\gamma_s\big| \widehat{\sigma}\big(\gamma_s\big)\big|^2   -\widehat c_2\big(\gamma_s\big) - 
       \bar \mu \delta^2 (T-s)^2\big) \drm s  
   + \frac12 \int_0^T \Big( h+ r \Lambda \delta^2 (T-t)^2 \Big) |\sigma|^2 \drm t.
\end{align*}
In addition, in this setting, $c(\nufb) = c_1(\afb)$ and $ - 
\gfb= h + \rho \delta^2 (T-t)^2$ $\le$ $\frac{1}{\bar \lambda}$ because $\rho < r$, and thus $
\widehat b_j(\gfb) = 1$, and we have $
c_1(\afb(t)) = \frac12 \bar \mu \delta^2 (T-t)^2.
$
Thus we get
\[
I = \frac12 \int_0^T \gfb(t) 
                              \big|\widehat \sigma\big( \bfb(t)\big)\big|^2 \drm t 
     + \frac12 \int_0^T \big( h+ r \Lambda \delta^2 (T-t)^2 \big) |\sigma|^2 \drm t
=
 \frac12|\sigma|^2\big( r \Lambda   
                                      - \rho 
                              \big) \int_0^T \delta^2(T-t)^2 \drm t,
\]
where we used the fact that $ - \gfb = h + \rho \delta^2 (T-t)^2$. The required expression follows.
\qed

\bibliographystyle{plain}

\end{document}